\documentclass[11pt]{article}
\usepackage[margin=1in]{geometry}
\usepackage{amsmath,amssymb,amsthm}
\usepackage{etoolbox}
\usepackage{mathtools}
\usepackage{graphicx}
\graphicspath{{figures/}}
\usepackage{booktabs}
\usepackage{microtype}
\usepackage{pdflscape}
\usepackage{float}
\usepackage{needspace}
\usepackage{tikz}
\usetikzlibrary{arrows.meta,positioning,calc}
\usepackage{xcolor}
\usepackage{hyperref}

\newtheorem{theorem}{Theorem}[section]
\newtheorem{proposition}[theorem]{Proposition}
\newtheorem{lemma}[theorem]{Lemma}
\newtheorem{corollary}[theorem]{Corollary}
\theoremstyle{definition}
\newtheorem{example}[theorem]{Example}
\newtheorem{remark}[theorem]{Remark}
\newcommand{\bestiaryendmark}{\,\tikz[baseline=-0.4ex]{\draw[line width=0.6pt] (0,0) circle (1.6pt);}}
\AtEndEnvironment{example}{\hfill\bestiaryendmark}
\AtEndEnvironment{remark}{\hfill\bestiaryendmark}

\newcommand{\Cov}{\operatorname{Cov}}
\newcommand{\Var}{\operatorname{Var}}
\newcommand{\Corr}{\operatorname{Corr}}
\newcommand{\sgn}{\operatorname{sgn}}
\newcommand{\dd}{\mathrm{d}}
\newcommand{\NN}{\mathrm{NN}}

\newenvironment{beast}[2]{%
  \par\medskip\needspace{5\baselineskip}%
  \noindent\textbf{#1}\ \ {\footnotesize\itshape #2}\par
  \begin{list}{}{%
    \setlength{\leftmargin}{6.2em}\setlength{\labelwidth}{5.6em}%
    \setlength{\labelsep}{0.6em}\setlength{\itemsep}{0pt}%
    \setlength{\parsep}{0pt}\setlength{\topsep}{2pt}%
    }%
}{\end{list}\medskip}

\title{A Bestiary of Brownian Processes\\[3pt]{\large A Field Guide to Boundary Conditions and Conditioning}}
\author{Lars Lau Raket}
\date{}
\hypersetup{
  pdftitle={A Bestiary of Brownian Processes: A Field Guide to Boundary Conditions and Conditioning},
  pdfauthor={Lars Lau Raket}
}

\begin{document}
\maketitle

\begin{abstract}
\noindent The covariance of Brownian motion is the Green kernel of the one-dimensional Laplacian with a pinned left endpoint and a free right endpoint. This correspondence between Green's functions for the Laplacian on an interval and Brownian covariance functions turns out to be a general phenomenon. We use this correspondence to organize Brownian-derived Gaussian processes by three operations: choosing a self-adjoint realization, recentering a path, and conditioning a fixed law. Local Robin and selected nonlocal boundary conditions yield a catalogue of twelve processes with closed-form kernels, Brownian-path constructions, spectra and traces. The periodic and Neumann--Neumann realizations annihilate the constants; projecting the source and imposing zero time average selects the corresponding covariance pseudoinverses. Recentering and conditioning can both give a Brownian path zero time average, but the resulting laws differ: a constraint alone does not specify a process. A rank-one conditioning identity gives a scale-invariant criterion for the covariance effect of a shrinking family of observations. For Brownian finite differences of any fixed positive order, the covariance effect vanishes uniformly as the spacing tends to zero. For $m$-fold integrated Brownian motion, derivative constraints of orders $1,\ldots,m$ remain detectable at interior times. For simulation, we identify which grid constructions are exact and which are approximate.
\end{abstract}

\section{Introduction}

\subsection{The Laplacian--Brownian correspondence}\label{sec:intro_correspondence}

Let $\{W_s\}_{s\ge 0}$ be standard Brownian motion: the continuous, zero-mean Gaussian process with independent increments, $W_0=0$, and $\Var(W_t - W_s) = |t-s|$. Its covariance function is
\[
 \Cov(W_s, W_t) = s \wedge t,
\]
where $s\wedge t = \min(s,t)$. For fixed $t\in(0,1)$, the function $K(s,t) = s\wedge t$ rises with unit slope until $s=t$ and is constant afterwards. It is continuous, with a unit downward jump in slope at~$t$. Its first derivative in~$s$ is $\mathbf{1}_{[0,t]}(s)$ almost everywhere, and its distributional second derivative is minus a unit point mass at~$t$. Thus
\begin{equation}\label{eq:BM_green_intro}
 -\partial_s^2 K(s,t) = \delta(s-t), \qquad K(s,t)\big|_{s=0} = 0, \qquad \partial_s K(s,t)\big|_{s=1} = 0 .
\end{equation}
Here $\delta$ is the \emph{Dirac delta}, the unit point mass characterized by $\int_0^1 f(s)\,\delta(s-t)\,\dd s = f(t)$ for every continuous~$f$. The distributional identity $-\partial_s^2 K = \delta(\cdot-t)$ expresses the geometry above: $K(\cdot,t)$ is affine except for a unit downward kink at~$t$. See Stakgold and Holst~\cite{Stakgold2011} for a rigorous treatment.

The three relations in~\eqref{eq:BM_green_intro} identify the Brownian covariance as the \emph{Green's function} of the Laplace operator $\mathcal{L} = -\partial_s^2 = -\dd^2/\dd s^2$ on $[0,1]$---the kernel of its inverse---with a Dirichlet (pinned) condition at $s=0$ and a Neumann (free) condition at $s=1$. For each interior source point~$t$, both conditions can be read from the kernel: $K(0,t)=0$, and $\partial_s K(s,t)|_{s=1}=0$ because $K(\cdot,t)$ is flat to the right of~$t$. Brownian motion on the half-line $[0,\infty)$ needs no right-end boundary condition. The Neumann condition appears when its covariance on the bounded interval is represented as a Green kernel; it is not an extra restriction on the random path. Indeed, Brownian paths are almost surely nowhere differentiable. Remark~\ref{rem:neumann_vs_derivative} develops this distinction.

\paragraph{Why Green's functions of the Laplacian are covariance functions.}
The relationship extends beyond Brownian motion and these particular endpoint conditions. Call a set of boundary conditions~$\mathcal{B}$ \emph{admissible} when it makes $\mathcal{L}=-\dd^2/\dd s^2$ self-adjoint and strictly positive on its domain. The inverse is then an integral operator with a symmetric, positive-semidefinite kernel, hence a Gaussian covariance. Self-adjointness alone is not enough: an invertible operator may have negative eigenvalues, in which case its symmetric Green's function is not a covariance.

The operator is the process's precision. Integration by parts expresses its energy as derivative roughness plus boundary terms: a formal path-density interpretation, not finite energy carried by typical Brownian draws. For the nondegenerate Robin cases with nonnegative parameters, discretizing this energy gives a discrete Laplacian with endpoint modifications whose inverse agrees exactly with the sampled Green's function (Remark~\ref{rem:fem_exact}).

The base cases are Brownian motion (pinned left, free right), the Brownian bridge (pinned at both ends), and time-reversed Brownian motion (free left, pinned right). Brownian motion itself is integrated white noise. Starting from these familiar specimens, we build a catalogue using admissible boundary conditions and linear constraints.

\subsection{Three operations}\label{sec:two_mechanisms}

We distinguish three operations, each with its own starting point and covariance calculation. Different operations can produce the same law, so the resulting families overlap. We also distinguish the differential \emph{expression} $-\dd^2/\dd s^2$, a \emph{realization} $(\mathcal{L},\mathcal{B})$ obtained by fixing its domain, and the random \emph{path} of a process whose covariance is supplied by the inverse. Figure~\ref{fig:two_mechanisms} summarizes the three operations.

\paragraph{1. Choose a realization.}
The expression $-\dd^2/\dd s^2$ becomes an operator only after its domain is specified. Its homogeneous equation has a two-dimensional solution space, and two independent endpoint conditions specify a realization, though not necessarily a self-adjoint or invertible one. Admissible choices give different Green's functions and hence different processes. An ordinary inverse requires a trivial null space; in the degenerate cases, a source correction and a normalization are needed instead. Section~\ref{sec:counting} distinguishes these requirements from the endpoint conditions themselves.

\paragraph{2. Recenter a path.}
Subtract a functional of a given path: its value at a time, $W_s-W_\tau$, or its average, $W_s-\int_0^1W_u\,\dd u$. This is a linear transformation of an existing process, not conditioning or the specification of a new operator. Its covariance follows by bilinearity. The balanced normalization for the degenerate operators in Section~\ref{sec:degenerate} has this pathwise interpretation; Section~\ref{sec:bc_vs_constraints} compares it with conditioning.

\paragraph{3. Condition a fixed law.}
Given a zero-mean Gaussian process~$X$ with covariance~$K$, condition on a linear observation $\ell(X)=0$ by subtracting from each path value its Gaussian regression on~$\ell(X)$. This changes~$K$ by a rank-one subtraction and requires a specified starting covariance and a positive, finite $\Var(\ell(X))$. The precision operator need not remain the same. Conditioning restricts the energy form to a smaller domain, so the associated operator, or the space on which it acts, may change. For example, conditioning Brownian motion on $W_1=0$ gives the Brownian bridge: its precision has two Dirichlet conditions rather than one Dirichlet and one Neumann condition. Whether the conditioned covariance is another realization's Green kernel must therefore be established separately. Interior pinning, tying two path values together, and imposing zero time integral are examples of this operation.

Recentering and conditioning can use the same process and functional yet give different laws. For point evaluation and time averaging, adding an independent constant offset of increasing variance makes conditioning approach recentering (Proposition~\ref{prop:diffuse}). At finite variance they generally differ, as interior pinning illustrates (Section~\ref{sec:interior_pin}), but they already agree when the regression coefficient is identically one.

\begin{figure}[H]
 \centering
 \begin{tikzpicture}[
   font=\small,
   bx/.style={draw, rounded corners=2pt, align=center, inner sep=4pt, minimum height=8.5mm},
   hd/.style={font=\small\itshape, anchor=west},
   up/.style={font=\scriptsize, midway, above=0.6mm},
   dn/.style={font=\scriptsize, midway, below=0.6mm},
   >={Stealth[length=2mm,width=1.6mm]}]
 \node[hd] at (0,12mm)
   {\textbf{1.\ Choose a realization} \ (Sections~\ref{sec:robin}--\ref{sec:degenerate})};
 \node[bx, anchor=west] (bc) at (0,0) {boundary\\conditions $\mathcal{B}$};
 \node[bx, anchor=west] (op) at (48mm,0) {realization\\$(\mathcal{L},\mathcal{B})$};
 \node[bx, anchor=west] (pr) at (98mm,0) {process $X$,\\covariance $G$};
 \draw[->] (bc.east) -- node[up] {fix a domain} (op.west);
 \draw[->] (op.east) -- node[up] {invert} (pr.west);
 \node[hd] at (0,-23mm)
   {\textbf{2.\ Recenter a path} \ and \ \textbf{3.\ Condition a fixed law} \ (Sections~\ref{sec:conditioning}, \ref{sec:bc_vs_constraints})};
 \node[bx, anchor=north west] (p0) at (0,-29mm) {process $X$,\\covariance $K$ \emph{given}};
 \node[bx, anchor=west] (rc) at (88mm,-30mm) {$X_s-\ell(X)$};
 \node[bx, anchor=west] (cd) at (88mm,-44mm) {$X\mid\ell(X)=0$};
 \coordinate (br) at (52mm,-37mm);
 \draw (p0.east) -- (br);
 \draw[->] (br) |- node[up, pos=0.72] {subtract it outright} (rc.west);
 \draw[->] (br) |- node[dn, pos=0.72] {subtract fitted path values} (cd.west);
 \node[font=\scriptsize\itshape, gray, anchor=north west, align=left] at (0,-53mm)
   {For point values and averages, the outputs generally differ (Section~\ref{sec:interior_pin}), but\\agree in the diffuse-offset limit---and sometimes at finite variance (Proposition~\ref{prop:diffuse}).};
 \end{tikzpicture}
 \caption{Three operations, two starting points. The upper row begins with boundary conditions defining a realization; when admissible, its inverse supplies a covariance kernel. The lower branches begin with a specified covariance. Recentering subtracts the observed functional from every path value; conditioning subtracts the regression of each path value on that functional and requires a positive, finite $\Var(\ell(X))$. Section~\ref{sec:degenerate} constructs covariance pseudoinverses for the degenerate realizations. Section~\ref{sec:bc_vs_constraints} compares recentering and conditioning, and explains why an additional path constraint is not an additional operator boundary condition.}
 \label{fig:two_mechanisms}
\end{figure}

For a degenerate realization, ordinary inversion cannot supply the starting covariance~$K$ or the observation variance $\Var(\ell(X))$. Section~\ref{sec:degenerate} instead uses a corrected source and the balanced condition $\int_0^1 G(s,t)\,\dd s=0$ to construct a covariance pseudoinverse. This selects a process rather than conditions a previously specified law.

\subsection{Contributions and relation to prior work}

The link between Green's functions, covariance functions and smoothing splines is classical. The reproducing-kernel-Hilbert-space view of smoothing splines goes back to Kimeldorf and Wahba~\cite{Kimeldorf1970,Wahba1990}. The correspondence between linear stochastic differential equations and covariance functions underlies the state-space and SPDE views of Gaussian processes~\cite{Whittle1954,Lindgren2011,Sarkka2019,Rasmussen2006}, while Green's functions for Sturm--Liouville problems are standard~\cite{Stakgold2011,Zettl2005,Karatzas1991,Morters2010}. We specialize and extend the direct, RKHS-free treatment of~\cite{Raket2021}, which used the Laplacian and the Poisson equation to compute spline solutions and construct Brownian processes by Gaussian conditioning. Markussen~\cite{Markussen2013} used the operator--covariance correspondence for inference in functional mixed-effects models, replacing matrix computations with approximations in function and operator spaces. Raket and Markussen~\cite{RaketMarkussen2014} extended that approach to two-dimensional domains. Our aim is complementary: rather than develop inference, we catalogue the processes produced by the one-dimensional Laplacian and the boundary conditions that produce them. The underlying self-adjoint extension theory is developed in~\cite{Zettl2005}.

The contributions follow the three operations:
\begin{enumerate}
\item An \emph{explicit atlas} of boundary-defined processes. A single closed-form Green's function describes the two-parameter Robin family (Section~\ref{sec:robin}), including Brownian motion, the bridge and time-reversed Brownian motion at the corners of $[0,\infty]^2$. For the selected nonlocal conditions (Section~\ref{sec:nonlocal}), we give a self-adjointness criterion and an $\alpha$-family connecting those two Brownian motions with the anti-periodic process. The degenerate periodic and Neumann--Neumann realizations (Section~\ref{sec:degenerate}) share a constant null space. They require the corrected source $\delta(s-t)-1$ and a \emph{balanced} side condition, $\int_0^1 G(s,t)\,\dd s = 0$, to select a unique kernel, following the terminology of~\cite{Raket2021}. The atlas also marks a boundary of the construction: an endpoint--integral condition such as $f(0)=\int_0^1 f$ cannot define a symmetric realization when paired with another independent boundary condition of the type considered here. We do not attempt to exhaust the self-adjoint realizations.
\item A \emph{comparison of recentering and conditioning} (Sections~\ref{sec:conditioning} and~\ref{sec:bc_vs_constraints}). A single rank-one identity gives bridges and processes subject to interior pins, tied values, weighted observations or integral constraints. We also make the Markov decomposition at an interior pin explicit. Recentering and conditioning generally give different laws even for the same functional, as the two constructions of zero-time-average Brownian paths demonstrate. For the degenerate processes, recentering at an interior time and conditioning the value there to be zero differ by an explicit positive-semidefinite rank-one kernel. This gap vanishes as an independent constant offset of increasing variance is added.
\item A \emph{limiting criterion for conditioning} (Section~\ref{sec:detectability}).\footnote{We use ``detectable'' only in a covariance-theoretic sense, not a statistical one. A constraint $\ell(X)=0$ is \emph{detectable} for~$X$ when conditioning on it changes the covariance, equivalently when $X_s$ has nonzero covariance with $\ell(X)$ for some~$s$.} For the linear observations considered here, conditioning changes the covariance precisely when the observation has nonzero covariance with some path value. For a shrinking family of observations of positive variance, the covariance correction tends to zero pointwise exactly when $\Corr\bigl(X_s,\ell_\varepsilon(X)\bigr)\to 0$ at every~$s$ of positive marginal variance. For Brownian finite differences of any fixed positive order, the correction vanishes uniformly as the spacing tends to zero. For $m$-fold integrated Brownian motion, value and derivative constraints of orders up to~$m$ are detectable at interior times.
\end{enumerate}

\section{Green's functions and Gaussian covariance}

\subsection{The Laplace operator and Green's functions on \texorpdfstring{$[0,1]$}{[0,1]}}\label{sec:green}

Throughout, $\mathcal{L} = -\dd^2/\dd s^2$ acts on real functions on $[0,1]$. We use $\mathcal{L}_s$ to indicate action in the variable~$s$, and $\partial_s$, $\partial_s^2$ for partial derivatives in~$s$. Homogeneous boundary conditions~$\mathcal{B}$ specify a domain $\operatorname{dom}(\mathcal{L})\subset L^2[0,1]$. When these conditions are understood, $\mathcal{L}$ denotes the realization $(\mathcal{L},\mathcal{B})$, including its domain. If the realization is invertible, its \emph{Green's function} $G(s,t)$ is the kernel of~$\mathcal{L}^{-1}$. For fixed $t\in(0,1)$, $G(\cdot,t)$ is continuous, piecewise in $H^2$, satisfies~$\mathcal{B}$ and has its only singularity at $s=t$. It solves, distributionally,
\begin{equation}\label{eq:greens}
 \mathcal{L}_s\, G(s,t) = \delta(s-t),
\end{equation}
where $\delta$ is the Dirac delta of Section~\ref{sec:intro_correspondence}. Away from $s=t$, we have $\mathcal{L}_s G(s,t)=0$, so $G(\cdot,t)$ is affine on each side of~$t$. Equation~\eqref{eq:greens} also determines how the two pieces meet.
Write $G(t^+,t)$ and $G(t^-,t)$ for the limits as $s\to t$ from above and below, with $\partial_s G(t^+,t)$ and $\partial_s G(t^-,t)$ denoting the corresponding derivative limits. The matching conditions are continuity and a unit downward jump in slope:
\begin{equation}\label{eq:matching}
 G(t^-,t)=G(t^+,t) \quad\text{(continuity)}, \qquad \partial_s G(t^+,t) - \partial_s G(t^-,t) = -1 \quad\text{(unit kink)} .
\end{equation}
The sign follows from integrating~\eqref{eq:greens} across~$t$: for small $\varepsilon>0$,
\[
\partial_s G(t^-,t) - \partial_s G(t^+,t) = \int_{t-\varepsilon}^{t+\varepsilon}\delta(s-t)\,\dd s = 1.
\] The boundary conditions~$\mathcal{B}$ supply the remaining two relations.

We call $(\mathcal{L},\mathcal{B})$ \emph{admissible} when $\mathcal{L}$ is self-adjoint on the domain specified by~$\mathcal{B}$ and strictly positive: $\langle f,\mathcal{L}f\rangle>0$ for every nonzero $f\in\operatorname{dom}(\mathcal{L})$. In the degenerate cases of Section~\ref{sec:degenerate}, only $\langle f,\mathcal{L}f\rangle\ge 0$ holds, and a pseudoinverse replaces the inverse. For an admissible realization, self-adjointness gives $G(s,t)=G(t,s)$ and the positive spectrum makes~$G$ positive semidefinite, hence a covariance function.

\begin{remark}[Invertibility is not positivity]\label{rem:admissible}
Self-adjointness and invertibility do not imply positivity. If $\mathcal{L}$ has a negative eigenvalue, its Green's function remains symmetric but is not positive semidefinite and therefore cannot be a covariance. Section~\ref{sec:negative_branch} identifies where this occurs in the Robin parameter space.
\end{remark}

\subsection{Gaussian processes}

Throughout, $\langle f,g\rangle = \int_0^1 f(s)\,g(s)\,\dd s$ denotes the inner product of $L^2[0,1]$, and $\|f\|^2 = \langle f,f\rangle$. Integral operators act by $(\mathcal{K}f)(s) = \int_0^1 K(s,u)f(u)\,\dd u$.

The law of a zero-mean Gaussian process $\{X_s\}_{s\in[0,1]}$ is determined by its covariance $K(s,t) = \Cov(X_s,X_t)$. This symmetric kernel is \emph{positive semidefinite}: $\sum_{i,j} c_i c_j K(s_i,s_j)\ge 0$ for every finite point set $\{s_i\}$ and all real coefficients $c_i$. We write $\Var(X_s) = K(s,s)$ and, when both marginal variances are positive, $\Corr(X_s,X_t) = K(s,t)/\sqrt{K(s,s)K(t,t)}$.

Sections~\ref{sec:conditioning}--\ref{sec:detectability} use two kinds of linear functional: integrals $\ell(X) = \int_0^1 X_u\,g(u)\,\dd u$ with $g\in L^2[0,1]$, and finite combinations of point evaluations $\ell(X) = \sum_i c_i X_{s_i}$. Both lie in the closed linear span of the variables $X_s$ in $L^2(\mathbb{P})$ and are therefore Gaussian. Applying~$\ell$ to the covariance arguments gives their cross-covariance with the process and their variance:
\begin{equation}\label{eq:functional_cov}
 K_\ell(s) = \Cov\bigl(X_s,\ell(X)\bigr), \qquad V_\ell = \Var\bigl(\ell(X)\bigr) = \ell\bigl(K_\ell\bigr) .
\end{equation}
For an integral functional, $K_\ell(s) = \int_0^1 K(s,u)g(u)\,\dd u = (\mathcal{K}g)(s)$ and $V_\ell = \langle g,\mathcal{K}g\rangle$, which is finite whenever $K$ is bounded. For point evaluations, $K_\ell(s) = \sum_i c_i K(s,s_i)$ and $V_\ell = \sum_{i,j}c_ic_jK(s_i,s_j)$. Although point evaluation is not a bounded functional on $L^2[0,1]$, the random variable it produces is square-integrable and belongs to the closed linear span of $\{X_s\}$ in $L^2(\mathbb{P})$. Section~\ref{sec:detectability} studies families of observations with no limit in that span, including Brownian difference quotients.

\subsection{From Green's function to covariance}\label{sec:from_green_to_cov}

We now construct the Gaussian process whose covariance is the Green's function of an admissible $(\mathcal{L},\mathcal{B})$, proving the claim in Section~\ref{sec:green}. We first use the eigenfunctions of the realization, then give a stochastic-integral construction. Both use a factor of the covariance operator rather than the covariance operator itself.

\paragraph{Spectral construction.}
Let $(\mu_n,\varphi_n)_{n\ge 1}$ be the eigenvalues and eigenfunctions of $(\mathcal{L},\mathcal{B})$, with $\mathcal{L}\varphi_n = \mu_n\varphi_n$ and $\varphi_n\in\operatorname{dom}(\mathcal{L})$, and choose the $\varphi_n$ to form an orthonormal basis of $L^2[0,1]$. Admissibility (Section~\ref{sec:green}) makes the eigenvalues real and positive. For the second-order realizations considered here, $\mu_n\to\infty$ as $n\to\infty$. The inverse has eigenvalues $\mu_n^{-1}$, since $\mathcal{L}^{-1}\varphi_n = \mu_n^{-1}\varphi_n$, and its kernel is
\begin{equation}\label{eq:mercer}
 G(s,t) = \sum_{n=1}^{\infty} \frac{1}{\mu_n}\,\varphi_n(s)\,\varphi_n(t) ,
\end{equation}
which for the piecewise-polynomial kernels of this paper converges absolutely and uniformly. Now let $Z_1,Z_2,\dots$ be independent standard normal variables and set
\begin{equation}\label{eq:KL}
 X_s \;=\; \sum_{n=1}^{\infty} \frac{1}{\sqrt{\mu_n}}\,Z_n\,\varphi_n(s) .
\end{equation}
For each~$s$, the series converges in $L^2(\mathbb{P})$ because $\sum_{n=1}^{\infty} \mu_n^{-1}\varphi_n(s)^2 = G(s,s)<\infty$. The limit is zero-mean Gaussian, and termwise calculation using~\eqref{eq:mercer} gives
\[
 \Cov(X_s,X_t) = \sum_{n=1}^{\infty}\frac{1}{\mu_n}\,\varphi_n(s)\varphi_n(t) = G(s,t) .
\]
Thus~$G$ is the covariance of a Gaussian process with covariance operator~$\mathcal{L}^{-1}$ and precision~$\mathcal{L}$. Equation~\eqref{eq:KL} is its Karhunen--Lo\`eve expansion.

The exponent $-\tfrac{1}{2}$ in~\eqref{eq:KL} reflects the use of a covariance \emph{square root}: $\mathcal{L}^{-1/2}$ maps the white noise $\xi = \sum_{n=1}^{\infty} Z_n\varphi_n$ to~$X$. Applying $\mathcal{L}^{-1}$ instead gives $\sum_{n=1}^{\infty} \mu_n^{-1}Z_n\varphi_n$, whose covariance is the kernel of $\mathcal{L}^{-2}$, namely $\int_0^1 G(s,u)G(u,t)\,\dd u$. That construction produces a smoother process with a different covariance. 
\paragraph{Stochastic-integral form.}
Suppose $H$ is any kernel that factors~$G$, in the sense that
\begin{equation}\label{eq:factorization}
 \int_0^1 H(s,u)\,H(t,u)\,\dd u = G(s,t) \qquad\text{for all } s,t\in[0,1] .
\end{equation}
Then, with $W$ a standard Brownian motion, the process $X_s$ defined by the stochastic integral
\begin{equation}\label{eq:stochastic_integral}
 X_s \;=\; \int_0^1 H(s,u)\,\dd W_u
\end{equation}
has covariance~$G$ by the It\^o isometry: $\Cov(X_s,X_t) = \int_0^1 H(s,u)H(t,u)\,\dd u$. The spectral construction uses $H(s,u) = \sum_{n=1}^{\infty} \mu_n^{-1/2}\varphi_n(s)\varphi_n(u)$, the kernel of $\mathcal{L}^{-1/2}$, but other covariance factors can be simpler. For Brownian motion, take $H(s,u) = \mathbf{1}_{[0,s]}(u)$. Then $\int_0^1 \mathbf{1}_{[0,s]}(u)\mathbf{1}_{[0,t]}(u)\,\dd u = s\wedge t$, and~\eqref{eq:stochastic_integral} reduces to $X_s = \int_0^s \dd W_u = W_s$. For the Brownian bridge, $H(s,u) = \mathbf{1}_{[0,s]}(u) - s$ gives $\int_0^1 H(s,u)H(t,u)\,\dd u = s\wedge t - st$. Both constructions reflect the factorization $\mathcal{L} = \partial^*\partial$, with $\partial = \dd/\dd s$, underlying the roughness penalty in Section~\ref{sec:intro_correspondence}. The boundary conditions determine the domains of $\partial$ and~$\partial^*$.

\paragraph{Energy and endpoint penalties.}
Integration by parts gives 
\[
\langle f,\mathcal{L}f\rangle = \int_0^1 f'(s)^2\,\dd s + f'(0)f(0) - f'(1)f(1)
\] 
for sufficiently regular~$f$, in particular for $f\in H^2[0,1]$. If $f$ satisfies the Robin conditions $f(0)=\lambda_0 f'(0)$ and $f(1)=-\lambda_1 f'(1)$ introduced in~\eqref{eq:robin}, with positive finite parameters, then $f'(0) = f(0)/\lambda_0$ and $f'(1) = -f(1)/\lambda_1$, giving
\begin{equation}\label{eq:energy}
 \langle f,\mathcal{L}f\rangle = \int_0^1 f'(s)^2\,\dd s + \frac{f(0)^2}{\lambda_0} + \frac{f(1)^2}{\lambda_1} ,
\end{equation}
with the convention that $\lambda_i=\infty$ removes the boundary term (a free end costs nothing), while $\lambda_i=0$ requires $f$ to vanish at that end (an infinitely stiff tether). Formally, \eqref{eq:energy} combines a derivative-roughness penalty with penalties on the endpoint values; typical Brownian paths do not have finite energy. For positive finite~$\lambda_i$, the endpoint stiffness~$1/\lambda_i$ gives the elastic-tether interpretation and explains the interpolation between pinned and free ends.

\begin{remark}[The exact finite-dimensional analogue]\label{rem:fem_exact}
Take the grid $s_i = ih$, $i=0,\dots,n$, with $h=1/n$, and discretize~\eqref{eq:energy} as
\[
 Q(f) = \frac{1}{h}\sum_{i=1}^{n}(f_i - f_{i-1})^2 + \frac{f_0^2}{\lambda_0} + \frac{f_n^2}{\lambda_1} = f^\top M f ,
\]
omitting $f_0$ when $\lambda_0=0$ and $f_n$ when $\lambda_1=0$. On the retained coordinates, $M$ is the tridiagonal discrete Laplacian with endpoint penalties added to its corner entries:
\[
 M_{00} = \frac1h + \frac1{\lambda_0}, \qquad M_{nn} = \frac1h + \frac1{\lambda_1}, \qquad M_{ii} = \frac2h \ \ (0<i<n), \qquad M_{i,i+1}=M_{i+1,i}=-\frac1h .
\]
An endpoint coordinate appears in one difference term, contributing $\tfrac1h$ to the diagonal, compared with $2/h$ at an interior node. When $\lambda_i=\infty$, no endpoint penalty is added: this is the discrete Neumann condition. Unless $\lambda_0=\lambda_1=\infty$, when constants make~$M$ singular, a Gaussian vector with precision~$M$ has covariance~$M^{-1}$. This covariance is exact on the grid:
\[
 \bigl(M^{-1}\bigr)_{ij} = G(s_i,s_j) \qquad\text{exactly},
\]
for every~$n$. Indeed, $Q$ is the energy form~\eqref{eq:energy} restricted to continuous piecewise-linear functions on the grid. Each section $G(\cdot,s_j)$ is piecewise affine with its only possible kink at~$s_j$, so it already belongs to the trial space.
\end{remark}

For Brownian motion, the correspondence can also be checked directly.

\begin{proposition}[Brownian motion]\label{prop:BM}
The Brownian motion covariance $G(s,t) = s\wedge t$ is the Green's function of $\mathcal{L} = -\dd^2/\dd s^2$ on $[0,1]$, with $G(0,t)=0$ and $\partial_s G(1,t)=0$ for every interior source point $t\in(0,1)$.
\end{proposition}

\begin{proof}
Fix $t\in(0,1)$. For $s<t$, $G(s,t)=s$ and $\partial_s G(s, t) = 1$; for $s>t$, $G(s,t)=t$ and $\partial_s G(s,t) = 0$. Hence $\partial_s G(s,t) = \mathbf{1}_{[0,t]}(s)$ almost everywhere and $\partial_s^2 G(s,t) = -\delta(s-t)$ distributionally, proving~\eqref{eq:greens}. At the endpoints, $G(0,t)=0$ gives the Dirichlet condition, while $s=1>t$ gives $\partial_s G(s,t) = 0$, the Neumann condition. Continuity and the unit jump in slope give~\eqref{eq:matching}.
\end{proof}

Table~\ref{tab:roadmap} lists the processes by construction, with their defining conditions and the sections where they are derived. The derivations give the kernels; Table~\ref{tab:constructions} collects the field marks---spectrum, trace, triangularity and path construction---for comparison.

\begin{landscape}
\vspace*{\fill}
\begin{table}[H]
\centering
\caption{Roadmap of the constructions. The upper block lists boundary conditions defining~$\mathcal{L}$; the lower block lists constraints on a fixed Brownian motion~$W$. The periodic and Neumann--Neumann rows require both the corrected source~\eqref{eq:corrected} and the balanced condition $\int_0^1 G(s,t)\,\dd s=0$ in~\eqref{eq:zeromean}. The final row compares recentering with conditioning for the degenerate processes; these operations give different laws here, as they also do in the zero-time-average example.}\label{tab:roadmap}
\footnotesize
\setlength{\tabcolsep}{5pt}
\begin{tabular}{lll}
\toprule
Process & Condition & Section \\
\midrule
\multicolumn{3}{l}{\emph{Boundary conditions define the operator}} \\[4pt]
Brownian motion & $f(0)=0$, $f'(1)=0$ & \ref{sec:robin} \\[4pt]
Brownian bridge & $f(0)=f(1)=0$ & \ref{sec:robin} \\[4pt]
Time-reversed Brownian motion & $f'(0)=0$, $f(1)=0$ & \ref{sec:robin} \\[4pt]
Robin family & $f(0)=\lambda_0f'(0)$, $f(1)=-\lambda_1f'(1)$ & \ref{sec:robin} \\[4pt]
Random slope & $f(0)=0$, $f(1)=(1+\sigma^{-2})f'(1)$ & \ref{sec:negative_branch} \\[4pt]
$\alpha$-family & $f(0)=\alpha f(1)$, $f'(0)=f'(1)/\alpha$ & \ref{sec:nonlocal} \\[4pt]
Anti-periodic Brownian motion & $\alpha=-1$ & \ref{sec:nonlocal} \\[4pt]
Slepian Brownian motion & $f'(0)+f'(1)=0$, $f(0)+f(1)=f'(0)$ & \ref{sec:nonlocal} \\[4pt]
Periodic (circular) Brownian motion & $\alpha=1$, $\int_0^1 G(s,t)\,\dd s=0$ & \ref{sec:degenerate} \\[4pt]
Neumann--Neumann Brownian motion & $f'(0)=f'(1)=0$, $\int_0^1 G(s,t)\,\dd s=0$ & \ref{sec:degenerate} \\[8pt]
\multicolumn{3}{l}{\emph{Constraints condition a fixed process}} \\[4pt]
Brownian bridge & $W_1=0$ & \ref{sec:conditioning} \\[4pt]
Pinned at $\tau$ & $W_\tau=0$ & \ref{sec:conditioning} \\[4pt]
Tied times & $W_{s_1}-W_{s_0}=0$ & \ref{sec:conditioning} \\[4pt]
Weighted endpoints & $\alpha W_{s_0}+\beta W_{s_1}=0$ & \ref{sec:conditioning} \\[4pt]
Integral-conditioned Brownian motion & $\int_0^1 W_u\,\dd u=0$ & \ref{sec:conditioning} \\[4pt]
\emph{vanishing covariance effect} & Zero finite differences of fixed positive order as spacing vanishes (Corollary~\ref{cor:invisible}) & \ref{sec:detectability} \\[8pt]
Degenerate, pinned & $f(\tau)=0$ (double-centering) vs.\ $X_\tau=0$ (rank-one conditioning): they disagree & \ref{sec:interior_pin} \\
\bottomrule
\end{tabular}
\end{table}
\vspace*{\fill}
\end{landscape}

\subsection{Recentering and conditioning}\label{sec:recentring_conditioning}

Recentering and conditioning, operations 2 and 3 of Section~\ref{sec:two_mechanisms}, begin with a specified covariance rather than boundary conditions. Their covariance identities will be used in Sections~\ref{sec:robin}--\ref{sec:degenerate}; further conditioning examples appear in Section~\ref{sec:conditioning}.

\emph{Recentering} subtracts a functional of the path outright. If $\ell$ is a linear functional of~$X$ and we form $C_s = X_s-\ell(X)$, bilinearity gives
\begin{equation}\label{eq:recentring}
 \Cov(C_s,C_t) = K(s,t)-K_\ell(s)-K_\ell(t)+V_\ell ,
\end{equation}
where $K_\ell$ and $V_\ell$ are defined in~\eqref{eq:functional_cov}. Since $C$ is a linear image of~$X$, \eqref{eq:recentring} is a covariance whenever $K$ is; no conditioning or operator construction is needed.

\emph{Conditioning} subtracts instead the Gaussian regression of the path on the observed functional.

\begin{theorem}[Rank-one conditioning]\label{thm:conditioning}
Let $X$ be zero-mean Gaussian with covariance~$K$, let $\ell$ be a linear functional whose variance $V_\ell = \Var(\ell(X))$ lies in $(0,\infty)$, and write $K_\ell(s) = \Cov(X_s,\ell(X))$ as in~\eqref{eq:functional_cov}. The covariance of~$X$ conditioned on $\ell(X)=0$ is
\begin{equation}\label{eq:conditioning}
 G_0(s,t) = K(s,t) - \frac{K_\ell(s)\,K_\ell(t)}{V_\ell}.
\end{equation}
\end{theorem}

This is the standard Gaussian conditioning formula, a rank-one Schur complement. In~\eqref{eq:recentring}, each path value loses the observation itself; in~\eqref{eq:conditioning}, it loses its least-squares projection onto that observation. The resulting covariances generally differ, as Section~\ref{sec:bc_vs_constraints} shows. Section~\ref{sec:conditioning} develops the conditioning examples, while Section~\ref{sec:degenerate} uses recentering to construct covariances associated with degenerate realizations.

\section{Local boundary conditions: the Robin family}\label{sec:robin}

\subsection{The two-parameter family}

Robin conditions couple a function and its derivative at the same endpoint. We write
\begin{equation}\label{eq:robin}
 f(0) = \lambda_0\,f'(0), \qquad f(1) = -\lambda_1\,f'(1), \qquad \lambda_0,\lambda_1 \in [0,\infty],
\end{equation}
with signs chosen so that $\lambda_i=0$ gives a Dirichlet condition (pinned end) and $\lambda_i\to\infty$ gives a Neumann condition (free end). Positive finite~$\lambda_i$ give an elastic endpoint penalty of stiffness~$1/\lambda_i$, interpolating between these limits.

\subsection{The Green's function}

\begin{theorem}[Robin Green's function]\label{thm:robin}
For $\lambda_0,\lambda_1\in[0,\infty)$ the Green's function of $(\mathcal{L},\mathcal{B})$, with boundary conditions~$\mathcal{B}$ given by~\eqref{eq:robin}, is
\begin{equation}\label{eq:robin_green}
 G(s,t) = \frac{(\lambda_0 + s\wedge t)(\lambda_1 + 1 - s\vee t)}{1+\lambda_0+\lambda_1},
\end{equation}
where $s\vee t = \max(s,t)$. Figure~\ref{fig:robin_family} shows how the parameters shape the covariance.
\end{theorem}

\begin{proof}
Fix $t\in(0,1)$ and write $G(s,t) = a_1 + b_1 s$ for $s<t$ and $G(s,t)=a_2+b_2 s$ for $s>t$, with coefficients depending on~$t$. The matching conditions~\eqref{eq:matching} give
\[
 a_1 + b_1 t = a_2 + b_2 t, \qquad b_2 - b_1 = -1.
\]
At $s=0$, the Robin condition gives $a_1 = \lambda_0 b_1$; at $s=1$, it gives $a_2+b_2 = -\lambda_1 b_2$, or $a_2 = -(1+\lambda_1)b_2$. Substituting these relations from~\eqref{eq:robin} and $b_1 = b_2+1$ into the continuity equation yields
\[
 \lambda_0(b_2+1)+(b_2+1)t = -(1+\lambda_1)b_2 + b_2 t \implies b_2\,(1+\lambda_0+\lambda_1) = -(\lambda_0+t).
\]
Hence
\[
 b_2 = -\frac{\lambda_0+t}{1+\lambda_0+\lambda_1},\quad b_1 = \frac{\lambda_1+1-t}{1+\lambda_0+\lambda_1},\quad a_1 = \frac{\lambda_0(\lambda_1+1-t)}{1+\lambda_0+\lambda_1},\quad a_2 = \frac{(\lambda_0+t)(1+\lambda_1)}{1+\lambda_0+\lambda_1}.
\]
For $s<t$, this gives $G = (\lambda_0+s)(\lambda_1+1-t)/(1+\lambda_0+\lambda_1)$. For $s>t$, the coefficients give the same expression with $s$ and~$t$ exchanged. Combining the two cases with $s\wedge t$ and $s\vee t$ yields~\eqref{eq:robin_green}.
\end{proof}

For nonnegative finite parameters, the Robin kernel also has a construction from Brownian motion and two independent Gaussian variables. Here $\lambda_0$ is a prior variance at the left endpoint and $\lambda_1$ an observation-noise variance at the right, giving a probabilistic interpretation of the penalties in~\eqref{eq:energy}.

\begin{proposition}[Noisy-endpoint construction of the Robin family]\label{prop:robin_conditioning}
Let $W$ be standard Brownian motion and let $Z,Z'$ be standard normal variables independent of~$W$ and of each other. For $\lambda_0,\lambda_1\in[0,\infty)$ put
\begin{equation}\label{eq:robin_conditioning}
 Y_s = \sqrt{\lambda_0}\,Z + W_s , \qquad\text{and condition on}\qquad Y_1 + \sqrt{\lambda_1}\,Z' = 0 .
\end{equation}
The resulting process has covariance $G(s,t)$ as in~\eqref{eq:robin_green}.
\end{proposition}

\begin{proof}
The prior~$Y$ has $\Cov(Y_s,Y_t)=\lambda_0+s\wedge t$. The noisy observation $Y_1+\sqrt{\lambda_1}Z'$ is jointly Gaussian with~$Y$, with cross-covariance $K_\ell(s)=\lambda_0+s$ and variance $V_\ell = \lambda_0+1+\lambda_1$. Applying the same Gaussian conditioning identity as in~\eqref{eq:conditioning} gives $\lambda_0+s\wedge t - (\lambda_0+s)(\lambda_0+t)/(1+\lambda_0+\lambda_1)$, which simplifies to~\eqref{eq:robin_green}.
\end{proof}

Thus $\lambda_0$ is a prior variance at the left endpoint and $\lambda_1$ an observation-noise variance at the right. For positive parameters, their precisions $1/\lambda_0$ and $1/\lambda_1$ match the penalties in~\eqref{eq:energy}. At $\lambda_0=0$, the left value is known exactly and the prior is Brownian motion from zero. If also $\lambda_1=0$, the noiseless observation $Y_1=0$ gives the Brownian bridge. As $\lambda_1\to\infty$ instead, the observation becomes uninformative, leaving a free right end. The general construction uses the auxiliary variables $Z,Z'$ rather than conditioning $W$ alone. Operations~1 and~3 therefore reach the same nonnegative-parameter family by different routes: boundary conditions define it, and noisy-endpoint conditioning constructs it.

\begin{figure}[tbp]
 \centering
 \includegraphics[width=\textwidth]{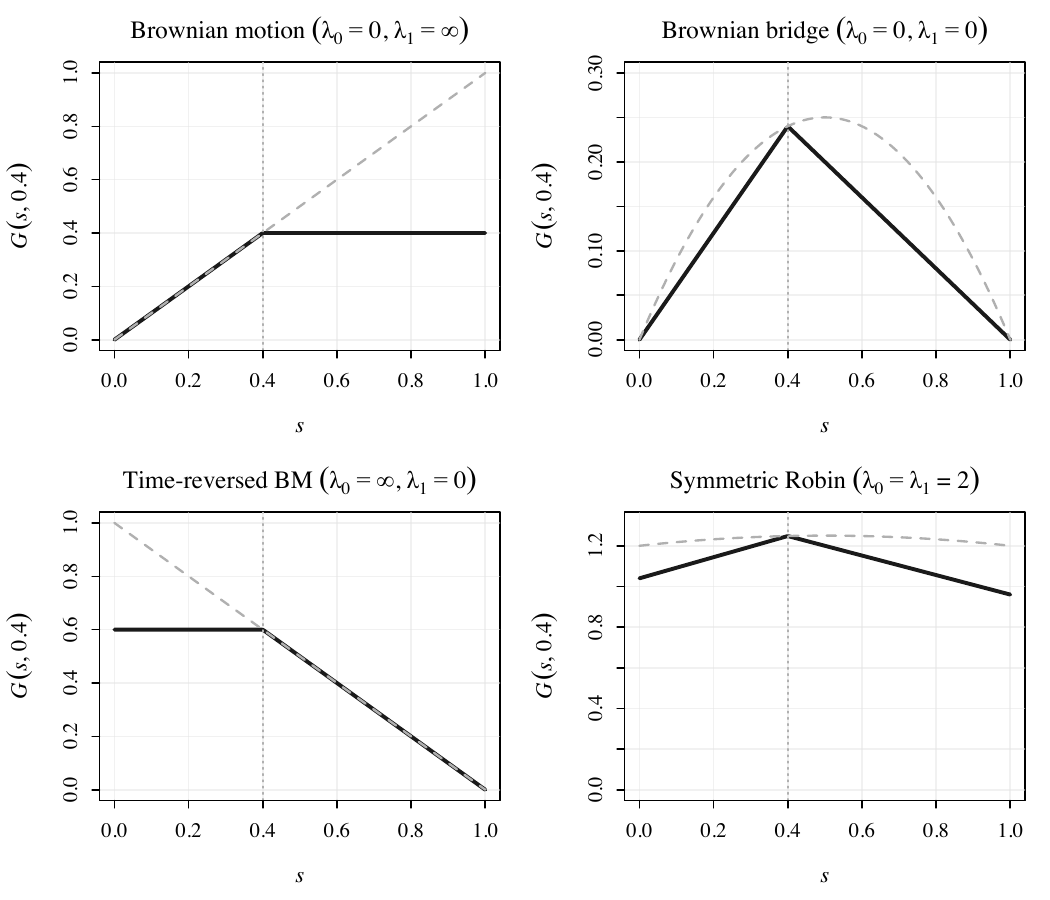}
 \caption{Robin covariance sections. Solid curves show $G(s,t)$ for fixed $t=0.4$; dashed curves show the variance $G(s,s)$, and the vertical line marks $s=t=0.4$. The selected values of $\lambda_0$ and $\lambda_1$ illustrate the family connecting Brownian motion, the Brownian bridge and time-reversed Brownian motion at the nondegenerate corners of $[0,\infty]^2$.}
 \label{fig:robin_family}
\end{figure}

\subsection{The corners of the family}\label{sec:corners}

The four corners of $(\lambda_0,\lambda_1)\in[0,\infty]^2$ have each $\lambda_i$ equal to $0$ or~$\infty$, with infinite values understood as limits of~\eqref{eq:robin_green}. Three give familiar nondegenerate processes: $(0,\infty)$ is Brownian motion, $(0,0)$ the Brownian bridge, and $(\infty,0)$ time-reversed Brownian motion. The fourth corner, $(\infty,\infty)$, is the degenerate Neumann--Neumann case, treated separately in Section~\ref{sec:degenerate}.

\begin{example}[Brownian motion: $\lambda_0=0$, $\lambda_1\to\infty$]\label{ex:BM}
The limiting boundary conditions are
\begin{align*}
 G(0,t) &= 0 \qquad \text{(Dirichlet at } s=0\text{)}, \\
 \partial_s G(1,t) &= 0 \qquad \text{(Neumann at } s=1\text{)}.
\end{align*}
The corresponding kernel is
\[
 G(s,t) = \lim_{\lambda_1\to\infty} \frac{(s\wedge t)(\lambda_1+1-s\vee t)}{1+\lambda_1} = s\wedge t,
\]
recovering Proposition~\ref{prop:BM}. The process is pinned at~0 and free at~1, with $\Var(X_s)=s$. Figure~\ref{fig:detail_bm} shows its covariance surface and sample paths.
\end{example}

\begin{figure}[tbp]
 \centering
 \includegraphics[width=0.85\textwidth]{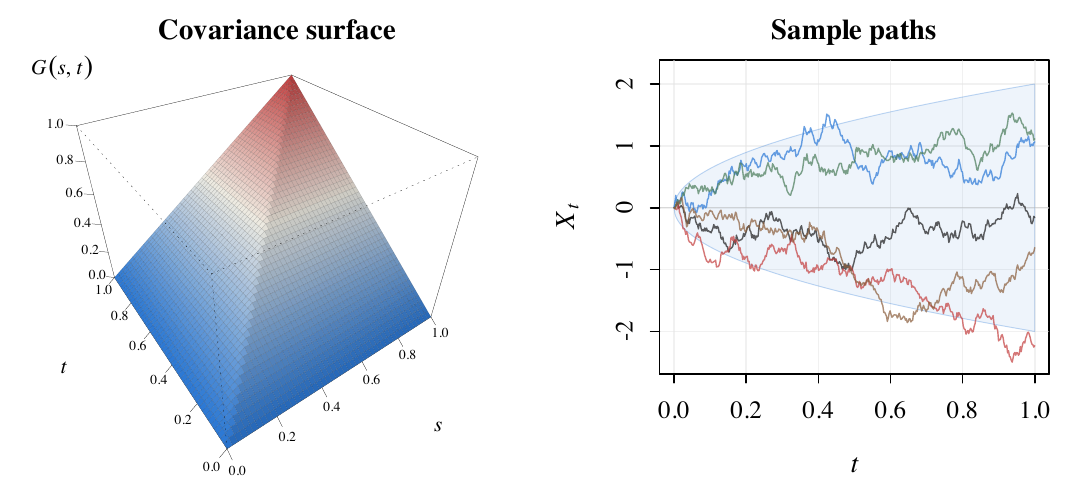}
 \caption{Brownian motion: covariance surface $G(s,t)$ over $[0,1]^2$ (left) and five sample paths with a pointwise $\pm2$-standard-deviation envelope (right).}
 \label{fig:detail_bm}
\end{figure}

\needspace{5\baselineskip}
\begin{example}[Brownian bridge: $\lambda_0=\lambda_1=0$]\label{ex:bridge}
The boundary conditions are
\begin{align*}
 G(0,t) &= 0 \qquad \text{(Dirichlet at } s=0\text{)}, \\
 G(1,t) &= 0 \qquad \text{(Dirichlet at } s=1\text{)}.
\end{align*}
The corresponding kernel is
\[
 G(s,t) = (s\wedge t)(1-s\vee t) = s\wedge t - st,
\]
with both endpoints pinned and variance $\Var(X_s) = s(1-s)$. The same law is obtained by conditioning Brownian motion on $W_1=0$ (Example~\ref{ex:bridge_cond}). Figure~\ref{fig:detail_bridge} shows the covariance surface and sample paths.
\end{example}

\begin{figure}[tbp]
 \centering
 \includegraphics[width=0.85\textwidth]{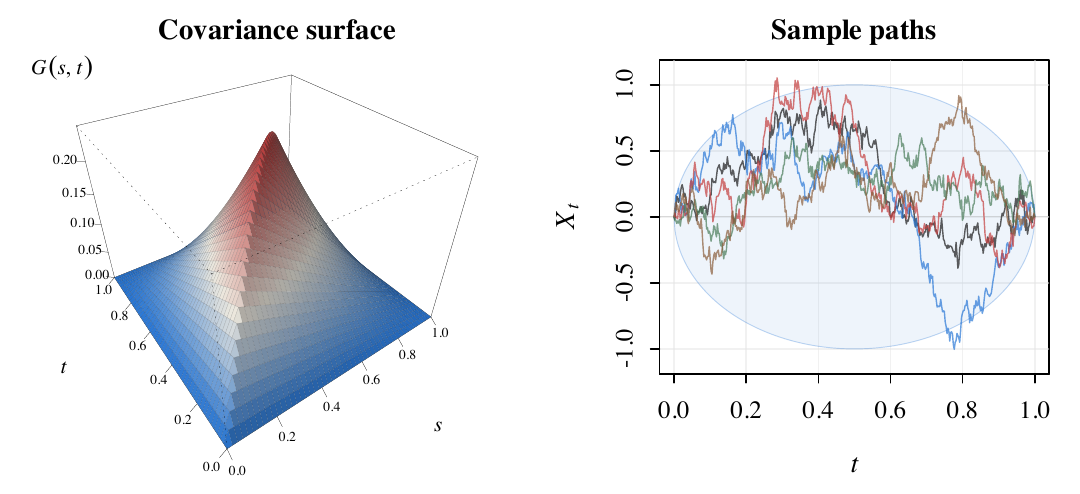}
 \caption{Brownian bridge: covariance surface $G(s,t)$ over $[0,1]^2$ (left) and five sample paths with a pointwise $\pm2$-standard-deviation envelope (right).}
 \label{fig:detail_bridge}
\end{figure}

\needspace{5\baselineskip}
\begin{example}[Time-reversed Brownian motion: $\lambda_0\!\to\!\infty$, $\lambda_1\!=\!0$]\label{ex:time_reversed}
The limiting boundary conditions are
\begin{align*}
 \partial_s G(0,t) &= 0 \qquad \text{(Neumann at } s=0\text{)}, \\
 G(1,t) &= 0 \qquad \text{(Dirichlet at } s=1\text{)}.
\end{align*}
The corresponding kernel is
\[
 G(s,t) = \lim_{\lambda_0\to\infty}\frac{(\lambda_0+s\wedge t)(1-s\vee t)}{1+\lambda_0} = 1 - s\vee t,
\]
with a free left end, a pinned right end and $\Var(X_s) = 1-s$. This is standard Brownian motion run backwards: under $s\mapsto 1-s$, $\Cov(W_{1-s},W_{1-t}) = (1-s)\wedge(1-t) = 1-(s\vee t)$. Figure~\ref{fig:detail_time_reversed} shows the covariance surface and sample paths.
\end{example}

\begin{figure}[tbp]
 \centering
 \includegraphics[width=0.85\textwidth]{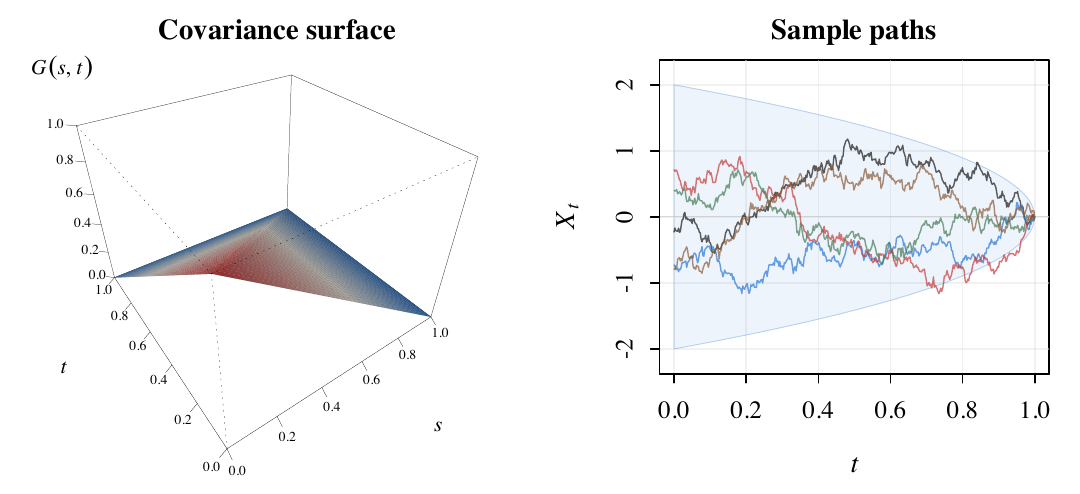}
 \caption{Time-reversed Brownian motion: covariance surface $G(s,t)$ over $[0,1]^2$ (left) and five sample paths with a pointwise $\pm2$-standard-deviation envelope (right).}
 \label{fig:detail_time_reversed}
\end{figure}

\begin{remark}[Endpoint correlations as coordinates]\label{rem:tau_correlation}
Evaluating~\eqref{eq:robin_green} at $s=t=0$, $s=t=1$ and $(s,t)=(0,1)$ gives the endpoint variances and covariance:
\[
 \Var(X_0)=\frac{\lambda_0(1+\lambda_1)}{1+\lambda_0+\lambda_1}, \qquad
 \Var(X_1)=\frac{(1+\lambda_0)\lambda_1}{1+\lambda_0+\lambda_1}, \qquad
 \Cov(X_0,X_1)=\frac{\lambda_0\lambda_1}{1+\lambda_0+\lambda_1} ,
\]
so that for $\lambda_0,\lambda_1>0$ the common denominator cancels and
\[
 \Corr(X_0,X_1) = \sqrt{\frac{\lambda_0}{1+\lambda_0}\cdot\frac{\lambda_1}{1+\lambda_1}} = \sqrt{\tau_0\tau_1} , \qquad \tau_i = \frac{\lambda_i}{1+\lambda_i} .
\]
The $\tau_i$ give alternative coordinates for the family. Substituting $\lambda_i = \tau_i/(1-\tau_i)$ into~\eqref{eq:robin_green} yields
\[
 G(s,t) = \frac{\bigl[\tau_0+(1-\tau_0)(s\wedge t)\bigr]\bigl[\tau_1+(1-\tau_1)(1-s\vee t)\bigr]}{1-\tau_0\tau_1} .
\]
These coordinates place the Neumann limit $\lambda_i=\infty$ on the finite edge $\tau_i = 1$ of the unit square. Each end contributes one factor $\tau_i$ to the squared endpoint correlation. On the symmetric ray $\tau=\tau_0=\tau_1$, we have $\Corr(X_0,X_1)=\tau$ wherever the correlation is defined: the coordinate is the endpoint correlation itself.

Approaching a Dirichlet edge from positive parameters, $\tau_i\to0$ gives $\Corr(X_0,X_1)\to0$. At the edge itself, one endpoint is deterministic, so its correlation with the other is undefined. Each of the three nondegenerate corner processes has at least one such endpoint. Approaching the fourth corner, $\tau_0,\tau_1\rightarrow1$ gives $\Corr(X_0,X_1)\to1$; Section~\ref{sec:NN_obstruction} examines this degenerate limit.
\end{remark}

\subsection{Outside the quadrant: negative Robin parameters}\label{sec:negative_branch}

The nonnegative quadrant $[0,\infty]^2$ of Theorem~\ref{thm:robin} is not the full admissible region. Formula~\eqref{eq:robin_green} also gives symmetric kernels for negative~$\lambda_i$ wherever its denominator is nonzero; the remaining question is positivity. Nonnegative endpoint stiffness, $1/\lambda_i\ge0$, is a modeling choice. These stiffnesses are coefficients in the energy form~\eqref{eq:energy}, not marginal endpoint precisions $1/\Var(X_i)$, which depend on both parameters (Remark~\ref{rem:tau_correlation}). A further admissible region lies beyond the quadrant.

To test positivity, write $a=1/\lambda_0$ and $b=1/\lambda_1$ for the boundary stiffnesses in~\eqref{eq:energy}, so that $\langle f,\mathcal{L}f\rangle = \int_0^1 f'(s)^2\,\dd s + a\,f(0)^2 + b\,f(1)^2$. For finite stiffnesses, minimize the associated energy form over $H^1[0,1]$ at fixed endpoint values $f(0)=p$, $f(1)=q$. The minimum derivative energy is $\int_0^1 f'(s)^2\,\dd s = (q-p)^2$, attained by the affine interpolant. This minimization takes place in the form domain, not $\operatorname{dom}(\mathcal{L})$. Strict positivity of the realization is therefore equivalent to positivity of the reduced quadratic form
\[
 Q(p,q) = (1+a)\,p^2 - 2pq + (1+b)\,q^2
\]
for every $(p,q)\ne(0,0)$. Evaluating at $(1,0)$ and $(0,1)$ requires $1+a>0$ and $1+b>0$. With $1+b>0$, completing the square in~$q$ gives
\[
 Q(p,q) = (1+b)\left(q-\frac{p}{1+b}\right)^2 + \left[(1+a)-\frac{1}{1+b}\right]p^2 ,
\]
so for fixed $p\ne0$, the minimum over~$q$ is positive exactly when the bracketed coefficient of~$p^2$ is positive, or $(1+a)(1+b)>1$. The positivity criterion can therefore be written as
\begin{equation}\label{eq:robin_admissible}
 1+a>0 \qquad\text{and}\qquad (1+a)(1+b)-1>0 .
\end{equation}

For finite nonzero parameters, the first condition in~\eqref{eq:robin_admissible} becomes $1+a=1+1/\lambda_0>0$, which holds exactly when $\lambda_0>0$ or $\lambda_0<-1$. Thus $\lambda_0\in[-1,0)$ is excluded, as is $\lambda_1\in[-1,0)$ by symmetry. The second condition becomes
\[
 (1+a)(1+b)-1 = \frac{1+\lambda_0+\lambda_1}{\lambda_0\lambda_1} > 0 .
\]
For $\lambda_0>0$, the admissible ranges are $\lambda_1\ge0$ and $\lambda_1<-(1+\lambda_0)$, with the Dirichlet edge checked directly from~\eqref{eq:robin_green}. The same alternatives hold at $\lambda_0=0$. For $\lambda_0<-1$, the admissible window is $0\le\lambda_1<-(1+\lambda_0)$; its width tends to zero as $\lambda_0\to-1^-$ and grows without bound as $\lambda_0\to-\infty$. The two branches beyond the quadrant are mirror images, each with one nonnegative parameter; two negative parameters are inadmissible. For $\lambda_0\ge0$, the inadmissible band $-(1+\lambda_0)<\lambda_1<0$ lies between the quadrant and the singular wall $1+\lambda_0+\lambda_1=0$, where the denominator of~\eqref{eq:robin_green} vanishes. Nonnegative endpoint stiffness is therefore a modeling choice, not a description of the full admissible set.

The next example lies beyond the wall: Brownian motion with an independent random slope.

\begin{example}[Brownian motion with a random slope]\label{ex:random_slope}
Take $\lambda_0=0$ (Dirichlet at the left end) and $\lambda_1 = -1-1/\sigma^2$ for some $\sigma>0$. Since $\lambda_1<-1$, this lies on the admissible Dirichlet edge identified above from~\eqref{eq:robin_admissible} and its limiting cases. Now $1+\lambda_0+\lambda_1 = -1/\sigma^2$, so~\eqref{eq:robin_green} gives
\begin{equation}\label{eq:random_slope}
 G(s,t) = \frac{(s\wedge t)\bigl(-\sigma^{-2}-s\vee t\bigr)}{-\sigma^{-2}} = (s\wedge t)\bigl(1+\sigma^2\,s\vee t\bigr) = s\wedge t + \sigma^2 st ,
\end{equation}
which is the covariance of
\[
 X_s = W_s + \sigma\,s\,Z, \qquad Z\sim N(0,1)\ \text{independent of Brownian motion}\ W.
\]
This is Brownian motion with an independent random slope, still pinned at $s=0$. As $\sigma\to0$, we have $\lambda_1\to-\infty$ and recover Brownian motion. The two limits $\lambda_1=+\infty$ and $\lambda_1=-\infty$ therefore meet at the same Neumann case. Joining the ends of the real line turns $\mathbb{R}\cup\{\infty\}$ into a circle. With $\lambda_0=0$, the admissible set $\lambda_1\ge0$ or $\lambda_1<-1$ is a single arc: it runs from Dirichlet at $\lambda_1=0$, through Neumann at $\lambda_1=\infty$, and towards the singular wall at $\lambda_1=-1$. The quadrant's edge $\lambda_1\in[0,\infty]$ covers only the first part of this arc. As $\lambda_1\rightarrow-1^ -$, the slope variance $\sigma^2$ diverges.

Its field marks appear in Table~\ref{tab:constructions}. The triangular kernel in~\eqref{eq:random_slope} makes the process Gauss--Markov, like the rest of the local family. The precision eigenvalues are $\omega^2$ for the positive roots of $\tan\omega=(1+\sigma^{-2})\omega$, while the covariance trace is $\tfrac12+\tfrac13\sigma^2$. The trace diverges with the slope variance as the wall is approached. Figure~\ref{fig:detail_random_slope} shows the covariance and sample paths.
\end{example}

\begin{figure}[tbp]
 \centering
 \includegraphics[width=0.85\textwidth]{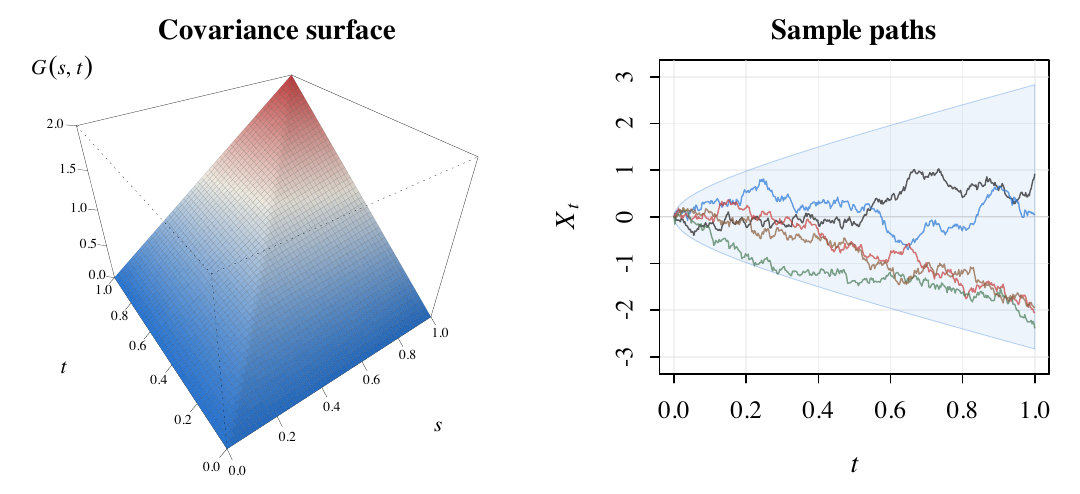}
 \caption{Brownian motion with a random slope at $\sigma=1$: covariance surface $G(s,t)=s\wedge t+\sigma^2st$ over $[0,1]^2$ (left) and five sample paths with a pointwise $\pm2$-standard-deviation envelope (right). The random linear trend adds a quadratic term to the variance, while the envelope has Brownian $\sqrt{s}$ behavior near the origin. The process starts at $0$ almost surely.}
 \label{fig:detail_random_slope}
\end{figure}

Returning to nonnegative parameters, the diagonal $\lambda_0=\lambda_1=\lambda$ gives
\begin{equation}\label{eq:symmetric_robin}
 G(s,t) = \frac{(\lambda+s\wedge t)(\lambda+1-s\vee t)}{1+2\lambda}, \qquad \Var(X_s) = \frac{(\lambda+s)(\lambda+1-s)}{1+2\lambda}.
\end{equation}
At $\lambda=0$ this is the bridge. As $\lambda\to\infty$, both endpoint penalties disappear and the covariance diverges, producing the Neumann--Neumann obstruction analyzed in Section~\ref{sec:NN_obstruction}.

For positive finite~$\lambda$, the endpoint penalties have finite stiffness. The term \emph{elastic} comes from Feller's boundary classification for one-dimensional diffusions, where $f(0)=\lambda f'(0)$ interpolates between absorption at $\lambda=0$ and reflection at $\lambda=\infty$; see \cite{ItoMcKean1974} for the pathwise description and \cite{Karatzas1991} for the analytic one. Here the same condition defines a Gaussian covariance, not the killed sub-Markov diffusion in that classification. In~\eqref{eq:energy}, the endpoint terms $\lambda^{-1}f(0)^2$ and $\lambda^{-1}f(1)^2$ supplement the roughness penalty $\int_0^1 f'(s)^2\,\dd s$. Large~$\lambda$ weakens the endpoint penalties, and small~$\lambda$ holds the endpoints near zero. The noisy-endpoint construction~\eqref{eq:robin_conditioning} gives the same interpolation through conditioning.

\subsection{The Neumann--Neumann obstruction}\label{sec:NN_obstruction}

Formula~\eqref{eq:robin_green} has no finite limit when $\lambda_0,\lambda_1\to\infty$ together. The covariance diverges because the limiting operator has a constant null mode.

\paragraph{Divergence of the covariance.}
Along the diagonal~\eqref{eq:symmetric_robin} the Green's function diverges,
\[
 G(s,t) = \frac{(\lambda+s\wedge t)(\lambda+1-s\vee t)}{1+2\lambda} \;\sim\; \frac{\lambda}{2} \;\longrightarrow\;\infty
 \qquad (\lambda\to\infty),
\]
uniformly in $(s,t)$. The correlation, on the other hand, converges uniformly to~$1$:
\[
 \Corr(X_s,X_t) = \frac{(\lambda+s\wedge t)(\lambda+1-s\vee t)}{\sqrt{(\lambda+s)(\lambda+1-s)\,(\lambda+t)(\lambda+1-t)}} \;\xrightarrow{\lambda\to\infty}\; 1 .
\]
Correlation tending to~$1$ does not make the unscaled process constant. Directly from~\eqref{eq:symmetric_robin}, the increments remain Brownian-sized: $\Var(X_t-X_s) = |t-s| - \frac{|t-s|^2}{1+2\lambda} \xrightarrow{\lambda\to\infty} |t-s|$. A common Gaussian component with variance $\sim\lambda/2$ dominates the marginal variances and drives the correlations to~$1$. The unscaled process has no finite limit. After multiplication by $\sqrt{2/\lambda}$, its covariance tends to that of a single standard Gaussian variable repeated at every time.
\paragraph{The constant null mode.}
Let $\mathbf{1}$ denote the constant function $\mathbf{1}(s)=1$ on $[0,1]$. With Neumann conditions at both ends, $\mathbf{1}$ belongs to $\operatorname{dom}(\mathcal{L})$ and
\[
 \mathcal{L}\mathbf{1} = 0 ,
\]
so $\ker(\mathcal{L}) = \operatorname{span}\{\mathbf{1}\}$: the operator annihilates the constants. It is therefore not invertible and has no Green's function in the sense of~\eqref{eq:greens}. Along the Robin family, the precision of the emerging null mode tends to zero while its variance diverges. Section~\ref{sec:degenerate} treats this \emph{degenerate operator} and the periodic realization by the same source correction and balanced normalization.

Figure~\ref{fig:parameter_map} summarizes the Robin parameter space. Panel~(a) places the three nondegenerate corners and the Neumann--Neumann corner on the compactified square; panel~(b) extends the map beyond the quadrant to the admissible branches and inadmissible band of Section~\ref{sec:negative_branch}.

\begin{figure}[!t]
 \centering
 \includegraphics[width=\textwidth]{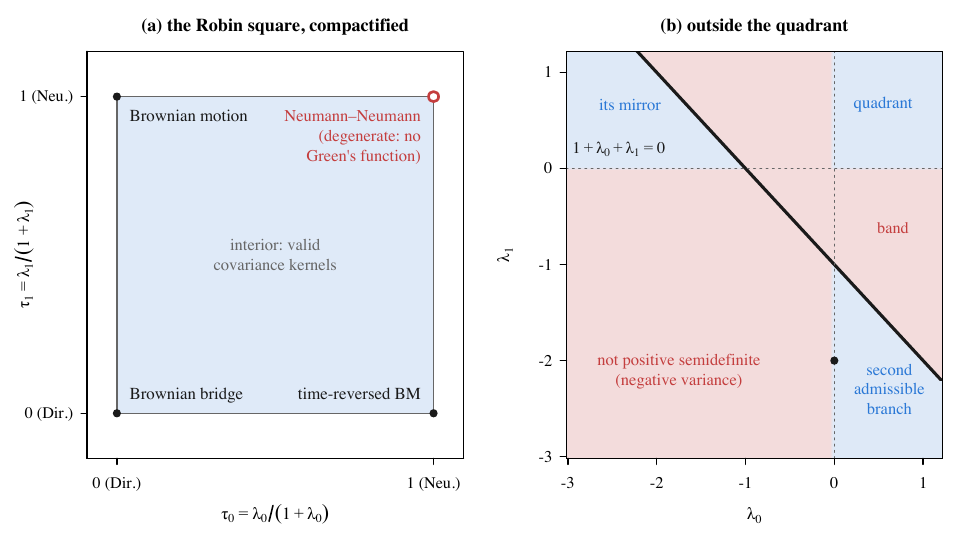}
 \caption{Robin parameter map. \emph{(a)} In the compactified coordinates $\tau_i=\lambda_i/(1+\lambda_i)$, Dirichlet ($\lambda_i=0$) lies at~$0$ and Neumann ($\lambda_i=\infty$) at~$1$. The corners give the Brownian bridge, Brownian motion, time-reversed Brownian motion and the degenerate Neumann--Neumann realization, which has no ordinary Green's function (Section~\ref{sec:NN_obstruction}). Each coordinate contributes one factor to the squared endpoint correlation (Remark~\ref{rem:tau_correlation}), and every interior point gives a covariance. \emph{(b)} The Robin family outside the nonnegative quadrant. Shading illustrates the admissible and inadmissible regions characterized in Section~\ref{sec:negative_branch}. The line $1+\lambda_0+\lambda_1=0$ marks where the Green kernel is singular. The marked point $(\lambda_0,\lambda_1)=(0,-2)$ gives Brownian motion with a random slope at $\sigma=1$, in the additional admissible branch.}
 \label{fig:parameter_map}
\end{figure}

\section{Nonlocal boundary conditions}\label{sec:nonlocal}

Robin conditions are local: each relates a function's value and derivative at a single endpoint. The nonlocal conditions considered first couple the two endpoints of $[0, 1]$: 
\begin{equation}\label{eq:nonlocal}
 f(0) = \alpha\,f(1), \qquad f'(0) = \beta\,f'(1), \qquad \alpha,\beta\in\mathbb{R}\setminus\{0\}.
\end{equation}

\subsection{Self-adjointness}

\begin{proposition}[Self-adjointness criterion]\label{prop:selfadj}
The operator $\mathcal{L}=-\dd^2/\dd s^2$ with conditions~\eqref{eq:nonlocal} is self-adjoint if and only if $\alpha\beta=1$.
\end{proposition}

\begin{proof}
Two integrations by parts give the boundary form $\langle\mathcal{L}f,g\rangle - \langle f,\mathcal{L}g\rangle = [f\,g'-f'\,g]_0^1$. For $f,g$ obeying~\eqref{eq:nonlocal},
\begin{align*}
 [fg'-f'g]_0^1 &= \bigl(f(1)g'(1)-f'(1)g(1)\bigr) - \bigl(\alpha f(1)\cdot\beta g'(1) - \beta f'(1)\cdot\alpha g(1)\bigr) \\
 &= (1-\alpha\beta)\bigl(f(1)g'(1)-f'(1)g(1)\bigr).
\end{align*}
This vanishes for all admissible $f,g$ if and only if $\alpha\beta=1$.
\end{proof}

Two self-adjoint cases are central: $\alpha=\beta=1$ (periodic) and $\alpha=\beta=-1$ (anti-periodic). The periodic operator annihilates constants and is degenerate, so we defer it to Section~\ref{sec:degenerate}. We treat the non-degenerate anti-periodic case next, then turn to the full one-parameter family.

\subsection{The anti-periodic process}

The case $\alpha=\beta=-1$ gives a stationary process with equal-and-opposite endpoints. It also appears as a special case of Theorem~\ref{thm:alpha} below.

\needspace{5\baselineskip}
\begin{example}[Anti-periodic Green's function]\label{thm:anti}
For $\alpha=\beta=-1$, the operator has trivial null space and Green's function
\begin{equation}\label{eq:anti}
 G_{\mathrm{anti}}(s,t) = \frac{1-2|s-t|}{4}.
\end{equation}

\begin{proof}
The two branches of $G(s,t)=(1-2|s-t|)/4$ are affine in~$s$, with slopes $\tfrac12$ and $-\tfrac12$. They agree at $s=t$ and have the derivative jump $b_2-b_1=-1$ required by~\eqref{eq:matching}. At the endpoints, $G(0,t)=\tfrac14-\tfrac t2$ and $G(1,t)=\tfrac t2-\tfrac14$ for $t\in(0,1)$, so $G(0,t)=-G(1,t)$; the constant slopes give $\partial_sG(0,t)=\tfrac12=-\partial_sG(1,t)$. Both boundary conditions hold.
\end{proof}

The anti-periodic process is stationary, with constant variance $G(s,s)=\tfrac{1}{4}$, vanishing correlation at lag~$\tfrac{1}{2}$, and---characteristically---perfectly anti-correlated endpoints:
\[
 \Corr(X_0,X_1) = \frac{G(0,1)}{G(0,0)} = \frac{-1/4}{1/4} = -1, \qquad\text{so}\quad X_1 = -X_0\text{ almost surely.}
\]
The operator is strictly positive: its eigenvalues under anti-periodic conditions are $\bigl((2k+1)\pi\bigr)^2$, $k\ge 0$. Its inverse therefore defines a covariance, although the endpoint pair is singular because $X_1=-X_0$ almost surely. See Figure~\ref{fig:detail_anti} for the covariance surface and sample paths.
\end{example}

\begin{figure}[tbp]
 \centering
 \includegraphics[width=0.85\textwidth]{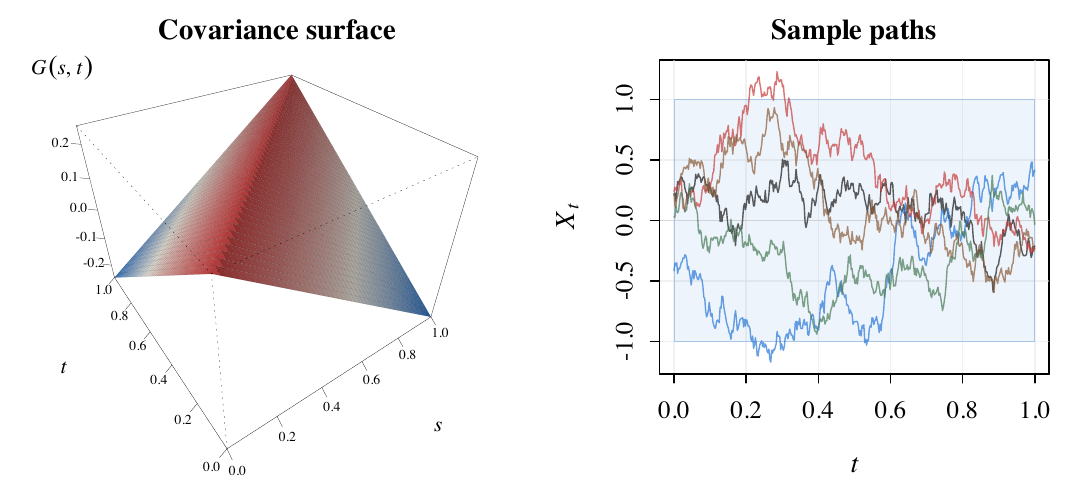}
 \caption{Anti-periodic process: covariance surface $G(s,t)$ on $[0,1]^2$ (left) and five sample paths with a $\pm2$-standard-deviation envelope (right).}
 \label{fig:detail_anti}
\end{figure}

\subsection{The general nonlocal family}

The self-adjointness constraint $\alpha\beta=1$ leaves a one-parameter family, with $\beta=1/\alpha$. Its non-degenerate members have $\alpha\in\mathbb{R}\setminus\{0,1\}$ and include the anti-periodic case $\alpha=-1$. The periodic case $\alpha=1$ is degenerate, not a finite covariance limit of this family. Other parameter limits connect to the Robin family. Figure~\ref{fig:nonlocal_family} shows the resulting covariance structure.

\begin{figure}[tbp]
 \centering
 \includegraphics[width=\textwidth]{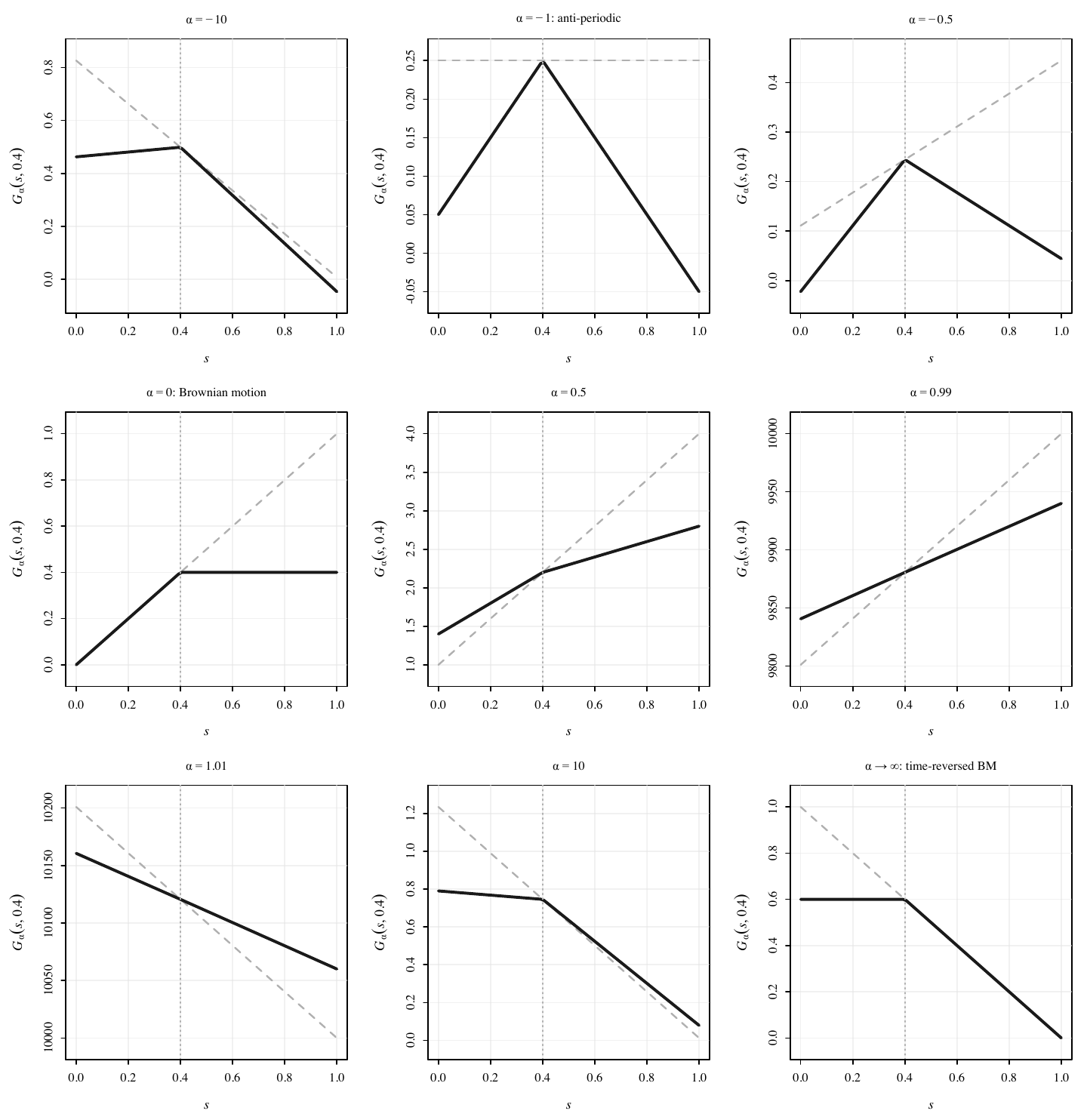}
 \caption{The $\alpha$-family of nonlocal covariance functions: the panels show $G_\alpha(s,0.4)$ (solid) and $\Var(X_s)=G_\alpha(s,s)$ (dashed) for $\alpha=-10,-1,-0.5,0,0.5,0.99,1.01,10,\infty$. The $\alpha=0$ and $\alpha\to\infty$ panels use the limiting kernels $s\wedge t$ (Brownian motion) and $1-s\vee t$ (time-reversed Brownian motion); $\alpha=-1$ is the anti-periodic process. The panels at $\alpha=0.99$ and $\alpha=1.01$ flank the excluded value $\alpha=1$, where the Green's function diverges.}
 \label{fig:nonlocal_family}
\end{figure}

We first write down the process, whose covariance follows by bilinearity, and then verify that this covariance solves the boundary-value problem.

\begin{proposition}[The $\alpha$-family from a Brownian path]\label{prop:alpha_pathwise}
Let $W$ be standard Brownian motion, let $\alpha\in\mathbb{R}\setminus\{0,1\}$, and set
\begin{equation}\label{eq:alpha_pathwise}
 X_s = W_s - a\,W_1 , \qquad a = \frac{\alpha}{\alpha-1} .
\end{equation}
Then
\begin{equation}\label{eq:alpha}
 \Cov(X_s,X_t) \;=\; G_\alpha(s,t) \;=\; \frac{\alpha^2}{(\alpha-1)^2} - \frac{(\alpha+1)(s+t)}{2(\alpha-1)} - \frac{|s-t|}{2}.
\end{equation}
The map $\alpha\mapsto a$ is an involution, with inverse $a\mapsto a/(a-1)$; it sends $\alpha=0$ to $a=0$, $\alpha=-1$ to $a=\tfrac12$, and $\alpha\to\infty$ to $a=1$, while the excluded value $\alpha=1$ corresponds to $a=\infty$.
\end{proposition}

\begin{proof}
Bilinearity and $\Cov(W_s,W_1)=s$ give $\Cov(X_s,X_t) = s\wedge t-a(s+t)+a^2$. Since $s\wedge t=\tfrac12(s+t)-\tfrac12|s-t|$, this equals
\[
 \Bigl(\tfrac12-a\Bigr)(s+t) + a^2 - \frac{|s-t|}{2} ,
\]
and the identities $\tfrac12-a = -(\alpha+1)/\bigl(2(\alpha-1)\bigr)$ and $a^2=\alpha^2/(\alpha-1)^2$ give~\eqref{eq:alpha}.
\end{proof}

The construction has two immediate consequences. Equation~\eqref{eq:alpha_pathwise} represents $X$ as a linear image of~$W$, so~\eqref{eq:alpha} is positive semidefinite. It also gives the endpoint coupling: $X_0=-a\,W_1$ and $X_1=(1-a)W_1$ are multiples of the same Gaussian variable, hence
\[
 X_0 = \frac{-a}{1-a}\,X_1 = \frac{a}{a-1}\,X_1 = \alpha\,X_1 \qquad\text{almost surely}.
\]
Thus the boundary condition $f(0)=\alpha f(1)$ holds pathwise, not just for the kernel. The construction does \emph{not} yet show that~\eqref{eq:alpha} is the corresponding operator's Green's function. The next theorem verifies this.

\begin{theorem}[The $\alpha$-family Green's function]\label{thm:alpha}
For $\alpha\in\mathbb{R}\setminus\{0,1\}$ with $\beta=1/\alpha$, the operator $\mathcal{L}=-\dd^2/\dd s^2$ with conditions $f(0)=\alpha f(1)$, $f'(0)=f'(1)/\alpha$ has trivial null space, and its Green's function is exactly the covariance~\eqref{eq:alpha}.
\end{theorem}

\begin{proof}
Fix $t\in(0,1)$ and write $G(s,t)=a_1+b_1 s$ for $s<t$ and $G(s,t)=a_2+b_2 s$ for $s>t$. The matching conditions~\eqref{eq:matching} give $a_1+b_1 t = a_2+b_2 t$ and $b_2 - b_1 = -1$. The boundary conditions $G(0,t)=\alpha\,G(1,t)$ and $\partial_s G(0,t)=(1/\alpha)\,\partial_s G(1,t)$ give $a_1 = \alpha(a_2+b_2)$ and $b_1 = b_2/\alpha$. From the derivative condition, $b_1 = b_2/\alpha$, combined with $b_2 = b_1-1$, we obtain $b_1 = (b_1-1)/\alpha$, hence $b_1(\alpha-1)=-1$ and $b_1 = -1/(\alpha-1)$. Then $b_2 = b_1-1 = -\alpha/(\alpha-1)$. Continuity gives $a_1-a_2 = (b_2-b_1)t = -t$. The value condition $a_1 = \alpha(a_2+b_2) = \alpha(a_1+t-\alpha/(\alpha-1))$ yields $a_1(1-\alpha) = \alpha t - \alpha^2/(\alpha-1)$, so $a_1 = \alpha^2/(\alpha-1)^2 - \alpha t/(\alpha-1)$. For $s<t$,
\[
 G = a_1 + b_1 s = \frac{\alpha^2}{(\alpha-1)^2} - \frac{\alpha t}{\alpha-1} - \frac{s}{\alpha-1} = \frac{\alpha^2}{(\alpha-1)^2} - \frac{\alpha t + s}{\alpha-1}.
\]
Writing $\alpha t + s = \tfrac{1}{2}(\alpha+1)(s+t) + \tfrac{1}{2}(\alpha-1)(t-s)$ and noting $|s-t|=t-s$ for $s<t$ gives~\eqref{eq:alpha}. The case $s>t$ follows by symmetry.
\end{proof}

\begin{remark}[Positivity from the operator side]\label{rem:alpha_positive}
Theorem~\ref{thm:alpha} establishes invertibility, not positivity. The energy form supplies positivity for every parameter value in the stated non-degenerate family. Integration by parts gives $\langle f,\mathcal{L}f\rangle = \int_0^1 f'(s)^2\,\dd s + f'(0)f(0) - f'(1)f(1)$ for any twice-differentiable~$f$, as in the derivation of~\eqref{eq:energy}. On the domain of the $\alpha$-family, $f(0)=\alpha f(1)$ and $f'(0)=f'(1)/\alpha$, so the two boundary terms cancel identically,
\[
 f'(0)f(0) - f'(1)f(1) = \frac{f'(1)}{\alpha}\cdot\alpha f(1) - f'(1)f(1) = 0 ,
\]
for every~$\alpha$ in the stated family, leaving
\begin{equation}\label{eq:alpha_energy}
 \langle f,\mathcal{L}f\rangle = \int_0^1 f'(s)^2\,\dd s .
\end{equation}
The right-hand side vanishes only for constant~$f$, and a constant in the domain must satisfy $c=\alpha c$ with $\alpha\ne1$, hence $c=0$; so $\langle f,\mathcal{L}f\rangle>0$ for every nonzero~$f$ in the domain. The pair is therefore admissible in the sense of Section~\ref{sec:green}, and $G_\alpha$ is a covariance function for every $\alpha\in\mathbb{R}\setminus\{0,1\}$. Proposition~\ref{prop:alpha_pathwise} reaches the same conclusion by representing the process as a linear image of a Brownian path.
\end{remark}

The variance and endpoint correlation are
\begin{align}
 \Var(X_s) &= G_\alpha(s,s) = \frac{\alpha^2}{(\alpha-1)^2} - \frac{(\alpha+1)s}{\alpha-1}, \label{eq:alpha_var} \\
 \Corr(X_0,X_1) &= \sgn(\alpha), \qquad X_0 = \alpha\, X_1 \text{ almost surely.} \label{eq:alpha_corr}
\end{align}
The deterministic endpoint coupling $X_0=\alpha X_1$ can be read off the covariance itself, with no appeal to the construction of Section~\ref{sec:from_green_to_cov}. From~\eqref{eq:alpha}, $G_\alpha(0,t) = \alpha\,G_\alpha(1,t)$ for every~$t$, and in particular $G_\alpha(0,0)=\alpha^2/(\alpha-1)^2$, $G_\alpha(0,1)=\alpha/(\alpha-1)^2$, $G_\alpha(1,1)=1/(\alpha-1)^2$. Hence
\[
 \Var(X_0-\alpha X_1) = G_\alpha(0,0) - 2\alpha\,G_\alpha(0,1) + \alpha^2 G_\alpha(1,1) = \frac{\alpha^2-2\alpha^2+\alpha^2}{(\alpha-1)^2} = 0 ,
\]
and a zero-mean Gaussian variable with zero variance vanishes almost surely. The algebraic boundary condition $f(0)=\alpha f(1)$ on the operator domain therefore reappears as an almost-sure identity between the process's endpoint values. When $\alpha>0$ the endpoints have the same sign; when $\alpha<0$ they have opposite signs, with the anti-periodic case $\alpha=-1$ giving equal-and-opposite endpoints.

The construction~\eqref{eq:alpha_pathwise} also gives $\Corr(X_0,X_1)=\sgn(\alpha)$ directly from the sign of $-a(1-a)$. It doubles as a simulation recipe: draw one Brownian path and subtract $a$ times its terminal value. The Brownian bridge is not of this form for any~$a$.

\begin{remark}[Limiting cases]\label{rem:alpha_limits}
Two parameter limits recover the Robin corners mentioned above:
\begin{itemize}
\item $\alpha\to 0$: here $\alpha^2/(\alpha-1)^2\to 0$ and $-(\alpha+1)(s+t)/(2(\alpha-1))\to(s+t)/2$, giving
\[
 G_\alpha(s,t)\to \frac{s+t}{2}-\frac{|s-t|}{2} = s\wedge t \quad\text{(Brownian motion).}
\]
\item $\alpha\to\pm\infty$: likewise $\alpha^2/(\alpha-1)^2\to 1$ and $-(\alpha+1)(s+t)/(2(\alpha-1))\to -(s+t)/2$ regardless of the sign of $\alpha$, giving
\[
 G_\alpha(s,t)\to 1-\frac{s+t}{2}-\frac{|s-t|}{2} = 1-s\vee t \quad\text{(time-reversed Brownian motion)}
\]
from either tail: the two ends of the unbounded parameter range meet at a single limit.
\end{itemize}
These limits connect the Robin and nonlocal families. As $\alpha\to0$, we obtain Brownian motion, the Robin corner $\lambda_0=0$, $\lambda_1=\infty$; as $\alpha\to\pm\infty$, we obtain time-reversed Brownian motion, the corner $\lambda_0=\infty$, $\lambda_1=0$. Both limits are bounded. The pathwise construction in Proposition~\ref{prop:alpha_pathwise} extends to every real~$a$, including $a=0,1$. At these two values the operator conditions are local: the processes are ordinary members of the pathwise family but only limits of the nonlocal operator family. Proposition~\ref{prop:robin_blowup} below examines the divergent Neumann--Neumann corner of the Robin family and relates it to the whole branch $\alpha<0$. The other named cases are $\alpha=-1$ (anti-periodic, Example~\ref{thm:anti}) and the excluded value $\alpha=1$ (periodic and degenerate, Section~\ref{sec:degenerate}).
\end{remark}

\begin{example}[Slepian Brownian motion]\label{ex:slepian}
The $\alpha$-family and its two limits do not exhaust the list of nonlocal Brownian beasts. Consider the pair $(\mathcal{L},\mathcal{B})$ with boundary conditions
\begin{equation}\label{eq:slepian_bc}
 f'(0)+f'(1) = 0, \qquad f(0)+f(1) = f'(0) .
\end{equation}
These couple the two endpoints, so they are nonlocal, but they are not of the form~\eqref{eq:nonlocal}, which relates $f(0)$ to $f(1)$ and $f'(0)$ to $f'(1)$ separately; the second condition in~\eqref{eq:slepian_bc} mixes endpoint values with a derivative. The resulting operator is nevertheless self-adjoint. Writing $p=f(0)$ and $q=f'(0)$, the conditions give $f(1)=q-p$ and $f'(1)=-q$, so for any two members $f,g$ of the domain, with $P=g(0)$ and $Q=g'(0)$,
\[
 \bigl[fg'-f'g\bigr]_0^1 = \bigl\{(q-p)(-Q)+q(Q-P)\bigr\} - \bigl\{pQ-qP\bigr\} = (pQ-qP)-(pQ-qP) = 0 ,
\]
and the operator is invertible, with
\begin{equation}\label{eq:slepian}
 G(s,t) = \frac{1-|s-t|}{2} .
\end{equation}
This is the anti-periodic kernel plus a constant, $G = G_{-1}+\tfrac14$, and no member of~\eqref{eq:alpha} supplies that constant. The coefficient of $s+t$ in $G_\alpha$ is $-(\alpha+1)/\bigl(2(\alpha-1)\bigr)$, which vanishes only at $\alpha=-1$, and that same choice pins the constant term to $\alpha^2/(\alpha-1)^2=\tfrac14$. Thus exactly one value of~$\alpha$ gives the right shape, but its constant falls short by~$\tfrac14$. Nor is~\eqref{eq:slepian} a Robin kernel: every member of~\eqref{eq:robin_green} has the triangular factorization, whereas~\eqref{eq:slepian} does not. It lies outside both families as parametrized. The constructions below show that~\eqref{eq:slepian} is a covariance.

We give three constructions, each showing a different property. Adding an independent Gaussian scalar to the anti-periodic process gives
\[
 X_s = W_s - \tfrac12 W_1 + \tfrac12 Z , \qquad Z\sim N(0,1) \text{ independent of } W,
\]
so the kernel is a rank-one enlargement of $G_{-1}$. Alternatively, start independent Brownian motions at opposite ends:
\[
 X_s = \frac{W_s+\widetilde W_{1-s}}{\sqrt2} , \qquad \widetilde W \text{ a standard Brownian motion independent of } W.
\]
This explains the independent endpoints: $X_0=\widetilde W_1/\sqrt2$ and $X_1=W_1/\sqrt2$ come from independent paths. The covariances add, with $s\wedge t+(1-s\vee t)=1-|s-t|$. Finally, the representation
\[
 \sqrt2\,X_s = W_{s+1}-W_s
\]
identifies the process: $\sqrt2\,X$, not $X$ itself, consists of Brownian increments over a moving window of unit width. This is the unit-variance \emph{Slepian process}~\cite{Slepian1961,Shepp1971}, restricted to $[0,1]$ and studied there for its boundary-crossing probabilities. Our~$X$ has $\Var(X_s)=\tfrac12$, so it is a rescaled version. We call it Slepian Brownian motion to keep it with the rest of the bestiary, under the same Brownian normalization. Brownian scaling turns the third construction into one that never leaves the unit interval: with $W'_t=W_{2t}/\sqrt2$, itself a standard Brownian motion on $[0,1]$, $X_s = W'_{(s+1)/2}-W'_{s/2}$. Thus the process is a linear image of a single Brownian path on $[0,1]$ after all. The coefficient comparison above shows, more specifically, that it is not of the form $W-aW_1$ for any~$a$.

It is stationary: $\Var(X_s)=\tfrac12$ for every~$s$, and $G$ depends on $s-t$ alone. The anti-periodic process is also stationary, but its ends are tied together, $X_0=-X_1$; more stationary beasts will appear later. Here the endpoints are neither tied nor pinned: they are independent, since $G(0,1)=0$. Within the catalogue, only Slepian Brownian motion combines stationarity with independent endpoints. It is not Gauss--Markov: for $s<t<u$, $G(s,t)G(t,u)-G(s,u)G(t,t)=\tfrac14(t-s)(u-t)\ne0$, so the kernel has no triangular factorization. The $|s-t|$ term alone is not the obstruction: the Brownian kernel also contains one, yet is triangular.

The spectrum splits into two sectors. For the anti-periodic operator, $\cos\omega_ns$ and $\sin\omega_ns$ share the eigenvalue $\omega_n^2$ at every $\omega_n=(2n-1)\pi$; the corresponding covariance eigenvalue is $1/\omega_n^2$. Adding a constant to the kernel breaks this degeneracy only in the sector spanned by the sine modes: $\int_0^1\cos\omega_ns\,\dd s=0$, whereas $\int_0^1\sin\omega_ns\,\dd s=2/\omega_n\neq0$. The cosine modes remain at $\omega=(2k+1)\pi$, while the frequencies in the other sector satisfy
\begin{equation}\label{eq:slepian_spectrum}
 \tfrac12\,\omega\tan(\omega/2) = 1 ,
\end{equation}
and every eigenvalue is now simple. The covariance trace $\int_0^1 G(s,s)\,\dd s=\tfrac12$ splits into $\tfrac18$ from the cosine sector and $\tfrac38$ from the original sine sector. The anti-periodic trace of $\tfrac14$ divides evenly between them; the added constant contributes a further $\tfrac14$ entirely to the latter. Figure~\ref{fig:detail_slepian} shows the covariance surface and sample paths, with stationarity and independent endpoints both visible: the ridge along $s=t$ sits at height $\tfrac12$, and the surface touches zero at the two off-diagonal corners.
\end{example}

\begin{figure}[tbp]
 \centering
 \includegraphics[width=0.85\textwidth]{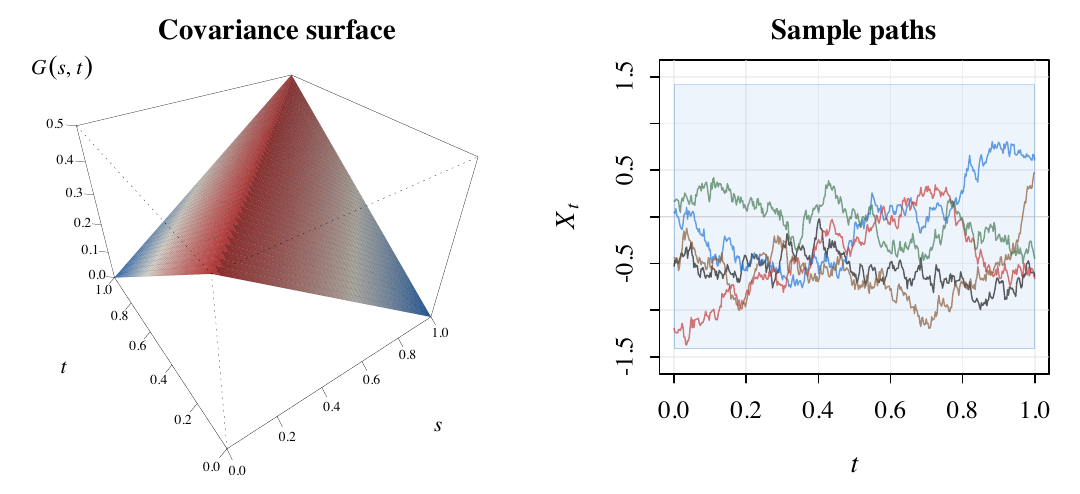}
 \caption{Slepian Brownian motion: covariance surface $G(s,t)$ on $[0,1]^2$ (left) and five sample paths with a $\pm2$-standard-deviation envelope (right). The constant-width envelope reflects stationarity. The covariance vanishes at $(0,1)$ and $(1,0)$, showing that the endpoints are independent.}
 \label{fig:detail_slepian}
\end{figure}

\begin{proposition}[The negative nonlocal branch as a Robin limit]\label{prop:robin_blowup}
For the Robin conditions $f(0)=\lambda_0\,f'(0)$, $f(1)=-\lambda_1\,f'(1)$ in~\eqref{eq:robin}, the corner $\lambda_0=\lambda_1=\infty$ is Neumann--Neumann: both ends are free, and no ordinary Green's function exists (Section~\ref{sec:NN_obstruction}). Fix $a\in(0,1)$ and approach this corner along the ray $\lambda_0=\lambda$, $\lambda_1=(1-a)\lambda/a$, with $\lambda\to\infty$. Then, uniformly in $(s,t)\in[0,1]^2$,
\begin{equation}\label{eq:robin_blowup}
 G_{\lambda_0,\lambda_1}(s,t) - \frac{\lambda_0\lambda_1}{1+\lambda_0+\lambda_1}
 \;=\; (1-a)\,(s\wedge t) + a\,(1-s\vee t) + O(1/\lambda) ,
\end{equation}
where the subtracted term is the endpoint covariance $\Cov(X_0,X_1)$ of Remark~\ref{rem:tau_correlation}, which carries the divergence described in Section~\ref{sec:NN_obstruction}. The limit is the covariance of $\sqrt{1-a}\,W_s+\sqrt{a}\,\widetilde W_{1-s}$ for independent Brownian motions $W,\widetilde W$. It equals $G_\alpha+a(1-a)$ for $\alpha=a/(a-1)$, with $G_\alpha$ as in~\eqref{eq:alpha}. Subtracting $(1-a)\lambda$ from the Robin kernel instead gives the limit $G_\alpha$ itself. The parameter $a$ is the coefficient of~\eqref{eq:alpha_pathwise}, and $a\mapsto\alpha=a/(a-1)$ is a decreasing bijection from $(0,1)$ onto $(-\infty,0)$. At $a=\tfrac12$, the ray ratio $r=(1-a)/a$ is $1$, so the ray is symmetric. Here $\alpha=-1$, and the limit in~\eqref{eq:robin_blowup} is exactly the Slepian Brownian motion covariance~\eqref{eq:slepian}.
\end{proposition}

\begin{proof}
Since $(\lambda_0+s\wedge t)(\lambda_1+1-s\vee t) = \lambda_0\lambda_1 + \lambda_1(s\wedge t) + \lambda_0(1-s\vee t) + (s\wedge t)(1-s\vee t)$,
\[
 G_{\lambda_0,\lambda_1}(s,t) - \frac{\lambda_0\lambda_1}{1+\lambda_0+\lambda_1}
 = \frac{\lambda_1(s\wedge t) + \lambda_0(1-s\vee t) + (s\wedge t)(1-s\vee t)}{1+\lambda_0+\lambda_1} .
\]
Along the ray $\lambda_0=\lambda$, $\lambda_1=r\lambda$ with $r=(1-a)/a$, write $L=(1-a)(s\wedge t)+a(1-s\vee t)=\bigl[r(s\wedge t)+(1-s\vee t)\bigr]/(1+r)$ for the claimed limit. Subtracting $L$ from the right-hand side and using $L(1+r)=r(s\wedge t)+(1-s\vee t)$ cancels the $\lambda$-terms in the numerator and gives the exact identity
\[
 G_{\lambda_0,\lambda_1}(s,t) - \frac{\lambda_0\lambda_1}{1+\lambda_0+\lambda_1} - L
 \;=\; \frac{(s\wedge t)(1-s\vee t)-L}{(1+r)\lambda+1}
 \;=\; \frac{a\bigl[(s\wedge t)(1-s\vee t)-L\bigr]}{\lambda+a} ,
\]
using $1+r=1/a$. The bracketed numerator is bounded on $[0,1]^2$, so this is $O(1/\lambda)$ uniformly in $(s,t)$ and in~$a$, giving~\eqref{eq:robin_blowup}. The same substitution gives the exact identity $\lambda_0\lambda_1/(1+\lambda_0+\lambda_1) - (1-a)\lambda = -a(1-a)\lambda/(a+\lambda) \to -a(1-a)$. Subtracting $(1-a)\lambda$ instead therefore shifts the limiting kernel down by~$a(1-a)$.

For the pathwise interpretation, independence of $W$ and $\widetilde W$, together with $(1-s)\wedge(1-t)=1-s\vee t$, gives
\[
 \Cov\bigl(\sqrt{1-a}\,W_s+\sqrt a\,\widetilde W_{1-s},\ \sqrt{1-a}\,W_t+\sqrt a\,\widetilde W_{1-t}\bigr)
 = (1-a)(s\wedge t)+a(1-s\vee t) ,
\]
so the limit is positive semidefinite. At $a=\tfrac12$, this recovers the second construction in Example~\ref{ex:slepian}. Since $s\wedge t=\tfrac12(s+t-|s-t|)$ and $1-s\vee t=\tfrac12(2-s-t-|s-t|)$,
\[
 (1-a)(s\wedge t)+a(1-s\vee t) = a + \frac{(1-2a)(s+t)}{2} - \frac{|s-t|}{2} ,
\]
which is $G_\alpha(s,t)+a(1-a)$ at $\alpha=a/(a-1)$: that value of~$\alpha$ gives $\alpha^2/(\alpha-1)^2=a^2$ and $-(\alpha+1)/\{2(\alpha-1)\}=(1-2a)/2$ (the computation in the proof of Proposition~\ref{prop:alpha_pathwise}), so~\eqref{eq:alpha} equals $a^2+(1-2a)(s+t)/2-|s-t|/2$. This lies exactly $a(1-a)$ below the preceding display, matching the shift just found. Subtracting $(1-a)\lambda$ from the Robin kernel therefore gives the limit $G_\alpha$ itself. Finally $a\mapsto a/(a-1)$ has derivative $-1/(a-1)^2<0$ and carries $(0,1)$ onto $(-\infty,0)$; at $a=\tfrac12$, $r=1$ and $(1-a)(s\wedge t)+a(1-s\vee t)=\tfrac12\{(s\wedge t)+(1-s\vee t)\}=\tfrac12(1-|s-t|)$, using $s\vee t-s\wedge t=|s-t|$, which is~\eqref{eq:slepian}.
\end{proof}

We note two features of~\eqref{eq:robin_blowup}. First, the coefficient comparison identifies the entire limiting kernel. Both limiting kernels contain the same parameter-independent kink term $-|s-t|/2$; the remainder is affine in $s+t$. Second, these rays reach only $\alpha<0$, where the endpoints are perfectly anticorrelated throughout the branch. Reaching $\alpha>0$ requires $a<0$ or $a>1$, hence $r<0$ and a negative Robin parameter. Along the same ray, take $\lambda\to+\infty$ when $a<0$ and $\lambda\to-\infty$ when $a>1$; the ray then lies in the branch $\lambda_1<-(1+\lambda_0)$ or in its mirror image, respectively, and the identity in the proof holds as $|\lambda|\to\infty$. Those members are approached from the branch in Section~\ref{sec:negative_branch}, not from the admissible quadrant. This restriction concerns the rays, not the existence of the processes. For every $\alpha\in \mathbb{R}\setminus\{1\}$, $G_\alpha$ is a covariance by the construction in Proposition~\ref{prop:alpha_pathwise}, extended to $\alpha=0$. The value $\alpha=1$ is a pole of~\eqref{eq:alpha}; its periodic operator is the separate degenerate case in Section~\ref{sec:degenerate}, not a finite covariance limit of this family. The endpoint-symmetric member $\alpha=-1$ lies on the symmetric ray $r=1$. The open interval of ray parameters has the two bounded limits in Remark~\ref{rem:alpha_limits}: $a\to0$ gives Brownian motion, and $a\to1$ gives its time reversal, the $\alpha\to-\infty$ limit of~\eqref{eq:alpha} agreeing there with the $\alpha\to\infty$ limit already recorded.

\subsection{Why a bulk condition cannot replace an endpoint condition}\label{sec:bulk}

The conditions~\eqref{eq:nonlocal} are nonlocal because they couple the two endpoints. Another natural possibility couples an endpoint to a functional of the whole function, as in
\begin{equation}\label{eq:bulk}
 f(0) = \int_0^1 f(s)\,\dd s .
\end{equation}
This homogeneous linear condition is neither a local Robin condition~\eqref{eq:robin} nor an endpoint coupling of the form~\eqref{eq:nonlocal}. Can it replace one of the two endpoint conditions and still produce an admissible operator? Proposition~\ref{prop:bulk} answers no for the class of linear conditions considered below. The obstruction comes from the boundary form, not simply the number of conditions.

Write $\mathcal{D}=H^2[0,1]$ for the maximal domain of the differential expression $\mathcal{L}$ in $L^2[0,1]$, and let
\[
 \Gamma f = \bigl(f(0),\,f'(0),\,f(1),\,f'(1)\bigr) \in \mathbb{R}^4
\]
collect the endpoint data of~$f$. As in Proposition~\ref{prop:selfadj}, two integrations by parts give $\langle\mathcal{L}f,g\rangle-\langle f,\mathcal{L}g\rangle=J(\Gamma f,\Gamma g)=[fg'-f'g]_0^1$. This boundary form depends only on the endpoint traces $\Gamma f,\Gamma g$: the interior has been integrated away. $J$ is a nondegenerate skew-symmetric form on $\mathbb{R}^4$, and a nondegenerate skew form on $\mathbb{R}^4$ vanishes identically on a subspace only if that subspace has dimension at most~$2$; a $2$-dimensional subspace on which it vanishes is called a \emph{Lagrangian plane} of~$J$. Two independent homogeneous conditions \emph{on the endpoint traces} define a $2$-dimensional subspace $\Gamma(\operatorname{dom}(\mathcal{L}))$ of~$\mathbb{R}^4$: two equations in four unknowns. Self-adjointness also requires this plane to be Lagrangian, that is, $J$ must vanish on it. That is exactly what Proposition~\ref{prop:selfadj} checks for~\eqref{eq:nonlocal}: the two conditions cut out a $2$-plane depending on $\alpha,\beta$, and $\alpha\beta=1$ is precisely the constraint making it Lagrangian.

A bulk condition behaves differently. Here a condition is called \emph{bulk} if it contains a nonzero integral term $\int_0^1 w(s)f(s)\,\dd s$, with $w\in L^2[0,1]$ and $w\ne0$; \eqref{eq:bulk} uses $w=\mathbf{1}$. The map $\Gamma$ only reads off values and derivatives \emph{at} the two endpoints, and $\int_0^1 w f$ can be changed at will by adding an interior bump function that vanishes, together with its derivative, at $0$ and $1$; so a bulk term places no restriction whatsoever on $\Gamma f$. For any second linear condition of the form stated below---local, endpoint-coupling or bulk---at most one independent restriction on $\Gamma f$ remains after this interior freedom is used. Thus $\Gamma(\operatorname{dom}(\mathcal{L}))$ has dimension at least~$3$. The nondegenerate skew form $J$ cannot vanish on a subspace of dimension greater than~$2$, so such a pair of conditions cannot make $\mathcal{L}$ symmetric.

\begin{proposition}[Bulk conditions cannot make the Laplacian symmetric]\label{prop:bulk}
Let $w_1,w_2\in L^2[0,1]$, not both zero, let $\Lambda_1,\Lambda_2$ be linear functionals on~$\mathbb{R}^4$, and restrict $\mathcal{L}=-\dd^2/\dd s^2$ to
\[
 \operatorname{dom}(\mathcal{L}) = \Bigl\{\, f\in\mathcal{D} \;:\; \Lambda_i(\Gamma f) + \int_0^1 w_i(s)\,f(s)\,\dd s = 0,\quad i=1,2 \,\Bigr\}.
\]
Then $\mathcal{L}$ is not symmetric on this domain: there are $f,g\in\operatorname{dom}(\mathcal{L})$ with $\langle\mathcal{L}f,g\rangle\ne\langle f,\mathcal{L}g\rangle$. In particular, no pair of conditions of the stated form containing a nonzero bulk term yields a self-adjoint operator. If an ordinary Green's function exists, it cannot be a covariance kernel.
\end{proposition}

\begin{proof}
The following boundary-form argument shows why a bulk condition cannot replace one of the two endpoint conditions.

As in Proposition~\ref{prop:selfadj}, $\langle\mathcal{L}f,g\rangle-\langle f,\mathcal{L}g\rangle = [fg'-f'g]_0^1$, and the right-hand side depends on $f$ and~$g$ only through $\Gamma f,\Gamma g$. Write $J$ for the induced bilinear form on~$\mathbb{R}^4$; in the coordinates $\Gamma f = (p,q,r,u)$,
\[
 J\bigl(\Gamma f,\Gamma g\bigr) = \bigl(r u' - u r'\bigr) - \bigl(p q' - q p'\bigr) ,
\]
which is skew-symmetric and nondegenerate. A nondegenerate skew form on $\mathbb{R}^4$ cannot vanish identically on a subspace of dimension greater than~$2$. It therefore suffices to show that the image $\Gamma\bigl(\operatorname{dom}(\mathcal{L})\bigr)$, which is a subspace of~$\mathbb{R}^4$, has dimension at least~$3$.

Interior perturbations leave the endpoint data unchanged: every $\psi\in C_c^\infty(0,1)$ has $\Gamma\psi=0$. Replacing $f$ by $f+\psi$ therefore changes the integral terms without changing $\Gamma f$. There are two cases. If $w_1,w_2$ are linearly independent, then $\psi\mapsto\bigl(\langle w_1,\psi\rangle,\langle w_2,\psi\rangle\bigr)$ maps $C_c^\infty(0,1)$ onto~$\mathbb{R}^2$, because $C_c^\infty(0,1)$ is dense in $L^2[0,1]$; thus, for any $\gamma\in\mathbb{R}^4$, choose $f_0\in\mathcal{D}$ with $\Gamma f_0=\gamma$, then choose $\psi$ to satisfy both conditions. Hence $\Gamma(\operatorname{dom}(\mathcal{L}))=\mathbb{R}^4$. Otherwise, after relabeling, $w_2=\mu w_1$ with $w_1\ne 0$; then $\Lambda_2-\mu\Lambda_1$ is a condition on endpoint data alone, one scalar correction suffices, and $\Gamma(\operatorname{dom}(\mathcal{L}))$ is the entire kernel of $\Lambda_2-\mu\Lambda_1$ in~$\mathbb{R}^4$, of dimension at least~$3$. In both cases the dimension exceeds~$2$, so $J$ does not vanish on it.
\end{proof}

The next example has a unique Green's function, but it is not symmetric.

\begin{example}[An integral condition with a nonsymmetric Green's function]\label{ex:bulk}
Pair the condition $f(0)=\int_0^1 f(s)\,\dd s$ with a Dirichlet condition at the right end, $f(1)=0$. To check uniqueness, write a solution of $\mathcal{L}u=0$ as $u(s)=A+Bs$. The condition $u(1)=0$ forces $B=-A$, and the integral condition then gives $A = \int_0^1 A(1-s)\,\dd s = A/2$. Thus $A=B=0$. The null space is trivial, the Green's function exists and is unique, and solving exactly as in Theorem~\ref{thm:robin} gives
\begin{equation}\label{eq:bulk_green}
 G(s,t) = \bigl(1-s\vee t\bigr)\,\bigl[\,t + (1-t)\,(s\wedge t)\,\bigr] .
\end{equation}
One checks directly that $\mathcal{L}_s\,G(s,t)=\delta(s-t)$, that $G(1,t)=0$, and that $G(0,t)=t(1-t)=\int_0^1 G(s,t)\,\dd s$. But~\eqref{eq:bulk_green} is not symmetric. Its antisymmetric part is
\begin{equation}\label{eq:bulk_asym}
 G(s,t) - G(t,s) = (1-s)(1-t)(t-s) ,
\end{equation}
which is nonzero whenever $s\ne t$ and $s,t<1$, and attains its maximum $\tfrac14$ at $(s,t)=(0,\tfrac12)$. A boundary-value problem can therefore have a unique Green's function that is not a covariance function. The failure of symmetry also has an explicit witness: $f(s)=s(1-s)(1-2s)$ and $g(s)=(1-s)(1+3s)$ both lie in the domain, since $f(0)=\int_0^1 f = 0$, $g(0)=\int_0^1 g = 1$ and $f(1)=g(1)=0$, and
\[
 \langle\mathcal{L}f,g\rangle - \langle f,\mathcal{L}g\rangle = f'(0)\,g(0) - f(0)\,g'(0) = 1 .
\]
\end{example}

The asymmetry is a property of the operator, not an artifact of any approximation: \eqref{eq:bulk_asym} is an identity, and it is nonzero wherever $s\ne t$ and $s,t<1$. Symmetrizing does not rescue the kernel either. Since $G(s,0)=0$ for every~$s$ but $G(0,t)=t(1-t)$, the symmetric part $\tfrac12\{G(s,t)+G(t,s)\}$ vanishes at $(0,0)$ yet equals $\tfrac12t(1-t)>0$ for $0<t<1$ along that row; a positive semidefinite kernel satisfies $K(s,t)^2\le K(s,s)K(t,t)$, so a zero diagonal entry forces its whole row to vanish. The $2\times2$ principal minor at $(0,\tfrac{1}{2})$ is $-\tfrac{1}{64}$.

\begin{remark}[The other pairings]\label{rem:bulk_other}
Pairing~\eqref{eq:bulk} with a Neumann condition instead, either $f'(1)=0$ or $f'(0)=0$, fails even earlier. The constant function~$\mathbf{1}$ satisfies~\eqref{eq:bulk}, because $\int_0^1\mathbf{1}=1$, and it satisfies either Neumann condition; hence $\mathcal{L}\mathbf{1}=0$, the operator has a nontrivial null space, and there is no ordinary Green's function. Proposition~\ref{prop:bulk} covers all three pairings at once.
\end{remark}

\section{Degenerate operators: null spaces and modified equations}\label{sec:degenerate}

In both the Neumann--Neumann case (Section~\ref{sec:NN_obstruction}) and the periodic case (Section~\ref{sec:nonlocal}), $\mathcal{L}=-\dd^2/\dd s^2$ annihilates constants. Projecting away this shared null space resolves both cases.

\subsection{Two mean-centered processes}\label{sec:centred_constructions}

We first construct the two processes. Both come from operation~2 in Section~\ref{sec:two_mechanisms}---recentering a process already in hand---and their covariances follow directly from~\eqref{eq:recentring}. Write $\bar{X}=\int_0^1X_u\,\dd u$ for the time average.

\begin{proposition}[The two mean-centered beasts]\label{prop:centred_constructions}
Let $W$ be standard Brownian motion and $B$ a Brownian bridge. Then
\begin{equation}\label{eq:meancentred_bm}
 \Cov\bigl(W_s-\bar{W},\,W_t-\bar{W}\bigr) = \tfrac12\bigl(s^2+t^2\bigr) - s\vee t + \tfrac13 ,
\end{equation}
\begin{equation}\label{eq:meancentred_bridge}
 \Cov\bigl(B_s-\bar{B},\,B_t-\bar{B}\bigr) = \tfrac12|s-t|^2 - \tfrac12|s-t| + \tfrac1{12} .
\end{equation}
\end{proposition}

\begin{proof}
For~\eqref{eq:meancentred_bm}, $K_\ell(s)=\Cov(W_s,\bar W)=\int_0^1(s\wedge u)\,\dd u = s-\tfrac12s^2$ and $V_\ell=\Var(\bar W)=\tfrac13$, so~\eqref{eq:recentring} gives $s\wedge t-\bigl(s-\tfrac12s^2\bigr)-\bigl(t-\tfrac12t^2\bigr)+\tfrac13$, which is~\eqref{eq:meancentred_bm} because $s\wedge t-s-t=-s\vee t$. For~\eqref{eq:meancentred_bridge}, $\Cov(B_s,\bar B)=\tfrac12(s-s^2)$ and $\Var(\bar B)=\tfrac13-\tfrac14=\tfrac1{12}$, and the same computation gives the stated kernel.
\end{proof}

We call these the \emph{mean-centered} Brownian motion and the \emph{mean-centered} bridge. The rest of the section identifies their covariances as the normalized kernels of the two degenerate operators: \eqref{eq:meancentred_bm} is the Neumann--Neumann kernel and~\eqref{eq:meancentred_bridge} the periodic one. This operator route needs a corrected source and a normalization; the direct pathwise constructions do not. The processes are already in hand, but the operator route explains why these particular kernels appear.

\paragraph{Zero average does not specify a law.}
The following example illustrates the distinction drawn in Section~\ref{sec:two_mechanisms}. Recentering and conditioning can use the same path and the same functional yet produce different laws. Take $\ell(X)=\bar X$ and set
\begin{equation}\label{eq:C_vs_R}
 C_s = W_s - \bar{W} \qquad\text{(recentered)}, \qquad\qquad R_s = W_s - 3\Bigl(s-\tfrac12s^2\Bigr)\bar{W} \qquad\text{(conditioned)} ,
\end{equation}
where $R$ is~\eqref{eq:conditioning} with $K_\ell(s)=s-\tfrac12s^2$ and $V_\ell=\tfrac13$. Both satisfy $\bar C=\bar R=0$ almost surely, so both are \emph{balanced}. Equivalently, their kernels satisfy $\int_0^1G(s,t)\,\dd s=0$, but this property does not determine a law. The difference is already visible at the left endpoint:
\[
 C_0 = -\bar{W} , \qquad\text{with } \Var(C_0)=\tfrac13 , \qquad\qquad R_0 = 0 \ \text{ almost surely} .
\]
Conditioning leaves $W_0=0$ fixed because $K_\ell(0)=0$. Recentering shifts the whole path by a random constant, giving the value at the origin a positive variance. We call $R$ the \emph{integral-conditioned} Brownian motion and reserve \emph{mean-centered} for~$C$. Section~\ref{sec:bc_vs_constraints} develops the general comparison. See Figure~\ref{fig:centring_vs_conditioning}, which applies both maps to a single Brownian path.

\begin{figure}[!t]
 \centering
 \includegraphics[width=\textwidth]{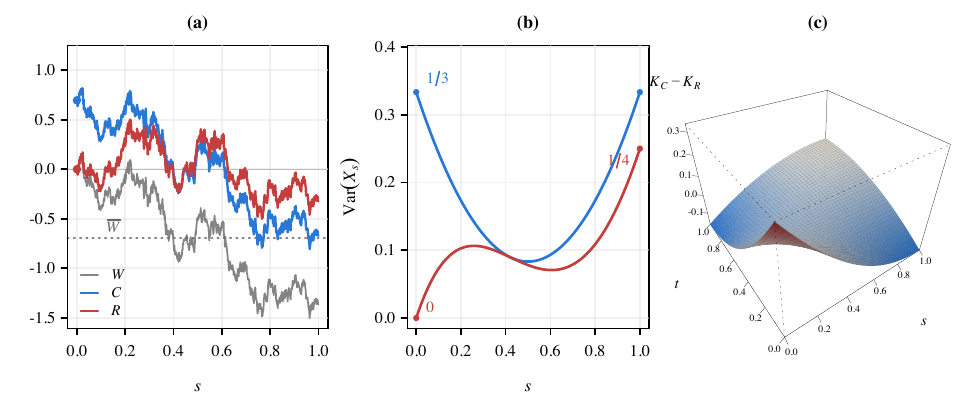}
 \caption{Recentering versus conditioning on one Brownian path. \emph{(a)} The path $W$ (gray, with its time average $\bar W$ dotted) and the transforms $C$ and $R$ in~\eqref{eq:C_vs_R}. Both have zero time average, but their marked starting values differ: $C_0=-\bar W$ has variance $\tfrac13$, whereas $R_0=0$ almost surely. \emph{(b)} Variance functions. Recentering gives variance $\tfrac13$ at both endpoints; conditioning leaves $s=0$ pinned and gives variance $\tfrac14$ at $s=1$. Recentering has at least as much variance at every time. \emph{(c)} The covariance difference $K_C-K_R$ is a positive-semidefinite rank-one kernel, although its off-diagonal entries change sign. The process $C$ is the Neumann--Neumann process derived below; $R$ comes from the rank-one conditioning formula~\eqref{eq:conditioning}.}
 \label{fig:centring_vs_conditioning}
\end{figure}

\subsection{The solvability obstruction}\label{sec:solvability}

An operator with nontrivial null space has no ordinary inverse. The Fredholm alternative explains the obstruction: for self-adjoint $\mathcal{L}$, the equation $\mathcal{L}u=h$ can have a solution only if the source~$h$ is orthogonal to the null space. We state this precisely for our setting.

\begin{lemma}[Fredholm solvability condition]\label{lem:fredholm}
Let $\mathcal{L}$ be self-adjoint on its domain, with closed range and null space $\mathcal{N}=\ker(\mathcal{L})=\operatorname{span}\{\varphi\}$. Then $\mathcal{L}u=h$ is solvable if and only if $\langle h,\varphi\rangle = 0$. For $\mathcal{L}=-\dd^2/\dd s^2$ with $\mathcal{N}=\operatorname{span}\{\mathbf{1}\}$, the same compatibility test applies distributionally to the Green's-function equation. For $t\in(0,1)$, the uncorrected equation $\mathcal{L}_s\,G(s,t)=\delta(s-t)$ fails that test:
\[
 \langle \delta(\cdot-t),\, \mathbf{1}\rangle = \int_0^1 \delta(s-t)\cdot 1\,\dd s = 1 \ne 0.
\]
The source $\delta(\cdot-t)$ is not orthogonal to the null space, so no solution exists.
\end{lemma}

The operator cannot ``see'' constant shifts because they lie in its null space. By self-adjointness, its range is orthogonal to those same constants, so a source with nonzero constant component lies out of reach.

\subsection{Projecting the source}

The remedy is to project the source onto~$\mathcal{N}^\perp$ before solving. With $\varphi=\mathbf{1}$, the constant function with value~$1$ on $[0,1]$ (already $L^2$-normalized, with $\|\mathbf{1}\|^2=1$), the projection of $\delta(\cdot-t)$ onto the constants is $\langle\delta(\cdot-t),\mathbf{1}\rangle\,\mathbf{1} = \mathbf{1}$, so the corrected equation is
\begin{equation}\label{eq:corrected}
 \mathcal{L}_s\,G(s,t) = \delta(s-t) - 1 .
\end{equation}
The right-hand side now satisfies the Fredholm condition, $\int_0^1[\delta(s-t)-1]\,\dd s=1-1=0$, so a solution exists. For each fixed interior~$t$, it is unique up to an additive constant in~$\mathcal{N}$. We select it by imposing the \emph{balanced condition}
\begin{equation}\label{eq:zeromean}
 \int_0^1 G(s,t)\,\dd s = 0,
\end{equation}
which selects the representative in~$\mathcal{N}^\perp$. For the resulting zero-mean Gaussian process, \eqref{eq:zeromean} gives $\int_0^1 X_s\,\dd s=0$ almost surely. The process lies in the zero-average subspace on which the degenerate operator is invertible.

The balanced normalization fixes both the source correction and the additive constant. First, the constant weight~$\mathbf{1}$ used to correct the source in~\eqref{eq:corrected} is the only choice compatible with~\eqref{eq:zeromean}. Replacing it by an arbitrary weight $w$ with $\int_0^1 w = 1$ keeps the Fredholm condition intact, and, in the Neumann--Neumann case below, imposing symmetry gives a kernel of the form $V(s)+V(t)-s\vee t+k$, where $V''=w$ and $V(0)=V'(0)=0$; such a kernel has a row integral $\int_0^1 G(s,t)\,\dd s$ independent of~$t$ only when $V(t)-t^2/2$ is constant, that is, only when $w\equiv\mathbf{1}$. Second, once $w\equiv\mathbf{1}$, condition~\eqref{eq:zeromean} selects $k=\tfrac13$. Proposition~\ref{prop:k_range} shows that this is exactly the positive-semidefinite threshold within the constant-shift family. At this value the balanced kernel annihilates the constant function. At the kernel level, subtracting row and column averages removes the separable term $V(s)+V(t)$. All these kernels therefore have the same balanced part; this algebraic statement does not require the unbalanced kernels to be covariances.

\begin{remark}[The pseudoinverse interpretation]\label{rem:pseudoinverse}
Projecting the source and selecting the solution in~$\mathcal{N}^\perp$ defines the Moore--Penrose pseudoinverse $\mathcal{L}^\dagger$. The corrected source~\eqref{eq:corrected} and balanced condition~\eqref{eq:zeromean} are the two parts of this construction: the first removes the source's constant component, and the second selects the solution orthogonal to constants. The kernels $G_\NN$ and $G_{\mathrm{per}}$ are the kernels of~$\mathcal{L}^\dagger$. This also explains why Proposition~\ref{prop:bulk} does not rule out the balanced condition. It does not replace either endpoint condition in the domain of the original operator. Those two conditions already make the operator self-adjoint; only invertibility fails. Balance selects the pseudoinverse solution in~$\mathcal{N}^\perp$, rather than replacing an endpoint condition. The distinction is between normalizing a solution and defining a different operator domain.
\end{remark}

\begin{remark}[Polynomial terms in Green's functions]\label{rem:polynomial}
Away from $s=t$, equation \eqref{eq:corrected} reads $\mathcal{L}_s\,G(s,t) = -1$ rather than $\mathcal{L}_s\,G(s,t)=0$; since $\mathcal{L}=-\dd^2/\dd s^2$, that is $\partial_s^2 G(s,t) = 1$. Integrating twice produces a $\tfrac12s^2$ term. This connection between the source correction and polynomial covariance structure is analogous to the appearance of polynomial ``unpenalized'' terms in smoothing-spline theory~\cite{Wahba1990,Raket2021}.
\end{remark}

\subsection{The periodic process}

\needspace{5\baselineskip}
\begin{theorem}[Periodic Green's function]\label{thm:periodic}
For the periodic conditions $f(0)=f(1)$, $f'(0)=f'(1)$, equations~\eqref{eq:corrected}--\eqref{eq:zeromean} give
\begin{equation}\label{eq:periodic}
 G_{\mathrm{per}}(s,t) = \tfrac{1}{2}\,|s-t|^2 - \tfrac{1}{2}\,|s-t| + \tfrac{1}{12} = \tfrac{1}{2}B_2(|s-t|),
\end{equation}
where $B_2(x) = x^2 - x + \tfrac{1}{6}$ is the second Bernoulli polynomial.
\end{theorem}

\begin{proof}
Fix $t\in(0,1)$. Since $\mathcal{L}_s\,G(s,t)=-1$ off the diagonal, write $G(s,t) = \tfrac{1}{2}s^2 + a_1 s + c_1$ for $s<t$ and $G(s,t) = \tfrac{1}{2}s^2 + a_2 s + c_2$ for $s>t$. Matching~\eqref{eq:matching} gives $a_1 t + c_1 = a_2 t + c_2$ and $a_2 - a_1 = -1$. Periodicity of the values gives $c_1=\tfrac12+a_2+c_2$. Eliminating the coefficients yields $a_1 = \tfrac{1}{2}-t$, $a_2 = -\tfrac{1}{2}-t$, and $c_1-c_2=-t$. The balanced condition~\eqref{eq:zeromean} fixes $c_2 = \frac{t(t+1)}{2}+\frac{1}{12}$, $c_1 = c_2 - t$. For $s<t$,
\[
 G = \tfrac{1}{2}s^2 + (\tfrac{1}{2}-t)s + \tfrac{t^2}{2} - \tfrac{t}{2} + \tfrac{1}{12} = \tfrac{1}{2}(s-t)^2 + \tfrac{1}{2}(s-t) + \tfrac{1}{12},
\]
with the symmetric expression for $s>t$. Together these give~\eqref{eq:periodic}, which agrees with the mean-centered bridge covariance~\eqref{eq:meancentred_bridge}.
\end{proof}

The periodic process is stationary on the circle, with variance $\tfrac1{12}$ and $\Corr(X_0,X_1)=1$. Figure~\ref{fig:detail_periodic} shows its covariance surface and sample paths. The covariance~\eqref{eq:periodic} is invariant under $|s-t|\mapsto1-|s-t|$, as required for a real stationary covariance on the circle.

\begin{figure}[tbp]
 \centering
 \includegraphics[width=0.85\textwidth]{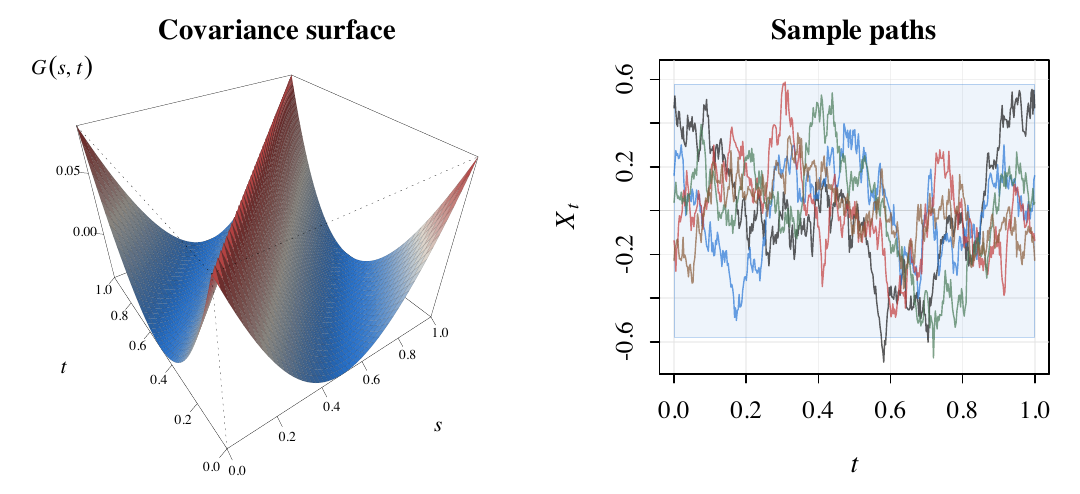}
 \caption{Periodic process: covariance surface $G(s,t)$ on $[0,1]^2$ (left) and five sample paths with a $\pm2$-standard-deviation envelope (right).}
 \label{fig:detail_periodic}
\end{figure}

\begin{proposition}[Derivative periodicity follows from the corrected equation]\label{prop:deriv_periodic}
For each interior source point $t\in(0,1)$, the corrected equation~\eqref{eq:corrected} implies $\partial_sG(0,t)=\partial_sG(1,t)$. Value periodicity and the balanced condition~\eqref{eq:zeromean} then determine the remaining coefficients.
\end{proposition}

\begin{proof}
The corrected equation already implies derivative periodicity:
\[
 \partial_s G(1,t)-\partial_s G(0,t) = \int_0^1\bigl[1-\delta(s-t)\bigr]\,\dd s = 0 .
\]
Value periodicity and the balanced condition determine the remaining coefficients and select a unique kernel; they are not needed for the displayed implication.
\end{proof}

Value periodicity alone still does not define a self-adjoint realization: one condition cannot reduce the endpoint-trace space $\mathbb{R}^4$ to a Lagrangian plane (Proposition~\ref{prop:bulk}). Balance normalizes the corrected solution; it does not replace a defining endpoint condition. It selects the representative orthogonal to the null space, not one within it.

\begin{example}[Value periodicity and balance alone are not enough]\label{ex:periodic_incomplete}
Replacing derivative periodicity by the balanced condition~\eqref{eq:zeromean} changes the operator. Solve $\mathcal{L}_s\,G(s,t)=\delta(s-t)$ directly on the domain $\{f: f(0)=f(1),\ \int_0^1 f = 0\}$. Value periodicity forces any affine null function to be constant, and balance then forces it to be zero. The resulting operator is non-degenerate, so no source correction is needed. Continuity, the derivative jump and the two side conditions give
\[
 G(s,t) = (1-t)s + \frac{t(t-1)}{2} \quad (s<t), \qquad G(s,t) = -ts + \frac{t(t+1)}{2} \quad (s>t).
\]
This kernel is not symmetric: for $s>t$,
\begin{equation}\label{eq:periodic_incomplete_asym}
 G(s,t)-G(t,s) = \frac{(t-s)(s+t-1)}{2} ,
\end{equation}
which vanishes only on the diagonal $s=t$ and the anti-diagonal $s+t=1$. In particular, $G(0.2,0.8)=G(0.8,0.2)$, so testing that pair alone would miss the asymmetry. The difference is nonzero elsewhere; it attains $-\tfrac18$ at $(s,t)=(1,\tfrac12)$ and $+\tfrac18$ at $(s,t)=(\tfrac12,0)$, both of maximal absolute value. The failure is exactly the one Proposition~\ref{prop:bulk} predicts: value periodicity alone leaves a three-dimensional image of~$\Gamma$, and the balanced condition, being a bulk condition, does not reduce it further.
\end{example}

\subsection{The Neumann--Neumann process}

\needspace{5\baselineskip}
\begin{theorem}[Neumann--Neumann Green's function]\label{thm:NN}
For $f'(0)=f'(1)=0$, equations~\eqref{eq:corrected}--\eqref{eq:zeromean} give
\begin{equation}\label{eq:NN}
 G_\NN(s,t) = \frac{s^2+t^2}{2} - s\vee t + \frac{1}{3}, \qquad \Var(X_s) = s^2-s+\frac{1}{3}.
\end{equation}
\end{theorem}

\begin{proof}
Fix $t\in(0,1)$. As in Theorem~\ref{thm:periodic}, write $G=\tfrac12s^2+a_i s+c_i$ off the diagonal. The Neumann conditions give $a_1=\partial_sG(0,t)=0$ and $a_2=\partial_sG(1,t)-1=-1$, consistent with the required derivative jump. Continuity gives $c_1-c_2 = -t$, and the balanced condition fixes $c_2 = \frac{t^2}{2}+\frac{1}{3}$, $c_1 = c_2-t$. For $s<t$, $G = \tfrac{1}{2}s^2 + c_1 = \tfrac{1}{2}(s^2+t^2) - t + \frac{1}{3}$, and symmetrically for $s>t$; both equal~\eqref{eq:NN}. This agrees with~\eqref{eq:meancentred_bm}: the corrected source and balanced condition have led back to the mean-centered Brownian motion of Proposition~\ref{prop:centred_constructions}.
\end{proof}

Unlike the periodic process, the Neumann--Neumann process is \emph{not} stationary. Its variance $s^2-s+\tfrac13$ has minimum $\tfrac1{12}$ at $s=\tfrac12$ and equals $\tfrac13$ at both endpoints. The endpoints are negatively correlated:
\[
 G_\NN(0,1) = \frac{0+1}{2}-1+\frac{1}{3} = -\frac{1}{6}, \qquad \Corr(X_0,X_1) = \frac{-1/6}{1/3} = -\frac{1}{2}.
\]
Neither endpoint is pinned, and subtracting the path's time average produces the negative endpoint correlation. See Figure~\ref{fig:detail_nn} for the covariance surface and sample paths, and Figure~\ref{fig:degenerate} for the comparison with the periodic case.

\begin{figure}[tbp]
 \centering
 \includegraphics[width=0.85\textwidth]{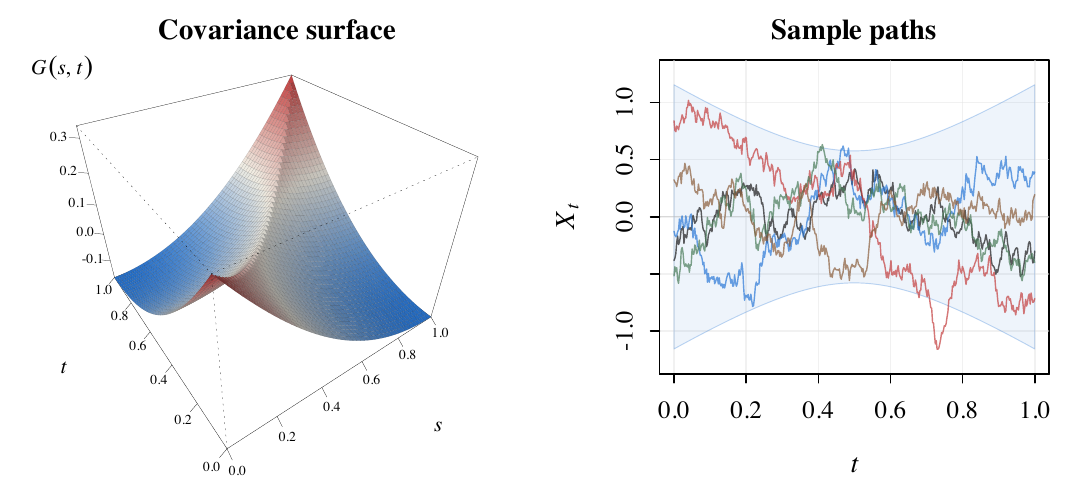}
 \caption{Neumann--Neumann process: covariance surface $G(s,t)$ on $[0,1]^2$ (left) and five sample paths with a $\pm2$-standard-deviation envelope (right).}
 \label{fig:detail_nn}
\end{figure}

\begin{figure}[tbp]
 \centering
 \includegraphics[width=\textwidth]{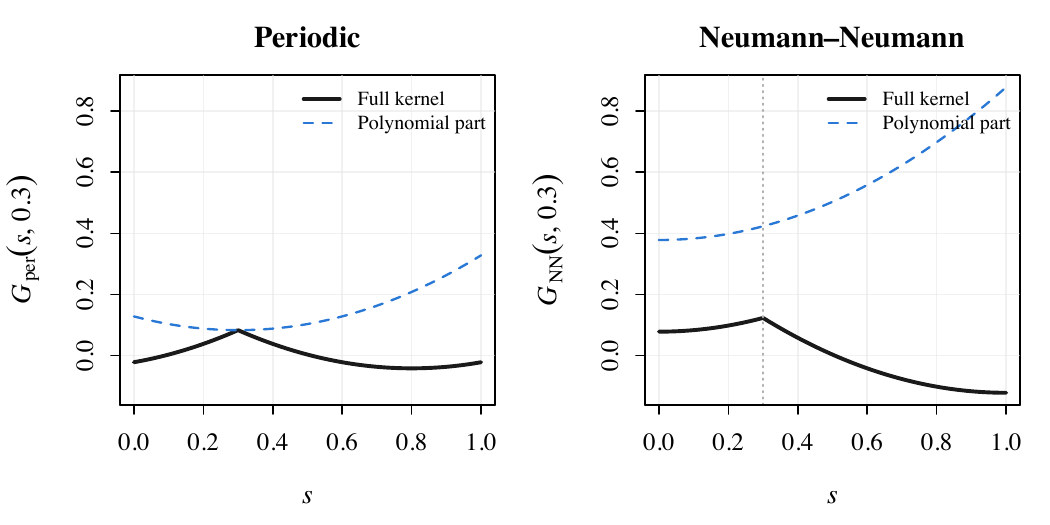}
 \caption{Covariance structure for the degenerate operators: periodic (left) and Neumann--Neumann (right). At $t=0.3$, each panel compares the full kernel (solid) with its polynomial part (dashed): $(s-t)^2/2+1/12$ for the periodic process and $(s^2+t^2)/2+1/3$ for Neumann--Neumann. The full kernels also contain $-|s-t|/2$ and $-\max(s,t)$, respectively. The quadratic terms come from the corrected source $\delta(s-t)-1$. The periodic process is stationary, with variance $1/12$; the Neumann--Neumann variance is parabolic.}
 \label{fig:degenerate}
\end{figure}

\subsection{The additive constant is not free}\label{sec:constant_not_free}

Equation~\eqref{eq:NN} belongs to the family $G_k(s,t)=\tfrac12(s^2+t^2)-s\vee t+k$, but $k$ is not an unrestricted normalization constant. Positive semidefiniteness imposes a lower bound, and balance selects the threshold value.

\begin{proposition}[The additive constant]\label{prop:k_range}
Let $G_k(s,t) = \tfrac{1}{2}(s^2+t^2) - s\vee t + k$.
\begin{enumerate}
\item[\emph{(a)}] $G_k$ is a valid covariance function (positive semidefinite) if and only if $k\ge\tfrac{1}{3}$.
\item[\emph{(b)}] The balanced condition~\eqref{eq:zeromean} forces $k=\tfrac13$, the smallest constant for which $G_k$ is positive semidefinite.
\item[\emph{(c)}] Pointwise pinning is not an alternative within this constant-shift family: no single~$k$ can make $G_k(s_0,t)=0$ for every~$t$, because it cannot cancel the nonconstant function $\tfrac12(s_0^2+t^2)-s_0\vee t$.
\end{enumerate}
\end{proposition}

\begin{proof}
\emph{(a)--(b).} Write $G_k = G_{1/3} + (k-\tfrac{1}{3})\cdot\mathbf{1}\otimes\mathbf{1}$, where $\mathbf{1}\otimes\mathbf{1}$ denotes the rank-one kernel with constant value~$1$. The kernel $G_{1/3} = G_\NN$ satisfies the balanced condition~\eqref{eq:zeromean}, i.e.\ $\int_0^1 G_{1/3}(s,t)\,\dd s = 0$ for all~$t$. Thus the constant function is an eigenfunction of the integral operator with kernel~$G_{1/3}$ with eigenvalue~$0$. The remaining eigenfunctions are $\cos(n\pi s)$, $n\ge 1$, with strictly positive eigenvalues $1/(n\pi)^2$. Adding $(k-\tfrac{1}{3})\cdot\mathbf{1}\otimes\mathbf{1}$ shifts the eigenvalue on constants from $0$ to $k-\tfrac{1}{3}$ and leaves all others unchanged. Therefore $G_k$ is positive semidefinite if and only if $k-\tfrac{1}{3}\ge 0$, and the balanced condition selects the smallest admissible constant, $k=\tfrac13$.

\emph{(c).} The expression $G_k(s_0,t) = \tfrac{1}{2}s_0^2+\tfrac{1}{2}t^2-s_0\vee t + k$ depends nontrivially on~$t$ (through the terms $\tfrac{1}{2}t^2$ and $s_0\vee t$). Setting this to zero for all~$t$ simultaneously would require $k$ to absorb a $t$-dependent quantity, which a single constant cannot do.
\end{proof}

\begin{remark}[Why $k=\tfrac13$ is the positivity threshold]\label{rem:extremality}
The value $k=\tfrac13$ is the positivity threshold within this constant-shift family, not an arbitrary normalization. Balance selects a solution in $\mathcal{N}^\perp$, setting the covariance eigenvalue on the constant mode to zero. This is exactly the Moore--Penrose construction: $\mathcal{L}^\dagger$ maps the null space to zero and gives the minimum-norm solution for a solvable source. See Figure~\ref{fig:psd_boundary}. The constant function shows that the bound is sharp:
\[
 \langle\mathbf{1},G_k\mathbf{1}\rangle \;=\; \int_0^1\!\!\int_0^1 G_k(s,t)\,\dd s\,\dd t \;=\; k-\tfrac{1}{3},
\]
which is negative for every $k<\tfrac13$. Thus $\mathbf{1}$ directly witnesses the failure of positive semidefiniteness. At $k=\tfrac13$, the kernel is the covariance of mean-centered Brownian motion. The failing direction below the threshold is precisely the null direction of~$\mathcal{L}$. Finite grids require some care. On an equispaced grid with $n+1$ nodes and fixed $k<\tfrac13$, the smallest eigenvalue of the sampled kernel matrix diverges to $-\infty$, with $\lambda_{\min}\approx(n+1)(k-\tfrac13)$. It tracks the negative constant mode rather than tending to zero. The sampled matrix is positive semidefinite exactly for $k\ge\tfrac13-\tfrac1{12n^2}$ on this grid, and at $k=\tfrac13$ itself it is strictly positive definite, with its smallest eigenvalue tending to zero as $n\to\infty$. Grid refinement approaches the continuous threshold; a finite grid does not locate it exactly.
\end{remark}

\begin{figure}[tbp]
 \centering
 \includegraphics[width=0.8\textwidth]{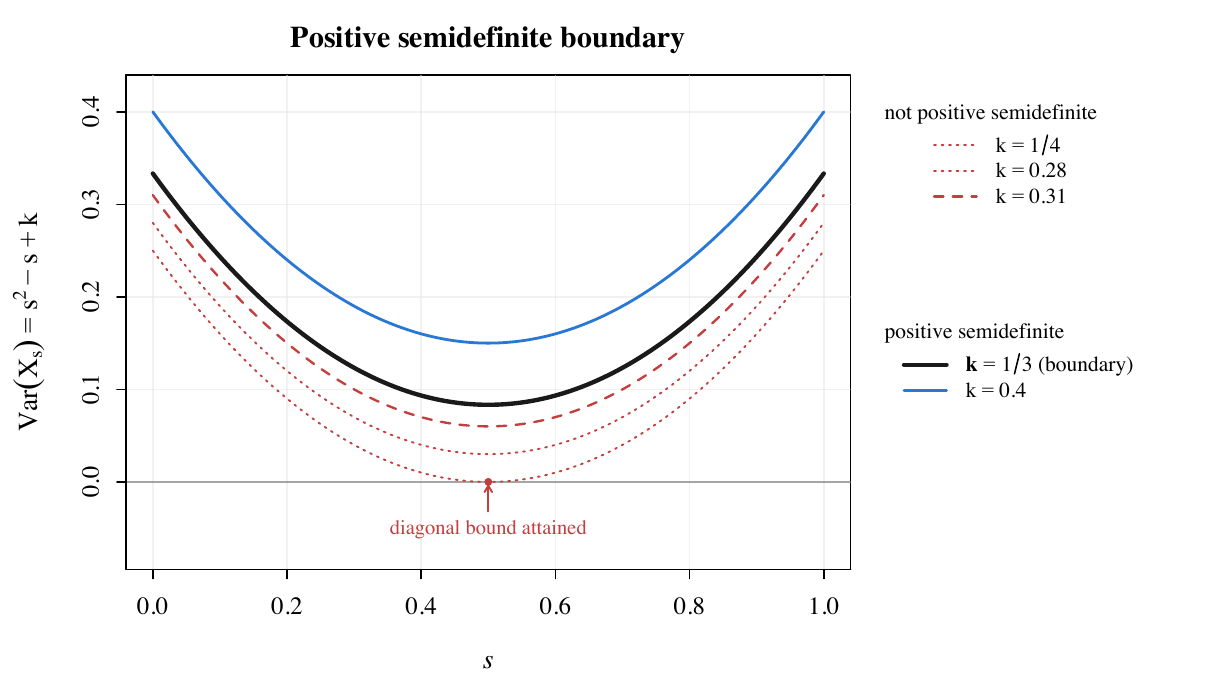}
 \caption{Diagonal profiles $G_k(s,s)=s^2-s+k$ at $k=\tfrac14,\,0.28,\,0.31,\,\tfrac13,\,0.4$ for the Neumann--Neumann constant-shift family. Nonnegative variance is necessary but not sufficient: all five curves stay at or above zero, but only $k=\tfrac13$ and $k=0.4$ give covariance kernels in this family (Proposition~\ref{prop:k_range}). The lowest curve attains the diagonal bound $k\ge\tfrac14$, touching zero at $s=\tfrac12$.}
 \label{fig:psd_boundary}
\end{figure}

Normalizing an inverse and constraining a process are different operations. Here balance selects the pseudoinverse on $\mathcal{N}^\perp$. A constant shift cannot instead pin the process at a point, as Proposition~\ref{prop:k_range}(c) shows. Whether a path constraint changes a given process is a separate question, taken up below under detectability.

\begin{remark}[A constant shift can change the operator]\label{rem:slepian_constant}
For either degenerate operator just studied, adding a constant to the balanced kernel uses the freedom left by the null space. The corrected equation allows this freedom; balance and positive semidefiniteness restrict it. Example~\ref{ex:slepian} gives the contrasting case. The anti-periodic operator is non-degenerate, so $G_{-1}+v$ cannot be a family of Green's functions for \emph{one} operator. Instead, varying~$v$ changes the pair $(\mathcal{L},\mathcal{B})$, whose boundary conditions are
\begin{equation}\label{eq:slepian_family}
 f'(0)+f'(1)=0, \qquad f(0)+f(1) = 4v\,f'(0) ,
\end{equation}
giving a self-adjoint operator for every $v\in\mathbb{R}$ by the calculation in Example~\ref{ex:slepian}. The covariance family runs from the anti-periodic process at $v=0$ through Slepian Brownian motion at $v=\tfrac14$. It has no finite kernel limit as $v\to\infty$. Dividing the second condition by $4v$ and letting $v\to\infty$ gives $f'(0)=0$, hence $f'(1)=0$ by the first: the boundary conditions degenerate to Neumann--Neumann. The limiting pair is degenerate and has no ordinary Green's function; correspondingly, $G_{-1}+v$ diverges. As in Proposition~\ref{prop:robin_blowup}, the boundary conditions approach a degenerate pair while the kernels diverge.

Positive semidefiniteness holds exactly for $v\ge0$. The constant direction gives $\langle\mathbf{1},(G_{-1}+v)\mathbf{1}\rangle=\tfrac1{12}+v$, which stays nonnegative down to $v=-\tfrac1{12}$, below the true threshold. Unlike in the degenerate case, testing the constant direction is therefore necessary but not sufficient. The obstruction is that $\mathbf{1}$ lies outside the reproducing kernel Hilbert space of $G_{-1}$. The next remark confirms the threshold directly from the energy form.
\end{remark}

\begin{remark}[The same threshold from the energy form]\label{rem:slepian_energy}
The energy form locates the threshold that the constant-direction test misses. Substituting $f'(1)=-f'(0)$ and $f(0)+f(1)=4v\,f'(0)$ from~\eqref{eq:slepian_family} into the integration-by-parts identity of Remark~\ref{rem:alpha_positive} gives $f'(0)f(0)-f'(1)f(1) = f'(0)\bigl(f(0)+f(1)\bigr) = 4v\,f'(0)^2$, so
\begin{equation}\label{eq:slepian_energy}
\langle f,\mathcal{L}f\rangle = \int_0^1 f'(s)^2\,\dd s + 4v\,f'(0)^2 .
\end{equation}
For $v\ge0$, both terms are nonnegative. If the form vanishes, then $f$ is constant, and the domain condition gives $2f(0)=4v\,f'(0)=0$, so $f\equiv0$. Thus the admissible range includes the endpoint $v=0$. For $v\ne0$, eliminating $f'(0)=\bigl(f(0)+f(1)\bigr)/(4v)$ turns~\eqref{eq:slepian_energy} into the closed $H^1$ form $q_v(f) = \int_0^1 f'(s)^2\,\dd s + \bigl(f(0)+f(1)\bigr)^2/(4v)$, continuous on $H^1[0,1]$ and requiring no pointwise trace of $f'$. At $v=0$, the boundary condition itself imposes $f(0)+f(1)=0$. For $v<0$, we can exhibit a negative-energy function in $\operatorname{dom}(\mathcal{L})\subset H^2$. A rectangular spike in $f'$ would be discontinuous, but a linear taper has the required regularity. Choose $h\ne0$ and $0<\varepsilon<\tfrac12$. Let $f'$ decrease linearly from $h$ at $0$ to $0$ at $\varepsilon$, remain zero until $1-\varepsilon$, then decrease linearly to $-h$ at $1$. Set $f(0)=2vh$ to satisfy the remaining boundary condition. Then $\int_0^1f'(s)^2\,\dd s=2\varepsilon h^2/3$ and $\langle f,\mathcal{L}f\rangle=h^2(2\varepsilon/3+4v)$, which is negative when also $\varepsilon<-6v$. The same distinction applies to the Robin energy~\eqref{eq:energy}: the derivative seminorm alone does not control its boundary terms. Writing the form in terms of endpoint values, as in $q_v$, avoids point evaluations of the derivative.
\end{remark}

\section{Conditioning: constraints on a fixed process}\label{sec:conditioning}

We now turn to operation~3: conditioning. The rank-one identity~\eqref{eq:conditioning}, established in Section~\ref{sec:recentring_conditioning}, generates the examples below; Figure~\ref{fig:conditioning} illustrates the mechanism. Several processes from Sections~\ref{sec:robin}--\ref{sec:degenerate} reappear along this route.

\begin{figure}[tbp]
 \centering
 \includegraphics[width=\textwidth]{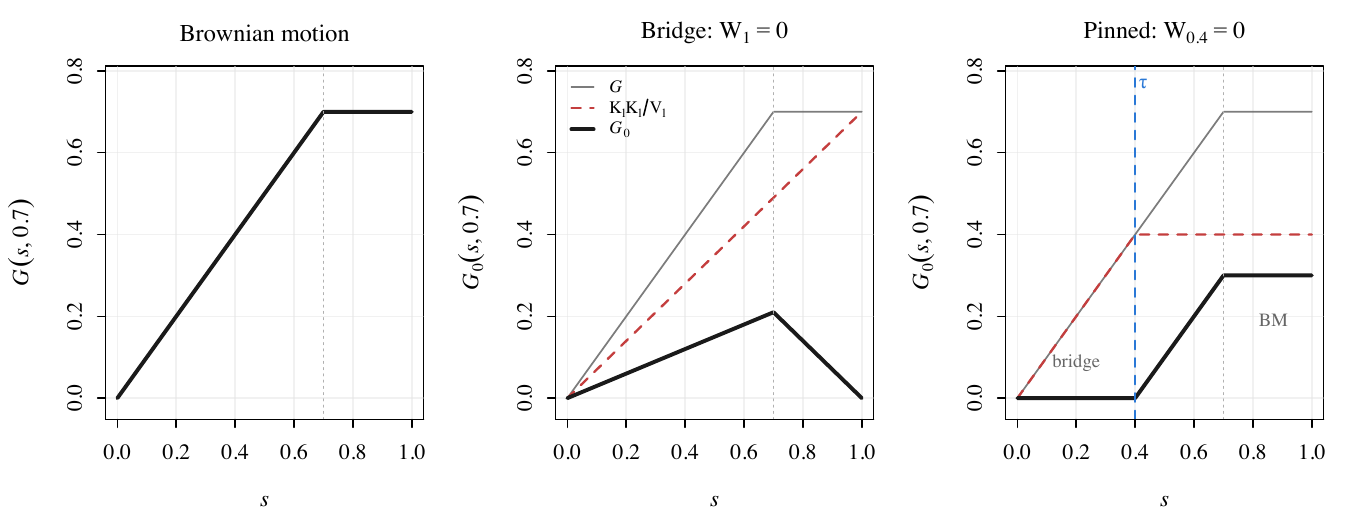}
 \caption{Conditioning as a rank-one covariance update, with every panel evaluated at $t=0.7$. Left: the unconditioned Brownian covariance $G(s,0.7)$. Middle: conditioning on $W_1=0$. Right: conditioning on $W_{0.4}=0$, which gives a bridge on $[0,0.4]$ and an independent Brownian motion on $[0.4,1]$. Each conditioned panel shows the base kernel $G(s,0.7)$ in gray, the rank-one correction $K_\ell(s)K_\ell(0.7)/V_\ell$ dashed in red, and their difference, $G_0(s,0.7)$, as the heavy solid curve. The dashed blue line marks the interior pin at $\tau=0.4$.}
 \label{fig:conditioning}
\end{figure}

\subsection{Interior points and the Markov split}

\begin{example}[Bridge by conditioning]\label{ex:bridge_cond}
Take $X=W$ (Brownian motion) and $\ell(W)=W_1$, so $K_\ell(s) = s$ and $V_\ell = 1$. Then $G_0(s,t) = s\wedge t - st$: the Brownian bridge of Example~\ref{ex:bridge}, now derived as a conditioned process rather than from boundary conditions. Operations~1 and~3 agree here.
\end{example}

\begin{example}[Interior pinning and independent segments]\label{ex:interior}
Condition $W$ on $W_\tau = 0$ for $\tau\in(0,1)$: $K_\ell(s)=s\wedge\tau$, $V_\ell=\tau$, and
\[
 G_0(s,t) = s\wedge t - \frac{(s\wedge\tau)(t\wedge\tau)}{\tau}.
\]
For $s,t\le\tau$ this is $s\wedge t - st/\tau$, a Brownian bridge on $[0,\tau]$. For $s,t\ge\tau$ it is $s\wedge t - \tau$, a fresh Brownian motion on $[\tau,1]$, started from zero at time~$\tau$. The variance $\Var(X_s)=s-(s\wedge\tau)^2/\tau$ vanishes at $s=\tau$, as required, and integrates to the trace $\tfrac12-\tau+\tfrac23\tau^2$ in Table~\ref{tab:constructions}. Pinning at an interior point thus splits the process into a bridge on $[0,\tau]$ and an independent Brownian motion on $[\tau,1]$, as the Markov property implies. See Figure~\ref{fig:detail_pinned}.
\end{example}

\begin{figure}[tbp]
 \centering
 \includegraphics[width=0.85\textwidth]{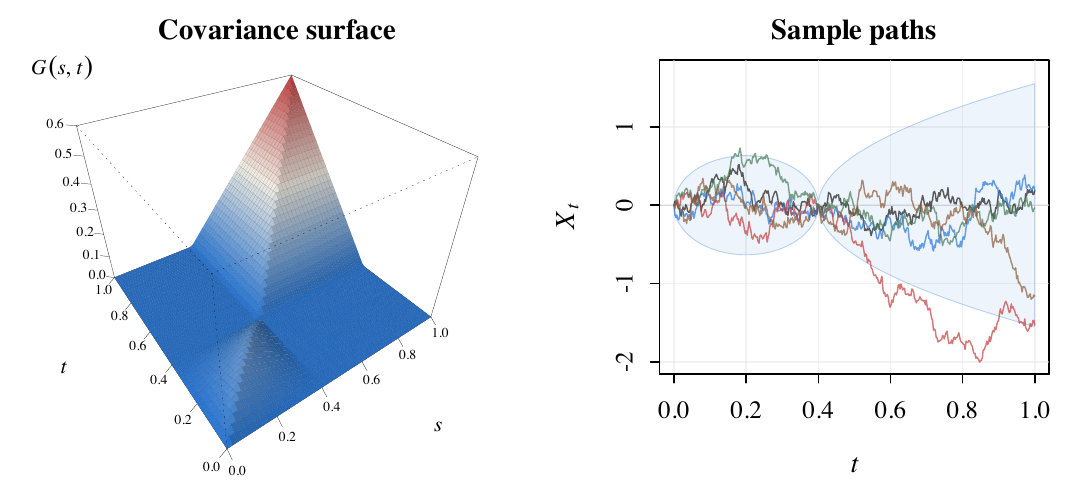}
 \caption{Brownian motion pinned at $\tau=0.4$: covariance surface $G(s,t)$ on $[0,1]^2$ (left) and five sample paths with a $\pm2$-standard-deviation envelope (right). Zero covariance in both off-diagonal blocks expresses independence across the pin. The envelope closes at the original starting point $s=0$ and at the new pin $s=\tau$.}
 \label{fig:detail_pinned}
\end{figure}

\paragraph{Domain splitting.}\label{sec:domain_splitting}

An interior pin is not an endpoint condition for the operator realizations considered here. Conditioning on an interior value nevertheless has a boundary-value interpretation: by Example~\ref{ex:interior}, pinning $W_\tau=0$ produces independent processes on $[0,\tau]$ and $[\tau,1]$. The boundary-value viewpoint reaches the same object by splitting the domain at~$\tau$ and solving a boundary-value problem on each piece, with a Dirichlet condition at~$\tau$ and the original condition retained at the outer end. For $W$ that is Dirichlet--Dirichlet on $[0,\tau]$, giving a bridge, and Dirichlet--Neumann on $[\tau,1]$, giving $s\wedge t-\tau$, a Brownian motion restarted from zero at time~$\tau$. Imposing Dirichlet conditions at both ends of each subinterval would instead give two bridges, a different process. This illustrates the Markov property of Brownian motion: knowing (or forcing) an interior value cuts the process into conditionally independent pieces.

At the level of the Green's function, the split is visible as the vanishing of the cross covariance: for $s\le\tau\le t$, the conditioned covariance $G_0(s,t)=0$, so the Green's function vanishes across the split point. This is the Markov property at the fixed time~$\tau$, expressed through covariance: the two segments are uncorrelated, and hence independent in this Gaussian setting, once $X_\tau$ is fixed.

\subsection{Ties, weighted ties and integral constraints}

\begin{example}[Tied times]\label{ex:tied}
To tie the values at two times, $W_{s_0}=W_{s_1}$ with $s_0<s_1$, take $\ell(W)=W_{s_1}-W_{s_0}$. Then $K_\ell(s) = (s\wedge s_1)-(s\wedge s_0)$ and $V_\ell = s_1-s_0$ (the increment variance), so
\[
 G_0(s,t) = s\wedge t - \frac{[(s\wedge s_1)-(s\wedge s_0)]\,[(t\wedge s_1)-(t\wedge s_0)]}{s_1-s_0}.
\]
\end{example}

\begin{example}[Weighted two-time constraint]\label{ex:weighted}
For $0\le s_0<s_1\le1$, impose $\alpha W_{s_0}+\beta W_{s_1}=0$ by taking $\ell(W)=\alpha W_{s_0}+\beta W_{s_1}$, so $K_\ell(s) = \alpha(s\wedge s_0)+\beta(s\wedge s_1)$ and $V_\ell = \alpha^2 s_0 + \beta^2 s_1 + 2\alpha\beta(s_0\wedge s_1) = s_0(\alpha+\beta)^2+\beta^2(s_1-s_0)$. Substitute these quantities into~\eqref{eq:conditioning} whenever $V_\ell>0$. The only exceptions are $(\alpha,\beta)=(0,0)$ and $\beta=0$ with $s_0=0$; in the latter case $\ell(W)=\alpha W_0$ vanishes identically.
\end{example}

\begin{example}[Integral constraint]\label{ex:integral}
To impose $\int_0^1W_u\,\dd u=0$, take $\ell(W)=\int_0^1W_u\,\dd u$. Then
\[
 K_\ell(s) = \int_0^1 (s\wedge u)\,\dd u = s - \tfrac{1}{2}s^2, \qquad V_\ell = \int_0^1\int_0^1 (u\wedge v)\,\dd u\,\dd v = \tfrac{1}{3},
\]
so $G_0(s,t) = s\wedge t - 3\bigl(s-\tfrac{1}{2}s^2\bigr)\bigl(t-\tfrac{1}{2}t^2\bigr)$. This is integral-conditioned Brownian motion: Brownian motion conditioned to have zero time average. Its variance is $\Var(X_s) = s - 3\bigl(s-\tfrac12 s^2\bigr)^2$, which integrates to the trace $\tfrac1{10}$ recorded in Table~\ref{tab:constructions}. The equivalent kernel $s\wedge t-\tfrac34(s^2-2s)(t^2-2t)$ appears as a Green kernel on a constrained space in~\cite{Raket2021}, Example 2.15. Thus conditioning and a constrained-space construction can reach the same object by different routes. See Figure~\ref{fig:detail_integral}.
\end{example}

\begin{figure}[tbp]
 \centering
 \includegraphics[width=0.85\textwidth]{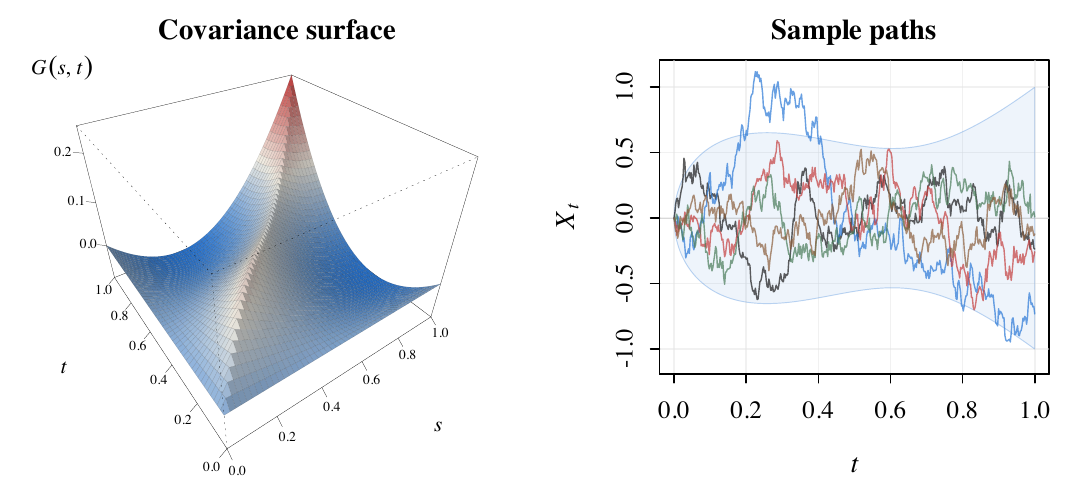}
 \caption{Integral-conditioned Brownian motion: covariance surface $G(s,t)$ on $[0,1]^2$ (left) and five sample paths with a $\pm2$-standard-deviation envelope (right). Every path has zero time average, but no additional value is pinned beyond the original $X_0=0$. The envelope is strictly positive on $(0,1]$, unlike the interior-pinned process, whose envelope also closes at its pin.}
 \label{fig:detail_integral}
\end{figure}

The endpoint pin $W_1=0$ gives the Brownian bridge (Example~\ref{ex:bridge_cond}); an interior pin $W_\tau=0$ splits the path into a bridge and an independent fresh Brownian motion (Example~\ref{ex:interior}). A tie or weighted tie between two times, $W_{s_0}=W_{s_1}$ or $\alpha W_{s_0}+\beta W_{s_1}=0$, gives the corresponding tied process (Examples~\ref{ex:tied}--\ref{ex:weighted}), and $\int_0^1W_u\,\dd u=0$ gives integral-conditioned Brownian motion (Example~\ref{ex:integral}). Derivatives require a different formulation: Brownian paths are not differentiable, so we instead condition on finite-difference quotients of a fixed positive order. Each quotient gives a nonzero rank-one update, but the correction vanishes as the observation window shrinks. That limiting phenomenon is our next subject.

\section{Detectability: which constraints matter}\label{sec:detectability}

A family of constraints can have a vanishing covariance effect as the observation window shrinks. When does this happen? We begin with a calculation for Brownian motion, then extract the general criterion.

\subsection{A shrinking finite-difference constraint}\label{sec:deriv_invisible}

\begin{example}[A symmetric-difference constraint on $W$]\label{ex:deriv_constraint}
Fix $\tau\in(0,1)$ and $0<\varepsilon<\min(\tau,1-\tau)$. Use the symmetric difference quotient $\ell_\varepsilon(W)=\tfrac{1}{2\varepsilon}(W_{\tau+\varepsilon}-W_{\tau-\varepsilon})$ as a derivative surrogate, and condition on $\ell_\varepsilon(W)=0$. Then
\[
 V_{\ell_\varepsilon} = \frac{1}{4\varepsilon^2}\bigl[(\tau+\varepsilon)-2(\tau-\varepsilon)+(\tau-\varepsilon)\bigr] = \frac{1}{2\varepsilon},
\]
while
\[
 K_{\ell_\varepsilon}(s) = \frac{1}{2\varepsilon}\bigl[s\wedge(\tau+\varepsilon) - s\wedge(\tau-\varepsilon)\bigr] \in [0,1]
\]
for \emph{every}~$s$: it is $0$ for $s\le\tau-\varepsilon$, rises affinely across the window, and equals~$1$ for $s\ge\tau+\varepsilon$. Hence
\[
 \sup_s \frac{K_{\ell_\varepsilon}(s)^2}{V_{\ell_\varepsilon}} = 2\varepsilon \;\xrightarrow{\varepsilon\to 0}\; 0,
\]
so the rank-one covariance correction in~\eqref{eq:conditioning} tends uniformly to zero as $\varepsilon\to0$. At each fixed $s>0$, $\Corr\bigl(W_s,\ell_\varepsilon(W)\bigr)=O(\sqrt\varepsilon)$. The surrogate decorrelates from each fixed nondegenerate path value as the window closes; it does not converge to a Brownian derivative. Write $h_\varepsilon=K_{\ell_\varepsilon}/\sqrt{V_{\ell_\varepsilon}}$ and $Z_\varepsilon=\ell_\varepsilon(W)/\sqrt{V_{\ell_\varepsilon}}$. The residual $W_s-h_\varepsilon(s)Z_\varepsilon$ has the conditioned law. Since $Z_\varepsilon$ is standard normal and $\|h_\varepsilon\|_\infty=O(\sqrt\varepsilon)$, this coupling converges uniformly in probability to~$W$. The conditional laws therefore converge weakly to Wiener measure on $C[0,1]$, although every finite-$\varepsilon$ law is supported on a Wiener-null linear constraint and differs from Wiener measure. Dropping the factor $1/(2\varepsilon)$ does not change the conditioning event $\{W_{\tau+\varepsilon}=W_{\tau-\varepsilon}\}$, so it leaves the conclusion unchanged. Yet the observation variance would then tend to~$0$, not~$\infty$. This decorrelation for shrinking difference quotients is noted in~\cite{Raket2021}, Section 4.3, Examples 4.2--4.4, using a one-sided quotient at an endpoint; Example 4.4 there gives an explicit Brownian-bridge limit. Here the quotient is symmetric and centered at an interior time. Theorem~\ref{thm:detect} below isolates the governing criterion, and Corollary~\ref{cor:invisible} extends the calculation to every fixed positive order.
\end{example}

Contrast this with observing a \emph{value}. For $\ell=\delta_\tau$ at an interior time, $K_\ell(s)=s\wedge\tau$ and $V_\ell=\tau$, so the correction $\,(s\wedge\tau)(t\wedge\tau)/\tau$ does not shrink at all: it is the interior pin of Example~\ref{ex:interior}, which splits the path. Both observations concern the same time~$\tau$, but one leaves the covariance unchanged in the limit while the other splits it into independent blocks. What separates them is the correlation with the path values, not the observation variance, which depends on the normalization.

\subsection{The criterion}\label{sec:detect_criterion}

One might guess that a growing observation variance makes a constraint asymptotically invisible. But $V_\ell=\Var(\ell(X))$ depends on the normalization of the observation, not just on the constraint. The events $\{\ell(X)=0\}$ and $\{c\,\ell(X)=0\}$ are identical for every $c\ne 0$, and the rank-one correction in~\eqref{eq:conditioning} is likewise unchanged, but $V_{c\ell}=c^2V_\ell$ can be sent to $0$ or~$\infty$ at will. The invariant quantity is the correlation.

\begin{theorem}[Detectability]\label{thm:detect}
Let $X$ be a zero-mean Gaussian process with covariance~$K$, and let $\ell$ be a linear functional whose value $\ell(X)$ lies in the closed linear span of $\{X_s\}_{s\in[0,1]}$ in $L^2(\mathbb{P})$, with $V_\ell=\Var(\ell(X))<\infty$. Then
\begin{enumerate}
\item[(a)] $V_\ell>0$ if and only if $K_\ell\not\equiv 0$, in which case conditioning on $\ell(X)=0$ changes the covariance by the nonzero rank-one term in~\eqref{eq:conditioning};
\item[(b)] for a family $\{\ell_\varepsilon\}$ of such functionals with $V_{\ell_\varepsilon}>0$,
\begin{equation}\label{eq:detect_corr}
 \lim_{\varepsilon\to 0} G_0^{(\varepsilon)}(s,t) = K(s,t)\ \text{ for all } s,t
 \quad\Longleftrightarrow\quad
 \lim_{\varepsilon\to 0}\frac{K_{\ell_\varepsilon}(s)^2}{V_{\ell_\varepsilon}} = 0\ \text{ for each } s,
\end{equation}
and at every $s$ with $\Var(X_s)>0$ the right-hand condition reads $\Corr\bigl(X_s,\ell_\varepsilon(X)\bigr)\to 0$.
\end{enumerate}
\end{theorem}

\begin{proof}
(a) If $K_\ell\equiv 0$ then $\ell(X)$ is orthogonal in $L^2(\mathbb{P})$ to every $X_s$, hence to their closed linear span; that span contains $\ell(X)$, so $V_\ell = \langle \ell(X),\ell(X)\rangle = 0$. Conversely $V_\ell=0$ makes $\ell(X)=0$ almost surely and $K_\ell\equiv0$. When $K_\ell\not\equiv0$ the correction $K_\ell(s)K_\ell(t)/V_\ell$ in~\eqref{eq:conditioning} is a nonzero kernel.

(b) The correction is $K_{\ell_\varepsilon}(s)K_{\ell_\varepsilon}(t)/V_{\ell_\varepsilon}$, whose diagonal is $K_{\ell_\varepsilon}(s)^2/V_{\ell_\varepsilon}$; taking $s=t$ gives ``$\Rightarrow$''. For ``$\Leftarrow$'', factorizing the rank-one correction gives
\[
 \frac{|K_{\ell_\varepsilon}(s)K_{\ell_\varepsilon}(t)|}{V_{\ell_\varepsilon}}
 \le \Bigl(\frac{K_{\ell_\varepsilon}(s)^2}{V_{\ell_\varepsilon}}\Bigr)^{1/2}\Bigl(\frac{K_{\ell_\varepsilon}(t)^2}{V_{\ell_\varepsilon}}\Bigr)^{1/2}\longrightarrow 0 .
\]
Finally $K_{\ell_\varepsilon}(s)^2/V_{\ell_\varepsilon} = \Corr\bigl(X_s,\ell_\varepsilon(X)\bigr)^2\,\Var(X_s)$ by definition of the correlation, so for $\Var(X_s)>0$ the two conditions agree; where $\Var(X_s)=0$ we have $X_s=0$ almost surely and $K_{\ell_\varepsilon}(s)=0$, so the condition holds automatically.
\end{proof}

Both sides of~\eqref{eq:detect_corr} are unchanged by $\ell_\varepsilon\mapsto c_\varepsilon\ell_\varepsilon$ for $c_\varepsilon\ne0$. The scale-invariant summary is: \emph{a family of constraints becomes invisible in covariance precisely when every fixed nondegenerate path value decorrelates from the observations}.

\begin{remark}[Hypotheses and scope]\label{rem:detect_scope}
Three qualifications matter. First, the closed-span hypothesis is essential in part~(a). If an external, nondegenerate Gaussian observation $\xi$ is independent of~$X$, its covariance with every $X_s$ is zero despite its positive variance, and conditioning on it has no effect. The closed-span hypothesis is automatic for the functionals used here: finite combinations $\sum_j c_jX_{t_j}$ and their $L^2(\mathbb{P})$ limits, including $\int_0^1X_u\,\dd u$ and the difference quotients above. Second, $V_{\ell_\varepsilon}\to\infty$ is sufficient if $\sup_\varepsilon\|K_{\ell_\varepsilon}\|_\infty<\infty$, but not in general. For second and higher Brownian difference quotients, these cross-covariance functions are not uniformly bounded, so we use the ratio in~\eqref{eq:detect_corr} directly. Third, linearity and the closed-span hypothesis ensure that $\bigl(X_s,\ell(X)\bigr)$ is jointly Gaussian, allowing the rank-one Schur complement~\eqref{eq:conditioning}. We make no analogous claim for nonlinear functionals.
\end{remark}

\begin{corollary}\label{cor:invisible}
For Brownian motion, value constraints $W_\tau=0$ at interior times are detectable. For finite-difference quotients approximating any fixed derivative order $k\ge1$, the covariance correction tends uniformly to zero as the spacing shrinks.
\end{corollary}

\begin{proof}
For a value constraint, $\ell=\delta_\tau$ with $\tau\in(0,1]$ gives $K_\ell(s)=s\wedge\tau\not\equiv 0$, so Theorem~\ref{thm:detect}(a) applies and the covariance changes.

At an interior time~$\tau$, fix an order $k\ge1$ and let $\ell_\varepsilon$ be the $k$th forward difference quotient, with $\varepsilon>0$ small enough that $\tau+k\varepsilon\le1$: 
\[
 \ell_\varepsilon(W) = \varepsilon^{-k}\sum_{j=0}^{k}(-1)^{k-j}\binom{k}{j} W_{\tau+j\varepsilon} .
\]
Rewriting it in the independent increments $\xi_i = W_{\tau+i\varepsilon}-W_{\tau+(i-1)\varepsilon}\sim N(0,\varepsilon)$ gives $\ell_\varepsilon(W)=\varepsilon^{-k}\sum_{i=1}^{k}(-1)^{k-i}\binom{k-1}{i-1}\xi_i$, so by Vandermonde's identity
\begin{equation}\label{eq:kth_quotient_var}
 V_{\ell_\varepsilon} = \varepsilon^{1-2k}\sum_{i=1}^{k}\binom{k-1}{i-1}^{2} = \binom{2k-2}{k-1}\,\varepsilon^{1-2k} .
\end{equation}
Since $\sum_j(-1)^{k-j}\binom kj=0$ for $k\ge1$, the cross-covariance $K_{\ell_\varepsilon}(s)=\varepsilon^{-k}\sum_j(-1)^{k-j}\binom kj\,\bigl(s\wedge(\tau+j\varepsilon)\bigr)$ vanishes for $s\le\tau$, is affine between consecutive nodes, and is constant for $s\ge\tau+k\varepsilon$. Its absolute maximum is therefore attained at a node. At $s=\tau+m\varepsilon$, its value is $\varepsilon^{1-k}\sum_j(-1)^{k-j}\binom kj\,(m\wedge j)$. Hence $\|K_{\ell_\varepsilon}\|_\infty=c_k\,\varepsilon^{1-k}$, where $c_k$ is independent of $\varepsilon$ and~$\tau$, and
\[
 \sup_s \frac{K_{\ell_\varepsilon}(s)^2}{V_{\ell_\varepsilon}} = \frac{c_k^2}{\binom{2k-2}{k-1}}\,\varepsilon \;\longrightarrow\; 0 .
\]
Thus every fixed positive order has a uniformly vanishing covariance correction. For each fixed $s>0$ and~$k$, the correlation is $O(\sqrt\varepsilon)$. For $k\ge2$, $\|K_{\ell_\varepsilon}\|_\infty=c_k\varepsilon^{1-k}$ diverges, so the additional boundedness assumption in Remark~\ref{rem:detect_scope} does not apply.
\end{proof}

\begin{remark}[A Neumann condition is not a derivative constraint]\label{rem:neumann_vs_derivative}
Corollary~\ref{cor:invisible} may seem to conflict with~\eqref{eq:BM_green_intro}. For each interior source point $t\in(0,1)$, the Brownian covariance $s\wedge t$ satisfies the Neumann condition $\partial_sK(s,t)|_{s=1}=0$. Yet conditioning on a shrinking finite-difference quotient leaves the original covariance in the limit. There is no contradiction: the derivatives concern different objects.

First, writing ``$X'_\tau=0$'' for a Brownian path does not specify a canonical Gaussian conditional law. The continuous functions differentiable at~$\tau$ with derivative zero form a Borel subset of $C[0,1]$, but this set has Wiener measure zero. More importantly, there is no Brownian derivative random variable to use in the finite-variance Gaussian formula of Section~\ref{sec:recentring_conditioning}. We can instead condition on a difference quotient. In Example~\ref{ex:deriv_constraint}, its variance is $1/(2\varepsilon)$ and its cross-covariance is uniformly bounded, so the covariance correction vanishes as the window closes. At every fixed $\varepsilon>0$ conditioning changes the covariance, but the change vanishes as $\varepsilon\to0$.

Second, the Neumann condition belongs to the operator domain, not to its sample paths. For the Brownian operator, $\operatorname{dom}(\mathcal{L})$ is the range of the covariance operator, and its functions have derivative zero at $s=1$. The kernel $s\wedge t$ has a diagonal kink, but applying the covariance operator to an $L^2$ function produces a function with a well-defined endpoint derivative. Let $g\in L^2[0,1]$ and set $Y=\int_0^1 g(u)\,W_u\,\dd u$. Then $\Cov(W_s,Y) = \int_0^1 (s\wedge u)\,g(u)\,\dd u$, whose derivative in~$s$ is $\int_s^1 g(u)\,\dd u$, so
\begin{equation}\label{eq:neumann_prob}
 \partial_s \Cov(W_s,Y)\big|_{s=1} = 0 \qquad\text{for every } g\in L^2[0,1] .
\end{equation}
Equivalently, $\Cov(W_1-W_{1-h},\,Y) = \int_{1-h}^{1}g(u)(u-1+h)\,\dd u = O(h^{3/2}) = o(h)$ by Cauchy--Schwarz: the covariance between a short terminal increment and any such integrated linear observation is $o(h)$. This is the second-moment meaning of ``no flux at $s=1$''.

Thus the Neumann derivative acts on the deterministic function $s\mapsto\Cov(X_s,Y)$ for these integrated observations; a literal derivative constraint would act on the random Brownian path. The operator ranges make the distinction precise: the Cameron--Martin space, the range of $\mathcal{L}^{-1/2}$, includes smooth functions with $f'(1)\ne0$, while the Neumann condition holds throughout the range of $\mathcal{L}^{-1}$. Operations~1 and~3 in Section~\ref{sec:two_mechanisms} therefore cannot be interchanged in this way: conditioning on a Brownian path derivative is not an alternative formulation of the operator's Neumann condition.
\end{remark}

\subsection{A smoothness hierarchy}

Path smoothness determines the hierarchy. Let $X^{(m)}$ be $m$-fold integrated Brownian motion, $X_t^{(m)}=\int_0^tX_u^{(m-1)}\,\dd u$ with $X^{(0)}=W$. Almost surely its paths are $C^m$, and its $m$th derivative is the continuous, nowhere differentiable process~$W$. For $m\ge1$, the covariance kernels correspond to polyharmonic operators $(-1)^{m+1}\dd^{2(m+1)}/\dd s^{2(m+1)}$, not the Laplacian; they appear here to illustrate this hierarchy. For $k\le m$, the $k$th derivative is the finite-variance process $X^{(m-k)}$. The next derivative is white noise only in the distributional sense, not a pointwise observation. Table~\ref{tab:detect} therefore separates two cases: constraints on values and derivatives through order~$m$ are detectable at interior times, whereas the covariance effect of finite-difference constraints of higher fixed order vanishes as their spacing shrinks. At an endpoint the stencil must be one-sided. At $s=0$, values and derivatives through order~$m$ already vanish by construction.

The detectable case is immediate at an interior time: for $k\le m$, the $k$th derivative is $X^{(m-k)}$, so the difference quotients converge to a finite-variance functional with $K_\ell\not\equiv0$. For higher orders, the Wiener representation $X^{(m)}_t=\int_0^t\frac{(t-u)^m}{m!}\,\dd W_u$ supplies the analogue of the estimate in Corollary~\ref{cor:invisible}. For a $k$th forward difference quotient at spacing $\varepsilon$ with $k>m$, the polynomial $(t-u)^m/m!$ has degree $m<k$ in~$t$, so the difference annihilates every contribution from $u<\tau$ exactly; what survives is supported on an interval of length $k\varepsilon$ with amplitude of order $\varepsilon^{m-k}$. Hence $V_{\ell_\varepsilon}=C_{m,k}\,\varepsilon^{2m-2k+1}$ with $C_{m,k}>0$ and $\|K_{\ell_\varepsilon}\|_\infty=O(\varepsilon^{m-k+1})$, so $\sup_s K_{\ell_\varepsilon}(s)^2/V_{\ell_\varepsilon}=O(\varepsilon)$. Thus the covariance correction is uniformly $O(\varepsilon)$, as for Brownian motion. The corresponding correlation rates apply at fixed times with positive variance, not uniformly near times where $\Var(X^{(m)}_s)$ tends to zero. At an endpoint, use a one-sided stencil; values and derivatives through order~$m$ at $s=0$ are already fixed by construction.

\begin{table}[t]
\centering
\caption{Detectability at a fixed interior time, classified by almost-sure path regularity. ``BM'' denotes Brownian motion. For derivative orders beyond the path's regularity, ``invisible'' means that conditioning a finite-difference quotient to zero has a vanishing covariance effect as its spacing shrinks, not that a nonexistent pointwise derivative can be conditioned on.}\label{tab:detect}
\small
\begin{tabular}{lccc}
\toprule
Process & Path regularity & Value $f(\tau)=0$ & Derivative $f^{(k)}(\tau)=0$ \\
\midrule
Brownian motion $X^{(0)}$ & $C^0$, nowhere diff.\ & detectable & invisible for all $k\ge 1$ \\
Integrated BM $X^{(1)}$ & $C^1$ & detectable & detectable ($k=1$); invisible ($k\ge 2$) \\
$m$-fold integrated BM $X^{(m)}$ & $C^m$ & detectable & detectable ($k\le m$); invisible ($k>m$) \\
\bottomrule
\end{tabular}
\end{table}

\section{Boundary conditions versus constraints}\label{sec:bc_vs_constraints}

We now compare the three operations of Section~\ref{sec:two_mechanisms}, focusing on what happens when a condition is added to an already specified problem.

\subsection{Counting conditions}\label{sec:counting}

The homogeneous equation for a second-order operator has a two-dimensional solution space. Two independent homogeneous endpoint conditions specify a realization; fewer generally leave nonuniqueness, while additional independent conditions generally overdetermine the boundary-value problem. This count alone guarantees neither self-adjointness nor invertibility: the boundary form must vanish, and the null space must be trivial for an ordinary inverse to exist. Proposition~\ref{prop:bulk} illustrates why merely counting conditions is insufficient. In the degenerate cases, two endpoint conditions still leave a null space, so the balanced normalization is an additional condition selecting a unique solution after the source has been corrected. By contrast, a third condition on a non-degenerate problem may overdetermine it as a boundary-value problem yet define an ordinary rank-one covariance update through~\eqref{eq:conditioning}.

\subsection{Overconstraining: boundary-value problem versus conditioning}\label{sec:overconstraining}

An additional condition therefore has different meanings in the two constructions. As an operator-domain restriction it may overdetermine the boundary-value problem; as Gaussian conditioning it changes the covariance according to~\eqref{eq:conditioning}. For a shrinking family of observations, the limiting change depends on detectability. Brownian motion with the additional condition $f'(0)=0$ gives a sharp example: the boundary-value problem is overdetermined, whereas conditioning on the corresponding vanishing-scale difference quotients has no limiting effect, by Corollary~\ref{cor:invisible}. In a degenerate problem, a pin can instead replace the balanced normalization. Comparing the resulting recentering with conditioning is the subject of the rest of this section.

\subsection{An interior pin in place of the balanced condition}\label{sec:interior_pin}

The balanced condition~\eqref{eq:zeromean} selects a unique kernel for a degenerate operator. Another possibility is a point condition, $f(\tau)=0$ for some $\tau\in[0,1]$, since the only constant vanishing at~$\tau$ is zero. The pin removes the ambiguity, but it does not by itself resolve the source obstruction. This distinction makes it a useful test case for comparing recentering and conditioning.

\paragraph{Pinning repairs uniqueness, not solvability.}
An interior condition does not remove the obstruction of Lemma~\ref{lem:fredholm}. Simply imposing $u(\tau)=0$ restricts the original $\operatorname{dom}(\mathcal{L})$ and cannot enlarge the range to admit the uncorrected source $\delta(\cdot-t)$. A source correction is still needed. With the correction in~\eqref{eq:corrected}, solving $\mathcal{L}u=\delta(\cdot-t)-1$ under $f'(0)=f'(1)=0$ and $u(\tau)=0$ gives $G(s,t)-G(\tau,t)$, which is not symmetric. The symmetric recentered kernel below instead uses the dipole source $\delta(\cdot-t)-\delta(\cdot-\tau)$; the kernel's integral in the first argument generally depends on~$t$, so the kernel is not balanced. In the discretization of Remark~\ref{rem:fem_exact} with Neumann conditions at both ends and $n=200$, the stiffness matrix has rank $200$ of $201$; appending the row $e_\tau^\top$ gives full column rank.

\paragraph{The pinned kernel.}
Two kinds of kernel modification should be distinguished. Symmetric kernels satisfying the corrected equation~\eqref{eq:corrected} together with~$\mathcal{B}$ differ only by an additive constant. The wider algebraic family $G+c\otimes\mathbf{1}+\mathbf{1}\otimes c$ need not satisfy~\eqref{eq:corrected} pointwise. Part~(b) below shows that all these modifications disappear under double-centering. Interpreting them as $\{1\}$-inverses satisfying $\mathcal{L}H\mathcal{L}=\mathcal{L}$ requires separate domain and range conditions; the covariance identity itself needs no such operator interpretation.

Operation~2 gives the pinned kernel directly: subtract $X_\tau$ from the path. This is recentering with $\ell=\delta_\tau$ in place of $\ell=\int_0^1$, so its covariance follows from~\eqref{eq:recentring}. The next proposition records which properties survive.

\begin{proposition}[The pinned kernel]\label{prop:pinned}
Let $(\mathcal{L},\mathcal{B})$ be either degenerate operator of Section~\ref{sec:degenerate}, let $G$ be its balanced Green's function~\eqref{eq:periodic} or~\eqref{eq:NN}, let $X$ be the corresponding process, and fix $\tau\in[0,1]$. Put
\begin{equation}\label{eq:pinned}
 G_\tau(s,t) = G(s,t) - G(s,\tau) - G(\tau,t) + G(\tau,\tau) .
\end{equation}
\begin{enumerate}
\item[\emph{(a)}] $G_\tau(s,t) = \Cov\bigl(X_s-X_\tau,\;X_t-X_\tau\bigr)$; hence $G_\tau$ is symmetric and positive semidefinite, and $G_\tau(\tau,t)=0$ for every~$t$.
\item[\emph{(b)}] $G_\tau$ is invariant under $G\mapsto G+c\otimes\mathbf{1}+\mathbf{1}\otimes c$, for every function~$c$.
\item[\emph{(c)}] For $\tau,t\in(0,1)$ the function $G_\tau(\cdot,t)$ satisfies the same boundary conditions~$\mathcal{B}$ as $G(\cdot,t)$.
\item[\emph{(d)}] $\int_0^1 G_\tau(s,t)\,\dd s = G(\tau,\tau)-G(\tau,t)$. So $G_\tau$ is balanced if and only if $G(\tau,\cdot)$ is constant, which holds for neither degenerate operator.
\end{enumerate}
\end{proposition}

\begin{proof}
(a) Expanding $\Cov(X_s-X_\tau,X_t-X_\tau)$ by bilinearity gives the covariance identity and hence positive semidefiniteness. Setting $s=\tau$ in~\eqref{eq:pinned} gives $G_\tau(\tau,t)=0$.
(b) Substituting $G(s,t)+c(s)+c(t)$ into~\eqref{eq:pinned} adds $\{c(s)+c(t)\}-\{c(s)+c(\tau)\}-\{c(\tau)+c(t)\}+2c(\tau)=0$.
(c) As functions of~$s$, the terms in~\eqref{eq:pinned} are $G(\cdot,t)$, $-G(\cdot,\tau)$ and constants. The first two satisfy~$\mathcal{B}$ because $t,\tau$ are interior; constants satisfy~$\mathcal{B}$ in both degenerate cases. The claim follows by linearity.
(d) The first two terms integrate to zero by~\eqref{eq:zeromean}. For the periodic kernel, $G(\tau,t)=\tfrac12 B_2(|\tau-t|)$ is nonconstant in~$t$ for every~$\tau$, and likewise $G_\NN(\tau,t)=\tfrac12(\tau^2+t^2)-\tau\vee t+\tfrac13$.
\end{proof}

\paragraph{Recentering at $\tau$ is not conditioning on $X_\tau=0$.}
Conditioning also produces a kernel that vanishes at~$\tau$. Taking $\ell(X)=X_\tau$ in Theorem~\ref{thm:conditioning} gives
\begin{equation}\label{eq:pinned_cond}
 G_\tau^{\mathrm{cond}}(s,t) = G(s,t) - \frac{G(s,\tau)\,G(\tau,t)}{G(\tau,\tau)} ,
\end{equation}
which is symmetric, positive semidefinite and zero at $s=\tau$, but differs from the recentered kernel. Subtracting~\eqref{eq:pinned_cond} from~\eqref{eq:pinned} gives
\begin{equation}\label{eq:pin_gap}
 G_\tau(s,t) - G_\tau^{\mathrm{cond}}(s,t) = \frac{\bigl\{G(\tau,\tau)-G(s,\tau)\bigr\}\,\bigl\{G(\tau,\tau)-G(\tau,t)\bigr\}}{G(\tau,\tau)} ,
\end{equation}
a positive-semidefinite rank-one kernel. It vanishes identically exactly when $G(\tau,\cdot)$ is constant, a case excluded by Proposition~\ref{prop:pinned}(d). The recentered process therefore has at least as much variance at each time as the conditioned process, with strict inequality wherever $G(s,\tau)\ne G(\tau,\tau)$. For these kernels, the gap is nonnegative pointwise but vanishes whenever $s=\tau$ or $t=\tau$.

The constructions also differ in their response to an independent constant offset. Double-centering~\eqref{eq:pinned} is invariant under $G\mapsto G+c\,\mathbf{1}\otimes\mathbf{1}$ by Proposition~\ref{prop:pinned}(b), whereas conditioning depends on the starting covariance. For $c\ge0$, write $r=G(\tau,\tau)$ and let $G_{\tau,c}^{\mathrm{cond}}$ denote the conditioned kernel of $G+c\,\mathbf{1}\otimes\mathbf{1}$. Direct substitution gives
\[
 G_{\tau,c}^{\mathrm{cond}}(s,t) = G_\tau^{\mathrm{cond}}(s,t) + \frac{c}{c+r}\,\bigl\{G_\tau(s,t)-G_\tau^{\mathrm{cond}}(s,t)\bigr\} ,
\]
so the change is the fraction $c/(c+r)$ of the gap~\eqref{eq:pin_gap} and approaches that gap as $c\to\infty$. The recentered kernel~\eqref{eq:pinned} remains unchanged. Starting from the balanced covariance, recentering replaces zero time average with a pin, whereas conditioning adds the pin while preserving zero time average. Their different covariances reflect these different operations.

The three operations of Section~\ref{sec:two_mechanisms} thus need not give the same law. Boundary conditions and conditioning both produce the Brownian bridge in Example~\ref{ex:bridge_cond}, while Remark~\ref{rem:neumann_vs_derivative} explains why a covariance Neumann condition is not obtained by limiting derivative conditioning. In the present comparison, a balanced covariance has already been chosen for the degenerate realization. Both recentering and conditioning then give positive-semidefinite kernels vanishing at~$\tau$, but~\eqref{eq:pin_gap} measures their difference.

\begin{remark}[When the two routes agree]\label{rem:slepian_conditioning}
Although the interior pin separates recentering from conditioning, Example~\ref{ex:slepian} gives a case where they agree. Condition Slepian Brownian motion on $\ell(X)=\tfrac12(X_0+X_1)=0$. Here $K_\ell(s)=\tfrac12\bigl\{G(s,0)+G(s,1)\bigr\}=\tfrac14$ for every~$s$ and $V_\ell=\tfrac14$, so~\eqref{eq:conditioning} subtracts the constant $\tfrac14$ and returns $G_{-1}$. Thus the anti-periodic process is Slepian Brownian motion with its endpoint average conditioned to zero. Since $K_\ell\equiv V_\ell$, this is the exceptional case in Proposition~\ref{prop:diffuse}: the two routes agree at finite offset variance, not just in the diffuse limit.

The agreement extends to the whole family~\eqref{eq:slepian_family}. Since $G_{-1}(s,0)+G_{-1}(s,1)=0$ for every~$s$, only the additive constant contributes to $K_\ell$: for $G_{-1}+v$, we have $K_\ell\equiv v$ and $V_\ell=v$. Conditioning therefore returns $G_{-1}$ for every $v>0$, taking a one-parameter family of distinct operators to a single covariance. At $v=0$ the condition already holds almost surely: $\Corr(X_0,X_1)=-1$, so $X_1=-X_0$ and $V_\ell=0$. Agreement in this family does not make recentering and conditioning interchangeable in general.
\end{remark}

\paragraph{The pinned process.}
The pin replaces, rather than preserves, the zero-time-average normalization. Proposition~\ref{prop:pinned}(d) gives $\int_0^1 G_\tau(s,t)\,\dd s = G(\tau,\tau)-G(\tau,t)$. For the Neumann--Neumann kernel, this yields
\[
 \max_t\Bigl|\int_0^1 G_\tau(s,t)\,\dd s\Bigr| \;=\; \tfrac12\max\bigl\{\tau^2,(1-\tau)^2\bigr\} ,
\]
with the maximum attained at $t=1$ for $\tau<\tfrac12$ and at $t=0$ for $\tau>\tfrac12$. Its minimum over~$\tau$ is $\tfrac18$, attained at $\tau=\tfrac12$, so it never vanishes. For the periodic kernel, $G(\tau,\tau)-G(\tau,t) = \tfrac12|\tau-t|(1-|\tau-t|)$ has maximum $\tfrac18$ for every~$\tau$, attained wherever $|t-\tau|=\tfrac12$. These nonzero integrals distinguish each pinned process from any rescaling of its balanced counterpart. Proposition~\ref{prop:pinned}(a) and~(c) identifies what does survive: symmetry, positive semidefiniteness and, for interior $t,\tau$, the original covariance boundary conditions. The identity $G_\tau(\tau,\cdot)=0$ records the deterministic pin. In particular, an interior pin preserves periodicity in the periodic case and both Neumann conditions in the Neumann--Neumann case, for interior source points.

\paragraph{What the pin does to each degenerate beast.}
Proposition~\ref{prop:pinned}(a) identifies the pinned beasts as increment processes about~$\tau$. For the Neumann--Neumann process, $X_s-X_\tau=W_s-W_\tau$ because the balanced centering cancels. Its covariance is block diagonal: $\Cov(W_s-W_\tau,W_t-W_\tau)=0$ whenever $s<\tau<t$. On $[0,\tau]$ the kernel is $\tau-s\vee t$, that of time-reversed Brownian motion; on $[\tau,1]$ it is $s\wedge t-\tau$, that of Brownian motion restarted at~$\tau$. The pin splits this beast into two independent one-sided relatives, giving another instance of the domain-splitting structure in Section~\ref{sec:domain_splitting}, here obtained by recentering. For $0<\tau<1$, the corresponding Neumann--Dirichlet / Dirichlet--Neumann precision operator has a piecewise $H^2$ domain: Neumann conditions at the outer endpoints, zero values at~$\tau$ from both sides, and no derivative matching there. This is not the original global $H^2$ domain with an added pin.

The periodic process does not split: its increment is $B_s-B_\tau$ for the bridge~$B$, and its cross block is $-(s-\tau)(t-\tau)$, strictly positive whenever $s<\tau<t$.

\begin{remark}[The coincident-kink exception]\label{rem:coincident_kink}
The endpoint exclusion in Proposition~\ref{prop:pinned}(c) is essential. For the Neumann--Neumann kernel, \eqref{eq:pinned} gives $\partial_s G_\tau(0,t) = \partial_s G(0,t)-\partial_s G(0,\tau)$. Each term vanishes when its second argument is interior. When that argument is~$0$, the kink lies on the boundary and the one-sided derivative is $-1$, not~$0$. Thus, for an interior pin, $\partial_s G_\tau(0,t)=0$ at every interior~$t$ but $\partial_s G_\tau(0,0)=-1$; for an endpoint pin $\tau=0$, we instead have $\partial_s G_0(0,t)=+1$ at every interior~$t$. The right endpoint has the corresponding reflected behavior. The kink's coincidence with the boundary explains why a pin at $\tau\in\{0,1\}$ behaves differently from an interior pin; it does not contradict the boundary condition stated for interior sources. Parts~(a), (b) and~(d) hold for every $\tau\in[0,1]$.
\end{remark}

\subsection{Choosing a centering}\label{sec:centring}

The endpoint exception in Remark~\ref{rem:coincident_kink} is useful: changing the covariance boundary condition at one end turns a degenerate beast into a familiar one. Mean-centering and point-centering are the same path operation with different functionals, $\ell=\int_0^1$ and $\ell=\delta_\tau$. Both give a kernel of the form $\Pi K\Pi^*$, where $\Pi f = f-\ell(f)\mathbf{1}$.

Choosing~$\ell$ fixes both the source correction and the side condition. If $w$ is the measure representing~$\ell$, then $\Pi^*\delta_t = \delta(\cdot-t)-w$: this gives~\eqref{eq:corrected} for $\ell=\int_0^1$ and the dipole $\delta(\cdot-t)-\delta(\cdot-\tau)$ for $\ell=\delta_\tau$. Here the two projections act on different natural spaces. Mean-centering is an orthogonal projection on $L^2$; point-centering is an idempotent map on continuous functions, and $\Pi^*$ denotes its dual action on measures, not an $L^2$ adjoint. Point evaluation is not $L^2$-bounded. Once the starting kernel and centering are fixed, the projected covariance has no freely chosen additive constant, consistent with Proposition~\ref{prop:k_range}.

\begin{proposition}[Point-centering at an endpoint]\label{prop:centring_identities}
Let $G_\NN$ and $G_{\mathrm{per}}$ be the balanced kernels~\eqref{eq:NN} and~\eqref{eq:periodic}. With $\tau=0$ in~\eqref{eq:pinned},
\[
 G_\NN(s,t)-G_\NN(s,0)-G_\NN(0,t)+G_\NN(0,0) = s\wedge t,
\]
\[
 G_{\mathrm{per}}(s,t)-G_{\mathrm{per}}(s,0)-G_{\mathrm{per}}(0,t)+G_{\mathrm{per}}(0,0) = s\wedge t-st .
\]
The free--free process pinned at its left end is Brownian motion; the loop pinned at a point is a Brownian bridge.
\end{proposition}

\begin{proof}
For the Neumann--Neumann kernel, $G_\NN(s,0)=\tfrac12 s^2-s+\tfrac13$, $G_\NN(0,t)=\tfrac12 t^2-t+\tfrac13$ and $G_\NN(0,0)=\tfrac13$, so the four terms of~\eqref{eq:pinned} leave $-s\vee t+s+t = s\wedge t$. For the periodic kernel, $G_{\mathrm{per}}(s,0)=\tfrac12 s^2-\tfrac12 s+\tfrac1{12}$ and $G_{\mathrm{per}}(0,0)=\tfrac1{12}$, and the same cancellation leaves $-st+\tfrac12(s+t-|s-t|) = s\wedge t-st$.
\end{proof}

Recentering at the endpoint changes the boundary conditions, as described in Remark~\ref{rem:coincident_kink}. Pinning at $\tau=0$ gives $G_0(0,t)=0$, introducing a Dirichlet condition at the left end. For the Neumann--Neumann kernel and interior~$t$, the left derivative changes from $0$ to~$+1$, as in Remark~\ref{rem:coincident_kink}. The resulting Dirichlet--Neumann kernel is Brownian motion. In the periodic case, $G_0(1,t)=G_0(0,t)=0$, so both ends acquire Dirichlet conditions and the result is the bridge. Brownian motion is therefore obtained by point-centering the balanced Neumann--Neumann process, not by adding $f(0)=0$ while retaining both original Neumann conditions. Indeed, $s\wedge t$ has derivative~$+1$ at $s=0$ for interior~$t$.

The loop also runs the other way: mean-centering Brownian motion and the Brownian bridge returns the two balanced degenerate processes, by Proposition~\ref{prop:centred_constructions}. This is the normalization~\eqref{eq:zeromean} expressed as an operation on paths.

\paragraph{Centering as a diffuse-conditioning limit.}
Section~\ref{sec:interior_pin} compares recentering and conditioning for a fixed balanced covariance. Their difference changes when an independent constant offset is added. Adding $v\,\mathbf{1}\otimes\mathbf{1}$ to a degenerate covariance puts variance~$v$ in the null direction of~$\mathcal{L}$. As this variance grows, conditioning converges to centering, as the next proposition shows.

\begin{proposition}[Centering is the diffuse limit of conditioning]\label{prop:diffuse}
Let $G$ be the covariance function of a centered Gaussian process~$X$, and let $\ell$ be a linear functional for which $\ell(X)$ is defined and has finite variance, with $\ell(\mathbf{1})=1$. Point evaluation $f\mapsto f(\tau)$ and the time average $f\mapsto\int_0^1 f$ satisfy the normalization. Write $K_\ell(s)=\ell_t G(s,t)$, $V_\ell=\ell_s\ell_t G(s,t)$, and $\Pi f = f-\ell(f)\mathbf{1}$, so that
\[
 (\Pi G \Pi^*)(s,t) = G(s,t)-K_\ell(s)-K_\ell(t)+V_\ell .
\]
For $v\ge 0$ with $V_\ell+v>0$, set $G_v = G+v\,\mathbf{1}\otimes\mathbf{1}$ and let
\[
 G_{0}^{(v)}(s,t) = G_v(s,t) - \frac{\bigl(\ell_uG_v(s,u)\bigr)\bigl(\ell_uG_v(u,t)\bigr)}{\ell_u\ell_wG_v(u,w)}
\]
be the covariance obtained by conditioning a centered Gaussian process with covariance $G_v$ on $\ell(X)=0$. Then
\begin{equation}\label{eq:diffuse}
 G_{0}^{(v)}(s,t) = (\Pi G \Pi^*)(s,t) \;-\; \frac{\bigl\{K_\ell(s)-V_\ell\bigr\}\bigl\{K_\ell(t)-V_\ell\bigr\}}{v+V_\ell} .
\end{equation}
Thus $G_0^{(v)}\to\Pi G \Pi^*$ pointwise at rate $O(1/v)$ as $v\to\infty$. At $v=0$, provided $V_\ell>0$, the subtracted term is the gap between recentering and conditioning. If $K_\ell\equiv V_\ell$, the term vanishes identically: the two constructions agree at every finite~$v$ with $V_\ell+v>0$, as in Remark~\ref{rem:slepian_conditioning}.
\end{proposition}

\begin{proof}
Since $\ell(\mathbf{1})=1$, the shifted kernel has $\ell_t G_v(s,t) = K_\ell(s)+v$ and $\ell_s\ell_t G_v = V_\ell+v$. Writing $p=K_\ell(s)$, $q=K_\ell(t)$, $r=V_\ell$, the displayed definition gives
\[
 G_0^{(v)}(s,t) = G(s,t)+v-\frac{(p+v)(q+v)}{r+v} .
\]
Polynomial division of $v^2+v(p+q)+pq$ by $v+r$ gives quotient $v+(p+q-r)$ and remainder $pq-r(p+q-r)=(p-r)(q-r)$. The right-hand side is therefore $G(s,t)-p-q+r-(p-r)(q-r)/(v+r)$, proving~\eqref{eq:diffuse}.
\end{proof}

For $\ell=\delta_\tau$, conditioning after adding offset variance~$v$ converges to the double-centered kernel~\eqref{eq:pinned}; at $v=0$, \eqref{eq:diffuse} recovers~\eqref{eq:pinned_cond} and the gap~\eqref{eq:pin_gap}. The same identity covers mean-centering when $\ell=\int_0^1$. Recentering subtracts the observed functional from each path value; conditioning subtracts each path value's least-squares regression on that observation. When $V_\ell>0$, adding variance in the constant direction reduces the covariance gap by the factor $V_\ell/(v+V_\ell)$. Unless $K_\ell\equiv V_\ell$, agreement occurs only in the diffuse limit, explaining the saturation in Section~\ref{sec:interior_pin}. In the exceptional case $K_\ell\equiv V_\ell$, the two constructions agree at every finite~$v$ with $v+V_\ell>0$.

\section{The bestiary}

The twelve beasts can be approached one at a time through the catalogue entries in Section~\ref{sec:entries}, or compared side by side in the table in Section~\ref{sec:tables}.

\subsection{The entries}\label{sec:entries}

The entries use nine fields: \textsc{condition}, \textsc{kernel}, \textsc{variance}, \textsc{endpoints}, \textsc{from}~$W$, \textsc{marks}, \textsc{spectrum}, \textsc{elsewhere} and \textsc{see}. The \textsc{condition} field states the boundary conditions and any normalization, or the conditioning event. The \textsc{marks} field records stationarity and triangularity, as in Table~\ref{tab:constructions}; \textsc{spectrum} gives the precision eigenvalues $\omega^2$ through an equation for~$\omega$, together with the covariance trace. Covariance eigenvalues are the reciprocals $\omega^{-2}$. The \textsc{see} field points to the derivation and figure. When an endpoint has zero variance, endpoint correlation is undefined, not zero. The order follows Table~\ref{tab:roadmap}: operator-based constructions first, then conditioned processes.

\begin{beast}{Brownian motion}{Wiener process; the Robin corner $(0,\infty)$}
\item[Condition] $f(0)=0$, $f'(1)=0$: Dirichlet--Neumann, the Robin corner $(\lambda_0,\lambda_1)=(0,\infty)$.
\item[Kernel] $G(s,t)=s\wedge t$.
\item[Variance] $\Var(X_s)=s$.
\item[Endpoints] $\Var(X_0)=0$, $\Var(X_1)=1$. Endpoint correlation is undefined because $X_0=0$ almost surely.
\item[From $W$] $X_s=W_s$.
\item[Marks] Not stationary; triangular, hence Gauss--Markov.
\item[Spectrum] $\omega=(n-\tfrac12)\pi$, $n\ge1$; trace $\tfrac12$.
\item[Elsewhere] The reference process of the correspondence (Section~\ref{sec:from_green_to_cov}).
\item[See] Example~\ref{ex:BM}; Figure~\ref{fig:detail_bm}.
\end{beast}

\begin{beast}{Brownian bridge}{tied-down Brownian motion; the Robin corner $(0,0)$}
\item[Condition] $f(0)=f(1)=0$: Dirichlet--Dirichlet, the Robin corner $(0,0)$. Conditioning $W$ on $W_1=0$ gives the same process, so the operator and conditioning routes agree here (Example~\ref{ex:bridge_cond}). Remark~\ref{rem:slepian_conditioning} gives a different agreement, between recentering and conditioning Slepian Brownian motion.
\item[Kernel] $G(s,t)=(s\wedge t)(1-s\vee t)=s\wedge t-st$.
\item[Variance] $\Var(X_s)=s(1-s)$.
\item[Endpoints] $\Var(X_0)=\Var(X_1)=0$. Endpoint correlation is undefined because both ends are deterministic.
\item[From $W$] $X_s=W_s-sW_1$.
\item[Marks] Not stationary; triangular, hence Gauss--Markov.
\item[Spectrum] $\omega=n\pi$, $n\ge1$; trace $\tfrac16$.
\item[Elsewhere] The Cram\'er--von Mises limit law is the distribution of $\int_0^1X_s^2\,\dd s$~\cite{AndersonDarling1952}. The isospectral Neumann--Neumann process gives the same law (Remark~\ref{rem:isospectral}).
\item[See] Example~\ref{ex:bridge}; Figure~\ref{fig:detail_bridge}.
\end{beast}

\begin{beast}{Time-reversed Brownian motion}{$s\mapsto W_{1-s}$; the Robin corner $(\infty,0)$}
\item[Condition] $f'(0)=0$, $f(1)=0$: Neumann--Dirichlet, the Robin corner $(\infty,0)$.
\item[Kernel] $G(s,t)=1-s\vee t$.
\item[Variance] $\Var(X_s)=1-s$.
\item[Endpoints] $\Var(X_0)=1$, $\Var(X_1)=0$. Endpoint correlation is undefined because $X_1=0$ almost surely.
\item[From $W$] $X_s=W_s-W_1$, the member $a=1$ of the family $W-aW_1$.
\item[Marks] Not stationary; triangular, hence Gauss--Markov.
\item[Spectrum] $\omega=(n-\tfrac12)\pi$; trace $\tfrac12$. Time reversal preserves the spectrum of Brownian motion.
\item[See] Example~\ref{ex:time_reversed}; Figure~\ref{fig:detail_time_reversed}.
\end{beast}

\begin{beast}{Robin family}{$(\lambda_0,\lambda_1)$; elastic or partially reflecting ends}
\item[Condition] $f(0)=\lambda_0f'(0)$, $f(1)=-\lambda_1f'(1)$ with $\lambda_i\in[0,\infty]$, excluding the degenerate pair $(\infty,\infty)$. The family also includes the finite-parameter branch $1+\lambda_0+\lambda_1<0$ with at least one parameter nonnegative (Section~\ref{sec:negative_branch}).
\item[Kernel] $G(s,t)=\dfrac{(\lambda_0+s\wedge t)(\lambda_1+1-s\vee t)}{1+\lambda_0+\lambda_1}$, from~\eqref{eq:robin_green}.
\item[Variance] $\Var(X_s)=\dfrac{(\lambda_0+s)(\lambda_1+1-s)}{1+\lambda_0+\lambda_1}$.
\item[Endpoints] $\Var(X_0)=\dfrac{\lambda_0(1+\lambda_1)}{1+\lambda_0+\lambda_1}$ and $\Var(X_1)=\dfrac{(1+\lambda_0)\lambda_1}{1+\lambda_0+\lambda_1}$, so for $\lambda_0,\lambda_1>0$
 \[
  \Corr(X_0,X_1)=\sqrt{\tau_0\tau_1}, \qquad \tau_i=\frac{\lambda_i}{1+\lambda_i},
 \]
 where the $\tau_i$ are the compactified coordinates of Figure~\ref{fig:parameter_map}(a). Remark~\ref{rem:tau_correlation} derives the formula and explains this choice of axes. The correlation tends to zero as either end approaches Dirichlet, but is undefined at the Dirichlet limit itself; it tends to~$1$ as both ends approach the degenerate Neumann--Neumann corner.
\item[From $W$] $Y_s=\sqrt{\lambda_0}Z+W_s$, conditioned on $Y_1+\sqrt{\lambda_1}Z'=0$, for finite $\lambda_0,\lambda_1\ge0$. Here $Z$ and $Z'$ are independent standard normal variables, also independent of~$W$.
\item[Marks] Not stationary; triangular for every parameter pair, so the whole local family is Gauss--Markov (Remark~\ref{rem:markov_column}).
\item[Spectrum] Positive roots of $\omega+\arctan\lambda_0\omega+\arctan\lambda_1\omega=n\pi$, with $n\ge1$ in the quadrant and $n\ge0$ on the negative branch. The extra positive root at $n=0$ is the lowest frequency and gives the largest covariance eigenvalue (Remark~\ref{rem:spectra}); trace parameter-dependent.
\item[Elsewhere] The three preceding entries occupy three corners of $[0,\infty]^2$. The fourth corner has no ordinary Green's function and is treated separately as a degenerate case (Section~\ref{sec:NN_obstruction}).
\item[See] Theorem~\ref{thm:robin}; Figure~\ref{fig:robin_family}; Figure~\ref{fig:parameter_map}(a).
\end{beast}

\begin{beast}{Brownian motion with a random slope}{a Robin member outside the quadrant, $\lambda_1=-1-\sigma^{-2}$}
\item[Condition] $f(0)=0$, $f(1)=(1+\sigma^{-2})f'(1)$: a Robin pair on the negative branch $1+\lambda_0+\lambda_1<0$, outside the nonnegative parameter quadrant (Section~\ref{sec:negative_branch}).
\item[Kernel] $G(s,t)=s\wedge t+\sigma^2st$.
\item[Variance] $\Var(X_s)=s+\sigma^2s^2$.
\item[Endpoints] $\Var(X_0)=0$, $\Var(X_1)=1+\sigma^2$. Endpoint correlation is undefined because $X_0=0$ almost surely.
\item[From $W$] $X_s=W_s+\sigma\,s\,Z$ with $Z$ an independent standard normal. The random linear trend gives the process its name and contributes the quadratic term to its variance.
\item[Marks] Not stationary; triangular, hence Gauss--Markov, like the rest of the local family.
\item[Spectrum] $\tan\omega=(1+\sigma^{-2})\,\omega$, excluding $\omega=0$; trace $\tfrac12+\tfrac13\sigma^2$.
\item[Elsewhere] This entry illustrates the Robin indexing outside the nonnegative quadrant. At $(\lambda_0,\lambda_1)=(0,-2)$, corresponding to $\sigma=1$, the extra positive root at $n=0$ gives the lowest frequency and the dominant covariance mode, carrying $0.883$ of the trace (Remark~\ref{rem:spectra}).
\item[See] Example~\ref{ex:random_slope}; Figure~\ref{fig:detail_random_slope}; Figure~\ref{fig:parameter_map}(b).
\end{beast}

\begin{beast}{$\alpha$-family}{nonlocal, the two ends tied}
\item[Condition] $f(0)=\alpha f(1)$, $f'(0)=f'(1)/\alpha$, with $\alpha\notin\{0,1\}$.
\item[Kernel] $G_\alpha(s,t)=\dfrac{\alpha^2}{(\alpha-1)^2}-\dfrac{(\alpha+1)(s+t)}{2(\alpha-1)}-\dfrac{|s-t|}{2}$, from~\eqref{eq:alpha}.
\item[Variance] $\Var(X_s)=\dfrac{\alpha^2}{(\alpha-1)^2}-\dfrac{(\alpha+1)s}{\alpha-1}$.
\item[Endpoints] $\Var(X_0)=\alpha^2/(\alpha-1)^2$ and $\Var(X_1)=1/(\alpha-1)^2$, with $\Corr(X_0,X_1)=\sgn\alpha$. The endpoint relation is deterministic; its correlation depends only on the sign of~$\alpha$.
\item[From $W$] $X_s=W_s-aW_1$ with $a=\alpha/(\alpha-1)$.
\item[Marks] Stationary only at $\alpha=-1$; never triangular.
\item[Spectrum] \eqref{eq:alpha_spectrum}; trace parameter-dependent.
\item[Elsewhere] The limits $\alpha\to0$ and $\alpha\to\infty$ recover Brownian motion and its time reversal at two corners of the Robin square (Remark~\ref{rem:alpha_limits}). The value $\alpha=-1$ gives the anti-periodic process below. At the excluded value $\alpha=1$, the operator is periodic and degenerate, explaining the covariance blow-up on approach (Remark~\ref{rem:spectra}).
\item[See] Theorem~\ref{thm:alpha}; Figure~\ref{fig:nonlocal_family}.
\end{beast}

\begin{beast}{Anti-periodic process}{the $\alpha=-1$ member}
\item[Condition] $f(0)=-f(1)$, $f'(0)=-f'(1)$.
\item[Kernel] \eqref{eq:anti}: $G(s,t)=\tfrac14\bigl(1-2|s-t|\bigr)$.
\item[Variance] $\Var(X_s)=\tfrac14$, constant.
\item[Endpoints] $\Var(X_0)=\Var(X_1)=\tfrac14$ and $\Corr(X_0,X_1)=-1$: $X_1=-X_0$ almost surely.
\item[From $W$] $X_s=W_s-\tfrac12W_1$.
\item[Marks] Stationary; not triangular.
\item[Spectrum] $\omega=(2k+1)\pi$, each eigenvalue twice; trace $\tfrac14$.
\item[Elsewhere] The $v=0$ member of~\eqref{eq:slepian_family}, obtained by dropping the independent normal from the Slepian construction (Example~\ref{ex:slepian}).
\item[See] Example~\ref{thm:anti}; Figure~\ref{fig:detail_anti}.
\end{beast}

\begin{beast}{Slepian Brownian motion}{the Slepian process}
\item[Condition] $f'(0)+f'(1)=0$, $f(0)+f(1)=f'(0)$: nonlocal conditions outside the $\alpha$-family.
\item[Kernel] $G(s,t)=\tfrac12\bigl(1-|s-t|\bigr)$.
\item[Variance] $\Var(X_s)=\tfrac12$, constant.
\item[Endpoints] $\Var(X_0)=\Var(X_1)=\tfrac12$ and $\Corr(X_0,X_1)=0$: the two ends are independent even though the process is stationary.
\item[From $W$] We give three constructions, each showing a different property. In $X_s=W_s-\tfrac12W_1+\tfrac12Z$, with $Z$ an independent standard normal, the added constant gives a rank-one enlargement of the anti-periodic covariance. The representation $\sqrt2\,X_s=W_{s+1}-W_s$ uses Brownian increments over a unit-width moving window. Alternatively, for independent standard Brownian motions $W,\widetilde W$,
\[
 X_s = \tfrac{1}{\sqrt2}\bigl(W_s + \widetilde{W}_{1-s}\bigr) ,
\]
a normalized sum of forward and independent backward Brownian motions. Bilinearity gives $\tfrac12\{s\wedge t + (1-s)\wedge(1-t)\} = \tfrac12(1-|s-t|)$, since $s\wedge t-s\vee t=-|s-t|$. The dependence on lag alone shows stationarity. The endpoint values depend on independent Brownian paths and are therefore independent. Example~\ref{ex:slepian} rescales the moving-window construction to $[0,1]$.
\item[Marks] Stationary; not triangular.
\item[Spectrum] $\tfrac12\omega\tan\tfrac{\omega}{2}=1$ or $\omega=(2k+1)\pi$, all simple; trace $\tfrac12$.
\item[Elsewhere] Studied for its boundary-crossing and first-passage probabilities~\cite{Slepian1961,Shepp1971}.
\item[See] Example~\ref{ex:slepian}; Figure~\ref{fig:detail_slepian}.
\end{beast}

\begin{beast}{Periodic process}{circular Brownian motion; the $\alpha=1$ degenerate case}
\item[Condition] $f(0)=f(1)$, $f'(0)=f'(1)$. The null space is the constants, so the corrected source~\eqref{eq:corrected} and the balanced condition~\eqref{eq:zeromean} are both required.
\item[Kernel] \eqref{eq:periodic}: $G(s,t)=\tfrac12|s-t|^2-\tfrac12|s-t|+\tfrac1{12}$.
\item[Variance] $\Var(X_s)=\tfrac1{12}$, constant.
\item[Endpoints] $\Var(X_0)=\Var(X_1)=\tfrac1{12}$ and $\Corr(X_0,X_1)=+1$: $X_1=X_0$ almost surely, and the path closes up on the circle.
\item[From $W$] $X_s=B_s-\int_0^1B_u\,\dd u$, the mean-centered Brownian bridge.
\item[Marks] Stationary; not triangular.
\item[Spectrum] $\omega=2\pi n$, each twice, on $\mathbf{1}^\perp$; trace $\tfrac1{12}$.
\item[Elsewhere] Watson's $U^2$ has limiting null law given by $\int_0^1X_s^2\,\dd s$~\cite{Watson1961}. Its finite-sample rotation invariance is mirrored by this process's stationarity (Remark~\ref{rem:cvm_watson}).
\item[See] Theorem~\ref{thm:periodic}; Figure~\ref{fig:detail_periodic}.
\end{beast}

\begin{beast}{Neumann--Neumann process}{both ends free, balanced}
\item[Condition] $f'(0)=f'(1)=0$, with the balanced normalization. This is the Robin corner $(\infty,\infty)$, whose constant null space prevents an ordinary Green's function (Section~\ref{sec:NN_obstruction}). The source correction and balanced condition together select the kernel below.
\item[Kernel] \eqref{eq:NN}: $G(s,t)=\tfrac12(s^2+t^2)-s\vee t+\tfrac13$.
\item[Variance] $\Var(X_s)=s^2-s+\tfrac13$, least at $s=\tfrac12$ and greatest at the two free ends.
\item[Endpoints] $\Var(X_0)=\Var(X_1)=\tfrac13$ and $\Corr(X_0,X_1)=-\tfrac12$.
\item[From $W$] $X_s=W_s-\int_0^1W_u\,\dd u$.
\item[Marks] Not stationary; not triangular.
\item[Spectrum] $\omega=n\pi$ on $\mathbf{1}^\perp$; trace $\tfrac16$.
\item[Elsewhere] Isospectral with the Brownian bridge, so its integrated square has the same Cram\'er--von Mises law (Remark~\ref{rem:isospectral}). The additive constant $\tfrac13$ is fixed by balance and lies on the boundary of the positive-semidefinite cone (Section~\ref{sec:constant_not_free}).
\item[See] Theorem~\ref{thm:NN}; Figure~\ref{fig:detail_nn}; Figure~\ref{fig:parameter_map}(a), which marks the degenerate Robin corner.
\end{beast}

\begin{beast}{Integral-conditioned Brownian motion}{zero time average through conditioning}
\item[Condition] A constraint, not a boundary condition: $W$ given $\int_0^1W_u\,\dd u=0$.
\item[Kernel] $G(s,t)=s\wedge t-3\bigl(s-\tfrac12s^2\bigr)\bigl(t-\tfrac12t^2\bigr)$, the rank-one update~\eqref{eq:conditioning} with $\ell=\int_0^1(\cdot)\,\dd u$.
\item[Variance] $\Var(X_s)=s-3\bigl(s-\tfrac12s^2\bigr)^2$.
\item[Endpoints] $\Var(X_0)=0$, $\Var(X_1)=\tfrac14$. Endpoint correlation is undefined: the constraint leaves $X_0=W_0=0$ almost surely.
\item[From $W$] $X_s=W_s-3\bigl(s-\tfrac12s^2\bigr)\int_0^1W_u\,\dd u$.
\item[Marks] Not stationary; not triangular. The integral constraint couples the whole path.
\item[Spectrum] $\tan\omega=\omega$, with simple positive roots $\omega_n\in(n\pi,n\pi+\tfrac12\pi)$ for $n\ge1$. The first root $\omega_1\approx4.4934$ gives the largest covariance eigenvalue $\omega_1^{-2}\approx0.0495$. The associated precision on $\mathbf{1}^\perp$ is the compression $f\mapsto-f''-f'(0)$ of~$\mathcal{L}$, with domain $\{f\in H^2: f(0)=0,\ f'(1)=0,\ \int_0^1f=0\}$. This is a constrained operator, not the original differential expression with three ordinary boundary conditions (Section~\ref{sec:counting}). The constant function is a null direction of the covariance operator, since the process has zero time average. Trace $\tfrac1{10}$.
\item[Elsewhere] The integral constraint in Section~\ref{sec:conditioning}; detectable because $K_\ell\not\equiv0$.
\item[See] Example~\ref{ex:integral}; Figure~\ref{fig:detail_integral}.
\end{beast}

\begin{beast}{Pinned Brownian motion}{pinned at $\tau\in(0,1)$}
\item[Condition] A constraint: $W$ given $W_\tau=0$.
\item[Kernel] $G(s,t)=s\wedge t-(s\wedge\tau)(t\wedge\tau)/\tau$.
\item[Variance] $\Var(X_s)=s-(s\wedge\tau)^2/\tau$, vanishing at $s=\tau$ as the constraint demands.
\item[Endpoints] $\Var(X_0)=0$, $\Var(X_1)=1-\tau$. Endpoint correlation is undefined because $X_0=0$ almost surely.
\item[From $W$] $X_s=W_s-(s\wedge\tau)W_\tau/\tau$.
\item[Marks] Not stationary; triangular only blockwise. A Brownian bridge on $[0,\tau]$ is independent of the Brownian motion starting at zero on $[\tau,1]$. The process is Markov, although its covariance does not factor globally (Remark~\ref{rem:markov_column}).
\item[Spectrum] Not tabulated: the direct sum of the Dirichlet--Dirichlet spectrum on $[0,\tau]$ and the Dirichlet--Neumann spectrum on $[\tau,1]$ (Section~\ref{sec:domain_splitting}). Trace $\tfrac12-\tau+\tfrac23\tau^2$.
\item[Elsewhere] The split Dirichlet--Dirichlet / Dirichlet--Neumann realization in Section~\ref{sec:domain_splitting} gives exactly this covariance: the operator and conditioning routes agree for the Brownian starting process. Section~\ref{sec:interior_pin} compares different starting laws, the balanced degenerate processes, using~\eqref{eq:pinned} and~\eqref{eq:pinned_cond}; see also the last row of Table~\ref{tab:roadmap}. For the Neumann--Neumann starting process, recentering gives time-reversed Brownian motion on $[0,\tau]$, not a bridge. The recentered periodic process does not split into independent blocks.
\item[See] Example~\ref{ex:interior}; Figure~\ref{fig:detail_pinned}.
\end{beast}

\subsection{Comparison table}\label{sec:tables}

Table~\ref{tab:constructions} compares the field marks of the twelve beasts. Kernels, conditions and endpoint correlations appear in the derivations and catalogue entries; Section~\ref{sec:conditioning} collects the rank-one constraints separately. The following remarks draw out comparisons across the table, including spectral limits, triangularity and simulation.

\begin{landscape}
\vspace*{\fill}
\begin{table}[H]
\centering
\caption{Field marks. \emph{Construction} uses a standard Brownian motion~$W$, a Brownian bridge~$B$, or standard normals $Z,Z'$ independent of the path. The Gaussian constructions guarantee positive semidefiniteness. \emph{Stationary} means that $G(s,t)$ depends on $s-t$ alone. Among the individual processes listed, the anti-periodic, Slepian and periodic processes qualify; stationarity does not uniquely identify them (Example~\ref{ex:slepian}). \emph{Triangular} means a global factorization $u(s\wedge t)\,v(s\vee t)$. For continuous covariance and everywhere-positive variance, this is Doob's characterization of the Gauss--Markov property~\cite{Doob1953,MehrMcFadden1965}. Continuity matters: a process equal to one standard normal on $[0,\tfrac12)$ and an independent one on $[\tfrac12,1]$ is Markov with unit variance but has no such factorization. Remark~\ref{rem:markov_column} explains the pinned process's blockwise exception. \emph{Spectrum} gives precision eigenvalues $\omega^2$; covariance eigenvalues are their reciprocals, with $\sum_{n=1}^{\infty}\omega_n^{-2}=\int_0^1 G(s,s)\,\dd s=\mathrm{trace}$. A spectral dash means not tabulated. The integral-conditioned spectrum belongs to a constrained compression and is given in its entry. The pinned spectrum is the direct sum of the Dirichlet--Dirichlet and Dirichlet--Neumann spectra split at~$\tau$ (Section~\ref{sec:domain_splitting}). A trace dash marks a parameter-dependent value. The random-slope row represents a Robin member outside $[0,\infty]^2$ (Example~\ref{ex:random_slope}); its extra lowest-frequency mode gives the largest covariance eigenvalue and requires the indexing described in Remark~\ref{rem:spectra}.}\label{tab:constructions}
\vspace{10pt}
\footnotesize
\setlength{\tabcolsep}{4pt}
\begin{tabular}{llccll}
\toprule
Process & Construction & Stationary? & Triangular? & Precision spectrum: $\omega^2$ with & Trace \\
\midrule
Brownian motion (D--N) & $W_s$ & no & yes & $\omega=(n-\tfrac12)\pi$ & $\tfrac12$ \\
Brownian bridge (D--D) & $W_s - sW_1$ & no & yes & $\omega=n\pi$ & $\tfrac16$ \\
Time-reversed BM (N--D) & $W_s - W_1$ \quad ($a=1$) & no & yes & $\omega=(n-\tfrac12)\pi$ & $\tfrac12$ \\
Robin$(\lambda_0,\lambda_1\ge0)$ & $Y_s=\sqrt{\lambda_0}Z+W_s$, given $Y_1+\sqrt{\lambda_1}Z'=0$ & no & yes & $\omega+\arctan\lambda_0\omega+\arctan\lambda_1\omega=n\pi$ & --- \\
Random slope ($\lambda_1=-1-\sigma^{-2}$) & $W_s + \sigma\,s\,Z$ & no & yes & $\tan\omega=(1+\sigma^{-2})\,\omega$ & $\tfrac12+\tfrac13\sigma^2$ \\[4pt]
$\alpha$-family & $W_s - a\,W_1$, \quad $a=\alpha/(\alpha-1)$ & only at $\alpha=-1$ & no & $\cos\omega=2\alpha/(1+\alpha^2)$ & --- \\
Anti-periodic ($\alpha=-1$) & $W_s - \tfrac12 W_1$ & yes & no & $\omega=(2k+1)\pi$, each twice & $\tfrac14$ \\
Slepian & $W_s - \tfrac12 W_1 + \tfrac12 Z$ & yes & no & $\tfrac12\omega\tan\tfrac{\omega}{2}=1$ or $\omega=(2k+1)\pi$; simple & $\tfrac12$ \\[4pt]
Neumann--Neumann (balanced) & $W_s - \int_0^1 W_u\,\dd u$ & no & no & $\omega=n\pi$, on $\mathbf{1}^\perp$ & $\tfrac16$ \\
Periodic (balanced) & $B_s - \int_0^1 B_u\,\dd u$ & yes & no & $\omega=2\pi n$, each twice, on $\mathbf{1}^\perp$ & $\tfrac1{12}$ \\[4pt]
Integral-conditioned BM & $W_s - 3\bigl(s-\tfrac12 s^2\bigr)\int_0^1W_u\,\dd u$ & no & no & --- & $\tfrac1{10}$ \\
Pinned BM (given $X_\tau=0$) & $W_s - (s\wedge\tau)\,W_\tau/\tau$ & no & blockwise & --- & $\tfrac12-\tau+\tfrac23\tau^2$ \\
\bottomrule
\end{tabular}
\end{table}
\vspace*{\fill}
\end{landscape}

\begin{remark}[Constructions give positivity and continuous paths]\label{rem:continuity}
The Construction column establishes both positive semidefiniteness and continuous sample paths. Its operations use continuous Brownian paths or bridges, endpoint evaluations, path integrals, and normal variables multiplied by continuous coefficient functions. The Robin conditioning step also has a continuous Gaussian regression-residual representation. These particular operations preserve continuity, so every beast has a continuous version without a separate appeal to a Kolmogorov criterion. By contrast, the pointwise $L^2(\mathbb{P})$ convergence in~\eqref{eq:KL} defines the process at each fixed~$s$ but does not by itself establish path continuity.
\end{remark}

\begin{remark}[Spectra, limits and mode indexing]\label{rem:spectra}
The Robin and $\alpha$ spectral equations recover the simpler spectra as limits. For the Robin family, substitute $f(s)=\sin(\omega s+\phi)$ into $f(0)=\lambda_0f'(0)$ and $f(1)=-\lambda_1f'(1)$. This gives $\tan\phi=\lambda_0\omega$ and $\tan(\omega+\phi)=-\lambda_1\omega$, hence the equation in the table. Setting $\lambda_0=\lambda_1=0$ yields $\omega=n\pi$; setting $\lambda_0=0$ and letting $\lambda_1\to\infty$ yields $\omega=(n-\tfrac12)\pi$. These are the Dirichlet--Dirichlet and Dirichlet--Neumann spectra. For a general fixed $\lambda_0$, the latter limit instead gives $\omega+\arctan\lambda_0\omega=(n-\tfrac12)\pi$; at $\lambda_0=1$ its first root is $0.860$, not $\pi/2$.

The indexing changes outside the nonnegative quadrant. The left-hand side has derivative $1+\lambda_0+\lambda_1$ at $\omega=0$ and is increasing for nonnegative parameters, so $n\ge1$ exhausts the positive spectrum there. On the admissible negative branch $1+\lambda_0+\lambda_1<0$, it initially decreases before turning upwards, giving an extra positive root at $n=0$. This is the lowest frequency and the dominant covariance mode. At $(\lambda_0,\lambda_1)=(0,-2)$, the random-slope process of Example~\ref{ex:random_slope} with $\sigma=1$, the covariance eigenvalue is $\omega^{-2}\approx0.736$, a fraction $0.883$ of the trace $\tfrac56$. Keeping $n\ge1$ on this branch would therefore discard the largest covariance eigenvalue, not merely relabel the modes.

For the $\alpha$-family, write $f=A\cos\omega s+B\sin\omega s$ and impose $f(0)=\alpha f(1)$, $f'(0)=f'(1)/\alpha$. The resulting $2\times2$ system is $A(1-\alpha\cos\omega)=\alpha B\sin\omega$, $A\sin\omega=B(\cos\omega-\alpha)$, with determinant $(1+\alpha^2)\cos\omega-2\alpha$. Thus the spectral condition is
\begin{equation}\label{eq:alpha_spectrum}
 \cos\omega = \frac{2\alpha}{1+\alpha^2} .
\end{equation}
Since $|2\alpha/(1+\alpha^2)|\le 1$, with equality only at $\alpha=\pm1$, the eigenvalues are simple for all other members of the family. At $\alpha=-1$, $\cos\omega=-1$ gives the frequencies $\omega=(2k+1)\pi$ of the anti-periodic process, each with multiplicity two. At the excluded value $\alpha=1$, $\cos\omega=1$ gives the positive frequencies $\omega=2\pi n$ of the periodic process, again with multiplicity two; its constant zero mode is removed by balance. The covariance blow-up as $\alpha\to1$ reflects this approaching null mode. Finally, $\alpha\to0$ or $\alpha\to\pm\infty$ gives $\cos\omega=0$, recovering Brownian motion and its time reversal.
\end{remark}

\begin{remark}[Reading the triangularity column]\label{rem:markov_column}
Triangularity is closely related to, but not identical with, the distinction between local and nonlocal conditions. Robin conditions act separately at each endpoint, and~\eqref{eq:robin_green} factors as $u(s\wedge t)v(s\vee t)$ with $u(s)=\lambda_0+s$ and $v(t)=(\lambda_1+1-t)/(1+\lambda_0+\lambda_1)$. The local family is therefore Gauss--Markov. By contrast, the $\alpha$-family ties $X_0$ to~$X_1$, while balance imposes a condition on the whole path.

An algebraic test distinguishes these cases. A triangular kernel $K(s,t)=u(s\wedge t)v(s\vee t)$ satisfies $K(s,t)K(t,u)=K(s,u)K(t,t)$ for every $s<t<u$, since both sides equal $u(s)u(t)v(t)v(u)$. The presence of $|s-t|$ is no obstruction: Brownian motion has $s\wedge t=\tfrac12(s+t-|s-t|)$ and is triangular. For the $\alpha$-family~\eqref{eq:alpha}, the left-hand side minus the right-hand side is $-\alpha(t-s)(u-t)/(\alpha-1)^2$, nonzero for every admissible $\alpha\ne0$. The balanced kernel~\eqref{eq:NN} also fails the identity. For these processes, the future depends on more of the past than the present value alone.

The last row explains why the heading is \emph{triangular}, not \emph{Markov}. Brownian motion conditioned on $X_\tau=0$ is Markov (Example~\ref{ex:interior}), but its covariance does not factor globally. Here $\Cov(X_s,X_t)=0$ whenever $s<\tau<t$, although both blocks are nondegenerate. A global factorization $u(s)v(t)$ vanishing on all straddling pairs would force $u\equiv 0$ on $[0,\tau)$ or $v\equiv0$ on $(\tau,1]$, eliminating one block. Instead, the kernel is triangular on each side of~$\tau$ and block diagonal across it. The pin creates a zero-variance point, outside the positive-variance hypothesis of Doob's product representation. This interior degeneracy is distinct from the boundary-kink issue in Remark~\ref{rem:coincident_kink}.
\end{remark}

\begin{remark}[The bridge and the Neumann--Neumann process are isospectral]\label{rem:isospectral}
The Dirichlet--Dirichlet operator and the Neumann--Neumann operator restricted to $\mathbf{1}^\perp$ have the same eigenvalues $(n\pi)^2$, $n\ge1$, as Table~\ref{tab:constructions} records. Their covariances nevertheless differ: $s\wedge t-st$ vanishes at both ends, whereas $G_\NN$ has $\Var(X_0)=\Var(X_1)=\tfrac13$. Equality of the nonzero covariance spectra implies that
\[
 \int_0^1 B_s^2\,\dd s \quad\text{and}\quad \int_0^1 X_s^2\,\dd s \ \ (X \text{ Neumann--Neumann})
\]
have the same distribution: $\sum_{n=1}^{\infty}(n\pi)^{-2}\chi^2_{1,n}$ with independent $\chi^2_1$ variables. Both traces are $\sum_{n=1}^{\infty}(n\pi)^{-2} = \tfrac16$, also obtained by evaluating $\int_0^1 G(s,s)\,\dd s$ directly from either kernel. The equality in law is due to Donati-Martin and Yor~\cite{DonatiMartinYor1991}; here it follows from the shared spectrum.
\end{remark}

\begin{remark}[Where the catalogue meets goodness-of-fit statistics]\label{rem:cvm_watson}
The Brownian bridge and periodic process underlie classical goodness-of-fit statistics. The Cram\'er--von Mises statistic has limiting null law given by $\int_0^1 B_s^2\,\dd s$ for the Brownian bridge~\cite{AndersonDarling1952}. By Remark~\ref{rem:isospectral}, this is also the law of $\int_0^1X_s^2\,\dd s$ for the Neumann--Neumann process. Watson's $U^2$~\cite{Watson1961} instead has limiting null law given by $\int_0^1(B_s-\bar B)^2\,\dd s$, where $\bar B=\int_0^1B_u\,\dd u$. By Proposition~\ref{prop:centred_constructions}, Watson's limit law is that of $\int_0^1 X_s^2\,\dd s$, where $X$ is the periodic process obtained by mean-centering a Brownian bridge.

Its doubled precision spectrum $(2\pi n)^2$ gives the classical series representation $$U^2 \sim \sum_{n=1}^{\infty}(2\pi n)^{-2}\chi^2_{2,n},$$ where $U^2$ here denotes the limiting variable and the chi-squared variables are independent. Each frequency contributes two degrees of freedom, and the trace is $\tfrac1{12}$. Rotation invariance of the statistic is exact at finite sample size: changing the circular starting point leaves $U^2$ unchanged. In the limit, this invariance is mirrored by the stationarity of the periodic process, whose covariance $\tfrac12|s-t|^2-\tfrac12|s-t|+\tfrac1{12}$ depends on $s-t$ alone.
\end{remark}

\paragraph{How to simulate.}
There are four practical routes, distinguished by what they discretize and whether they sample grid values exactly.

\emph{Transform a Brownian path.} For a construction in Table~\ref{tab:constructions}, generate Brownian grid values by cumulatively summing independent $N(0,h)$ increments on a grid of spacing~$h$, then apply the stated map. Subtracting $aW_1$ takes one pass; the finite nonnegative Robin construction uses two extra normal draws and a rank-one regression update. A trapezoidal approximation to $\int_0^1W$ does not, however, give exact mean-centered samples. At grid nodes it reproduces $K_\ell$ exactly, because the trapezoid rule integrates the piecewise-affine function $s\mapsto s\wedge t$ exactly when its kink is a node. It underestimates $V_\ell$ by $h^2/12$ on a uniform grid, giving $G_\NN-\tfrac{h^2}{12}\mathbf{1}\otimes\mathbf{1}$. This sampled covariance is singular in the trapezoidal-weight direction, explaining the finite-grid threshold $\tfrac13-1/(12n^2)$ rather than $\tfrac13$ in Remark~\ref{rem:extremality}.

For exact sampling, include the residual integral between nodes: $\int_0^1W=\sum_i\tfrac{h_i}{2}(W_{t_i}+W_{t_{i+1}})+\eta$, where $\eta\sim N\bigl(0,\sum_ih_i^3/12\bigr)$ is independent of the grid values. This follows by integrating the independent Brownian bridges between successive nodes. Figure~\ref{fig:centring_vs_conditioning} uses this construction.

\emph{Sample from a tridiagonal precision matrix.} For the admissible Robin family, the finite-dimensional energy in Remark~\ref{rem:fem_exact} has tridiagonal precision~$M$ satisfying $\bigl(M^{-1}\bigr)_{ij}=G(s_i,s_j)$ exactly for every~$n$. Gaussian sampling with this precision costs $O(n)$ and introduces no discretization error in the grid values. At a Dirichlet end, delete the pinned coordinate before factorization. The degenerate pair $\lambda_0=\lambda_1=\infty$ is excluded because constants make~$M$ singular.

\emph{Truncate the spectrum.} In the separated cases, the eigenfunctions are the sines and cosines associated with the spectral frequencies. Sample $X_s\approx\sum_{n\le N}\mu_n^{-1/2}\varphi_n(s)\xi_n$ using independent standard normals~$\xi_n$. The omitted variance is the tail $\sum_{n>N}\mu_n^{-1}\varphi_n(s)^2$, which is $O(1/N)$ uniformly in~$s$ here because the eigenfunctions are uniformly bounded and the precision eigenvalues grow quadratically.

\emph{Factor the sampled kernel.} Evaluate the closed form on the grid and factor the resulting covariance matrix, checking first for exact linear dependence. A null direction of the continuous covariance operator need not make this matrix singular: the balanced Neumann--Neumann kernel is strictly positive definite on distinct nodes (Remark~\ref{rem:extremality}). Its integral constraint imposes no relation among finitely many values. Boundary identifications do. On a grid containing both endpoints, periodicity gives $G(0,t)=G(1,t)$ and the exact null vector $e_0-e_n$; the $\alpha$-family, including its anti-periodic member, satisfies $X_0=\alpha X_1$. Any pinned coordinate, including a Dirichlet endpoint, gives a zero row. These matrices are exactly singular, so ordinary Cholesky factorization is unavailable. Remove deterministic or linearly dependent coordinates, factor the remaining matrix, and reconstruct the omitted values from their relations. Alternatively, use a rank-revealing positive-semidefinite factorization. Adding jitter to force Cholesky through simulates a different process.

Figure~\ref{fig:bestiary_summary} collects sample paths from six of the beasts side by side.

\begin{figure}[tbp]
 \centering
 \includegraphics[width=0.8\textwidth]{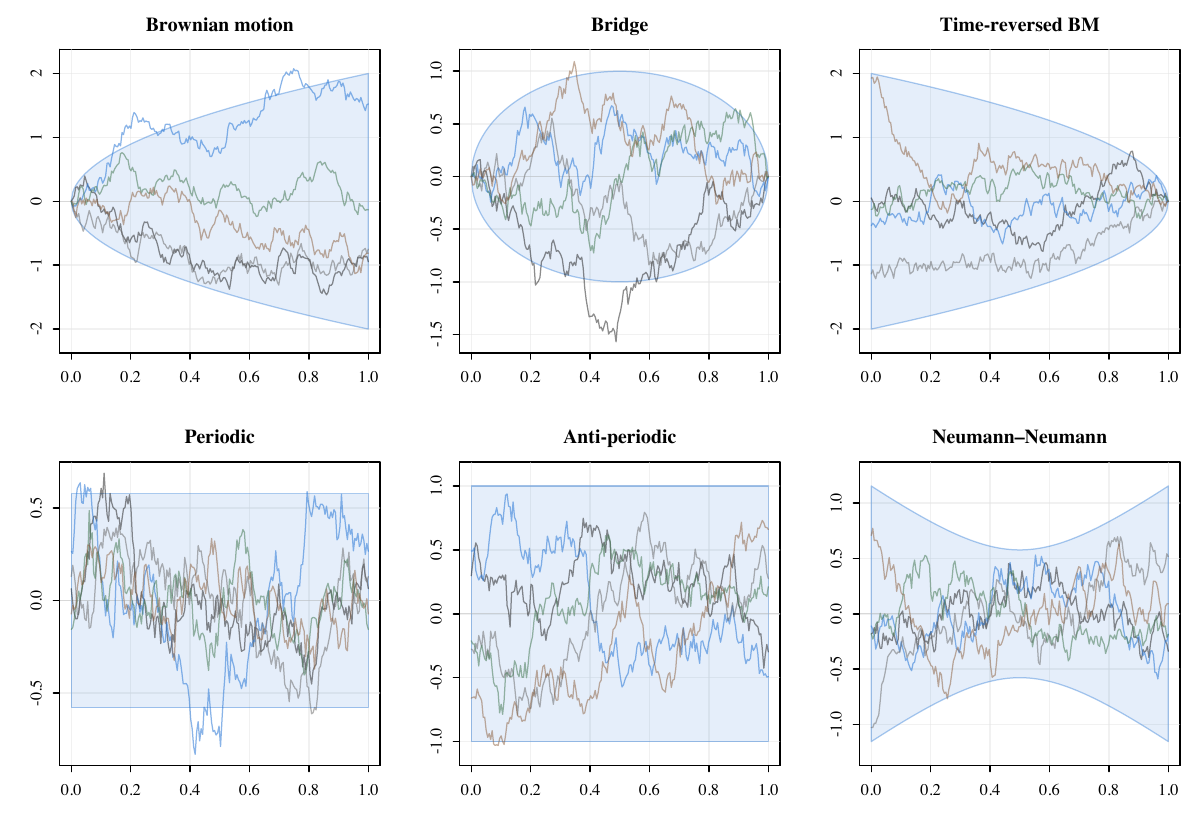}
 \caption{Sample paths from six beasts: Brownian motion, the Brownian bridge, time-reversed Brownian motion, the periodic process, the anti-periodic process and the Neumann--Neumann process. The envelopes collapse to zero at pinned endpoints: both ends for the bridge and one end for Brownian motion and its time reversal.}
 \label{fig:bestiary_summary}
\end{figure}

\section{Discussion}\label{sec:discussion}

The Laplacian on $[0,1]$ generates a rich family of Gaussian processes, organized here by three operations: choosing a realization, recentering a path, and conditioning an existing law (Section~\ref{sec:two_mechanisms}). Keeping them separate clarifies four points. Balance fixes the additive constant in the degenerate covariance rather than leaving it free (Section~\ref{sec:constant_not_free}). An extra condition may overdetermine a boundary-value problem while the corresponding regularized observations have no limiting conditioning effect (Section~\ref{sec:bc_vs_constraints}). The endpoint--integral conditions in Proposition~\ref{prop:bulk} meet the count of two but do not make the operator symmetric, so their Green's functions are not covariances. Finally, point-centering a balanced degenerate process differs from conditioning at that point by the rank-one kernel~\eqref{eq:pin_gap} (Section~\ref{sec:interior_pin}).

The two-parameter Robin family and one-parameter nonlocal $\alpha$-family do not exhaust the self-adjoint realizations of the Laplacian on $[0,1]$. Over complex $L^2[0,1]$, these realizations are parametrized by $\mathrm{U}(2)$, with four real parameters~\cite[Chapter 4]{Zettl2005}. Over the reals, they correspond to Lagrangian planes for the boundary form $[fg'-f'g]_0^1$ on the trace space $\mathbb{R}^4$ of $(f(0),f'(0),f(1),f'(1))$; the Lagrangian Grassmannian $\Lambda(2)$ has dimension $3$.

The Robin family gives a two-dimensional slice, while the $\alpha$-family imposes hard endpoint ties. Neither includes a general finite-stiffness cross-penalty $2c\,f(0)f(1)$ in the energy form. This coupling is distinct from the exact relation $f(0)=\alpha f(1)$. The family~\eqref{eq:slepian_family} already supplies examples outside both families, as the coefficient calculation in Example~\ref{ex:slepian} shows. In the complex setting, the imaginary part of the endpoint coupling supplies the additional parameter; it produces complex-valued Green's functions rather than covariances of real-valued processes.

Positivity imposes a further restriction, as Remark~\ref{rem:admissible} explains. For an ordinary inverse to be a covariance operator, the self-adjoint realization must be strictly positive; the degenerate nonnegative cases are treated after their null directions have been removed. Self-adjointness alone gives symmetry, not positive semidefiniteness: a negative eigenvalue rules out a covariance. Thus covariance-producing realizations form a proper subset of the self-adjoint class, whether described by $\mathrm{U}(2)$ or $\Lambda(2)$. In finite boundary-penalty charts, strict positivity is an open condition, not a further dimension count. Section~\ref{sec:negative_branch} illustrates the restriction in Robin coordinates $(\lambda_0,\lambda_1)$: positivity fails on the band $-(1+\lambda_0)<\lambda_1<0$ although self-adjointness persists. The degenerate Neumann--Neumann corner and the periodic and anti-periodic identifications require separate limiting descriptions; the finite-chart statement alone does not classify their position in the full space of realizations.

A companion paper~\cite{Raket2026classification} develops a three-parameter Green's function family indexed by a symmetric boundary precision matrix $\Sigma^{-1}=\bigl[\begin{smallmatrix}a&c\\c&b\end{smallmatrix}\bigr]$, complete for finite real value-coupled boundary conditions. No result here depends on that classification; the correspondences below follow directly from the kernels already given. Admissibility requires positive definiteness of $\Sigma^{-1}+J$, where $J=\bigl[\begin{smallmatrix}1&-1\\-1&1\end{smallmatrix}\bigr]$, not positive semidefiniteness of $\Sigma^{-1}$ alone. Equivalently, $a>-1$ and $(1+a)(1+b)-(1-c)^2>0$. The Robin family in Section~\ref{sec:robin} occupies the slice $c=0$, with finite boundary stiffnesses $a=1/\lambda_0$ and $b=1/\lambda_1$ as in Section~\ref{sec:negative_branch}. For $v>0$, the family~\eqref{eq:slepian_family} is the disjoint ray $a=b=c=1/(4v)$, with the Slepian process at $a=b=c=1$.

Hard endpoint ties, including the nonlocal $\alpha$-family and its anti-periodic member, arise in the closure of this finite family rather than as coordinate slices. They are not confined to limits along the Slepian ray. Fix $\alpha\ne1$ and, for $M>0$, take $\Sigma^{-1}_M=M\bigl[\begin{smallmatrix}1&-\alpha\\-\alpha&\alpha^2\end{smallmatrix}\bigr]$. Its energy is $\int_0^1f'^2+M\bigl(f(0)-\alpha f(1)\bigr)^2$. Since $\det(\Sigma^{-1}_M+J)=M(\alpha-1)^2>0$ and $(\Sigma^{-1}_M+J)_{00}=1+M>0$, the precision is admissible. Writing $h(s)=(1-s,s)^\top$ gives
\[
 K_M(s,t) = s\wedge t-st+h(s)^\top\bigl(\Sigma^{-1}_M+J\bigr)^{-1}h(t) = G_\alpha(s,t)+\frac{1}{M(\alpha-1)^2} ,
\]
and hence $K_M\to G_\alpha$ as $M\to\infty$. For $\alpha=2$, the path $(a,b,c)=(M,4M,-2M)$ lies off the Slepian ray and converges to the hard tie $X_0=2X_1$, not to Dirichlet pinning. Thus the finite-chart admissibility condition and the Slepian-ray correspondence coexist with a broader class of hard-tie limits; the anti-periodic process is only one example.

The analysis treats homogeneous linear conditions. Linearity ensures the joint Gaussianity of $\bigl(X_s,\ell(X)\bigr)$ needed for the rank-one identity~\eqref{eq:conditioning} in Section~\ref{sec:recentring_conditioning}; conditioning at a nonzero linear value remains Gaussian but generally introduces a nonzero mean. Nonlinear observations need not preserve joint Gaussianity, so the conditioned process may no longer be Gaussian. For example, conditioning on $W_1^2=1$ gives $B_s+s\varepsilon$, where the Brownian bridge~$B$ and equiprobable sign $\varepsilon=\pm1$ are independent. Its covariance is still $s\wedge t$, although $X_1=\varepsilon$ is not Gaussian. A covariance may therefore survive nonlinear conditioning without determining the resulting law. The detectability principle has the same Gaussian scope (Remark~\ref{rem:detect_scope}). Likewise, the count in Section~\ref{sec:counting} concerns homogeneous linear conditions and is not sufficient by itself, even within that class (Proposition~\ref{prop:bulk}).

Two extensions are natural. The first changes the operator: replacing $\mathcal{L}=-\dd^2/\dd s^2$ by the Helmholtz operator $\mathcal{L}_\kappa = -\dd^2/\dd s^2+\kappa^2$, with $\kappa>0$, replaces affine homogeneous solutions by combinations of $\sinh$ and $\cosh$. On the real line, it gives the exponential Mat\'ern-$\tfrac{1}{2}$ covariance $e^{-\kappa|s-t|}/(2\kappa)$. This suggests a hyperbolic branch of the bestiary, organized by the same distinction between realizations, recentering and conditioning. The second extension moves to higher dimensions and tensor constructions, including Brownian sheets and their pinned and periodic relatives, with natural connections to polyharmonic operators and smoothing splines~\cite{Wahba1990,Raket2021}.

\end{document}